\documentclass[10pt, leqno]{amsart} 
\usepackage{amssymb,amscd,amsfonts,amsbsy}
\usepackage{latexsym}
\usepackage{exscale}
\usepackage{amsmath,amsthm,amsfonts}
\usepackage{mathrsfs}
\usepackage{xcolor}
\usepackage[colorlinks=true, linkcolor=blue, citecolor=red, urlcolor=red, 
]{hyperref} 
\usepackage{esint}
\usepackage{stmaryrd}
\usepackage{pifont}
\usepackage{bm}

\usepackage[utf8]{inputenc}

\calclayout

\newcommand{\black}{\color{black}}

\allowdisplaybreaks

\makeatletter
\@namedef{subjclassname@2020}{%
  \textup{2020} Mathematics Subject Classification}
\makeatother

\theoremstyle{plain}
\newtheorem{theorem}[equation]{Theorem}

\newtheorem{lemma}[equation]{Lemma}

\theoremstyle{definition}
\newtheorem{definition}[equation]{Definition}

\numberwithin{equation}{section}

\def\C{\mathbb{C}}
\def\D{\mathcal{D}}
\def\E{\mathbb{E}}
\def\F{\mathscr{F}}

\def\G{\mathscr{G}}
\def\H{\mathrm{H}}
\def\I{\mathbb{I}}

\def\N{\mathbb{N}}

\def\S{\mathscr{S}}

\def\R{\mathbb{R}}
\def\Rn{\R^n}

\def\Z{\mathcal{Z}}
\def\ZZ{\mathbb{Z}}

\def\loc{\operatorname{loc}}

\def\supp{\operatorname{supp}}

\def\BMO{\operatorname{BMO}}

\def\d{\operatorname{d}}
\def\rd{\operatorname{rd}}

\def\ch{\operatorname{ch}}

\def\w{\omega}

\renewcommand{\emptyset}{\text{\textup{\O}}}

\begin{document}

\author[M. Cao]{Mingming Cao}
\address{Mingming Cao\\
Instituto de Ciencias Matem\'aticas CSIC-UAM-UC3M-UCM\\
Con\-se\-jo Superior de Investigaciones Cient{\'\i}ficas\\
C/ Nicol\'as Cabrera, 13-15\\
E-28049 Ma\-drid, Spain} \email{mingming.cao@icmat.es}

\author[J. Chen]{Jiao Chen}
\address{Jiao Chen \\
School of Mathematical Sciences\\ 
Chongqing Normal University\\ 
Chongqing 400000\\ 
People's Republic of China} \email{chenjiaobnu@163.com}

\author[Z. Li]{Zhengyang Li}
\address{School of Mathematics $\&$ Computational Science\\
Hunan University of Science and Technology\\
Xiangtan 411201\\ 
People's Republic of China} \email{zhengyli@hnust.edu.cn}

\author[F. Liao]{Fanghui Liao} 
\address{School of Mathematics $\&$ Computational Science\\
Xiangtan University\\ 
Xiangtan 411105\\ 
People's Republic of China} \email{liaofh@xtu.edu.cn}

\author[K. Yabuta]{K\^{o}z\^{o} Yabuta}
\address{K\^{o}z\^{o} Yabuta\\
Research Center for Mathematics and Data Science\\
Kwansei Gakuin University\\
Gakuen 2-1, Sanda 669-1337\\
Japan} \email{kyabuta3@kwansei.ac.jp}

\author[J. Zhang]{Juan Zhang}
\address{Juan Zhang\\
School of Science\\
Beijing Forestry University\\
Beijing, 100083 \\
People's Republic of China}\email{juanzhang@bjfu.edu.cn}

\thanks{The first author acknowledges financial support from Spanish Ministry of Science and Innovation through the Ram\'{o}n y Cajal  2021 (RYC2021-032600-I), through the ``Severo Ochoa Programme for Centres of Excellence in R\&D'' (CEX2023-001347-S), and through PID2022-141354NB-I00. The third author was supported by National Natural Science Foundation of China (No. 12101222). The fourth author was supported Hunan Education Department Project, China(22B0155). The last author was supported by National Natural Science Foundation of China (No. 12101049). Part of this work was carried out while the first two authors were visiting the Hausdorff Center for Mathematics. The authors would like to express their gratitude to the HCM for their support.
}

\year=2025 \month=04 \day=29

\date{\today}




\begin{abstract}
We develop a compact version of $T1$ theorem for singular integrals of Zygmund type on $\mathbb{R}^3$. More specifically, if a $(D_{\theta}, \delta_1, \delta_{2, 3})$-Calder\'{o}n--Zygmund operator $T$ associated with Zygmund dilations admits the compact full and partial kernel representations, and satisfies the weak compactness property and the cancellation condition, then $T$ can be extended to a compact operator on $L^p(w)$ whenever (i) $p \in (1, \infty)$, $w \in A_{p, \mathcal{R}}$, and $\theta, \delta_1, \delta_{2, 3} \in (0, 1]$, or (ii) $p \in (1, \infty)$, $w \in A_{p, \mathcal{Z}}$, $\theta = \delta_1 = 1$, and $\delta_{2, 3} \in (0, 1]$. Here $A_{p, \mathcal{R}}$ and $A_{p, \mathcal{Z}}$ respectively denote the class of of strong $A_p$ weights and the class of Zygmund $A_p$ weights. Beyond that, under similar bilinear assumptions, we prove bilinear Calder\'{o}n--Zygmund operators associated with Zygmund dilations are compact from $L^{p_1}(\R^3) \times L^{p_2}(\R^3)$ to $L^p(\R^3)$ for all $p_1, p_2 \in (1, \infty)$, where $\frac1p = \frac{1}{p_1} + \frac{1}{p_2}$. The core of the proof is a compact dyadic representation, which asserts that under the hypotheses above, a (bilinear) Calder\'{o}n--Zygmund operator associated with Zygmund dilations can be represented an average of some compact (bilinear) dyadic shifts of Zygmund nature. This further deepens our understanding of the compactness of singular integral operators.
\end{abstract}

\title[A compact $T1$ theorem]
{A compact $T1$ theorem for Calder\'{o}n--Zygmund operators associated with Zygmund dilations}

\maketitle
\tableofcontents

\section{Introduction}

\subsection{Background and main results} 
The exploration of singular integrals originated from the study of singular convolution operators, such as the Hilbert and Riesz transforms. Calder\'{o}n and Zygmund \cite{CZ} extended the concept of the Hilbert transform on $\R$ to the convolution type of singular integral operators on $\R^{n}$. The core of this theory lies in the invariance of regularity and cancellation conditions under the one-parameter dilation group on $\R^{n}$ defined by $\delta(x_1, \ldots, x_n) = (\delta x_1, \ldots, \delta x_n)$, $\delta>0$. Certainly, the classical singular integrals, maximal functions, and multipliers exhibit invariance under one-parameter dilations.

It is well known that the translation or dilation invariance is not essential for the theory. David and Journ\'{e} \cite{DJ} first formulated the famous $T1$ theorem, which gives a characterization of the $L^2$ boundedness for Calder\'{o}n--Zygmund operators, which is a class of operators of non-convolution type without either of the two invariances. Soon after, Journ\'{e} \cite{J85} used the vector-valued Calder\'{o}n--Zygmund theory to  establish a bi-parameter $T1$ theorem for a class of non-convolution singular integral operators on product spaces. This generalized the multi-parameter theory due to Fefferman and Stein \cite{FS}, which studied convolution operators satisfying certain quantitative properties enjoyed by tensor products of operators of Calder\'{o}n--Zygmund type, as for instance the double Hilbert transform $\R^2$. Subsequently, Pott and Vilarroya \cite{PV} gave a new type of the $T1$ theorem on product spaces by means of several new mixed type conditions instead of the vector-valued formulations. Interestingly, it was shown in \cite{G} that the vector-valued formulations of Journ\'{e} \cite{J85} are equivalent to the full and partial kernel assumptions of Pott and Vilarroya \cite{PV} for bi-parameter singular integrals. Such result was extended to arbitrarily many parameter setting in \cite{Ou} using mixed type characterizing conditions. See \cite{CF80, CF85, Fef86, Fef87, Fef88, FL, MPTT, RS} for more developments about the multi-parameter theory.

In terms of the multi-parameter theory, it is widely accepted that the Zygmund dilation 
\begin{align*}
(x_1, x_2, x_3) \mapsto (\delta_1 x_1, \delta_2 x_2, \delta_1 \delta_2 x_3), \qquad \delta_1, \delta_2 > 0, 
\end{align*}
is commonly acknowledged as the next simplest multi-parameter dilation group after the multi-parameter dilations. Stein \cite{Fef86} first established a correlation between the properties of maximal operators associated with Zygmund dilations and the boundary value problems pertaining to Poisson integrals on symmetric spaces. After that, Ricci and Stein \cite{RS} introduced a class of singular integrals with more general dilations and obtained the boundedness for multi-parameter singular integral operators. Fefferman and Pipher \cite{FP} further considered specific singular integral operators associated with Zygmund dilations and obtained their boundedness on weighted Lebesgue spaces. Recently, Han, et al. \cite{HLLT} introduced a particular category of singular integral operators associated with Zygmund dilations by providing appropriate versions of regularity conditions and cancellation conditions for the convolution kernel, which generalizes Nagel--Wainger singular integrals in \cite{NW} and  Fefferman--Pipher multipliers in \cite{FP}. In another work \cite{HLLTW}, they developed the multi-parameter theory on weighted Hardy spaces associated with Zygmund dilations, which, in particular, proved endpoint estimates for Ricci--Stein operators \cite{RS} and Fefferman--Pipher multipliers \cite{FP}. Until very recently, the authors \cite{HLMV} developed the dyadic multiresolution analysis and the related dyadic-probabilistic methods in the Zygmund dilation setting, which is quite useful to achieve weighted boundedness for singular integrals of Zygmund type.

On the other hand, in terms of the compactness of singular integrals, Villarroya \cite{Vil} first established a new $T1$ theorem to characterize the compactness of Calder\'{o}n--Zygmund operators. The study of the compactness of singular integrals has found the practical significance in the fields of partial differential equations and geometric measure theory. For example, Fabes et al. \cite{FJR} solved the Dirichlet and Neumann boundary value problems based on the compactness of layer potentials, and Hofmann et al. \cite[Section 4.6]{HMT} used the compactness of layer potentials and commutators to characterize regular SKT domains.

It has been an open problem in the last ten years that how one can establish the $T1$ theorem to obtain the compactness of multi-parameter singular integrals. This problem was first solved by Cao, Yabuta, and Yang \cite{CYY} for bi-parameter singular integrals. The proof method in \cite{CYY} is quite different from that of \cite{Vil}. In addition, the authors applied the idea of extrapolation of compactness \cite{COY, HL}, which is extremely effective in the bi-parameter theory and has played a significant role in the research of weighted compactness, such as multilinear commutators \cite{CIRXY}, bilinear Calder\'{o}n--Zygmund operators \cite{CLSY}, and multi-parameter singular integrals \cite{CY, CYY}.

However, so far, there is no work on the compactness in the Zygmund dilation setting. Therefore, in this paper, our objective is to formulate a general framework encompassing a diverse set of singular integrals associated with Zygmund dilations and develop a comprehensive theory of compactness for these operators.

To set the stage, let us give the definition of compactness of operators. Given $m \ge 1$ and normed spaces $\mathscr{X}_1, \ldots, \mathscr{X}_m, \mathscr{X}$, an $m$-linear operator $T: \mathscr{X}_1 \times \cdots \times \mathscr{X}_m \to \mathscr{X}$ is called \emph{compact} if for all bounded sets $B_j \subset \mathscr{X}_j$, $j=1, \ldots, m$, the set $T(B_1 \times \cdots \times B_m)$ is precompact in $\mathscr{X}$, i.e., $\overline{T(B_1 \times \cdots \times B_m)}$ is a compact subset of $\mathscr{X}$. More definitions and notation are postponed until later.

\newpage 
Our first main result can be formulated as follows. 

\begin{theorem}\label{thm:cpt} 
Let $\theta, \delta_1, \delta_{2, 3} \in (0, 1]$. Assume that $T$ is a $(D_{\theta}, \delta_1, \delta_{2, 3})$-CZZ operator (cf. Definition {\rm \ref{def:CZO}}) satisfying the hypotheses:
\begin{list}{\rm (\theenumi)}{\usecounter{enumi}\leftmargin=1.2cm \labelwidth=1cm \itemsep=0.2cm \topsep=0.2cm \renewcommand{\theenumi}{H\arabic{enumi}}}

\item\label{H1} $T$ admits the compact full kernel representation (cf. Definition $\ref{def:full}$), 

\item\label{H2} $T$ admits the compact partial kernel representation (cf. Definition $\ref{def:partial}$), 

\item\label{H3} $T$ satisfies the weak compactness property (cf. Definition $\ref{def:WCP}$), 

\item\label{H4} $T$ satisfies the cancellation condition (cf. Definition $\ref{def:cancellation}$). 
\end{list} 
Then the following statements hold: 
\begin{list}{\rm (\theenumi)}{\usecounter{enumi}\leftmargin=1.2cm \labelwidth=1cm \itemsep=0.2cm \topsep=0.2cm \renewcommand{\theenumi}{\alph{enumi}}}

\item\label{list-01} If $\theta, \delta_1, \delta_{2, 3} \in (0, 1]$, then $T$ is compact on $L^p(w)$ for all $p \in (1, \infty)$ and $w \in A_{p, \mathcal{R}}$. 

\item\label{list-02} If $\theta = \delta_1 = 1$ and $\delta_{2, 3} \in (0, 1]$, then $T$ is compact on $L^p(w)$ for all $p \in (1, \infty)$ and $w \in A_{p, \mathcal{Z}}$. 

\item\label{list-03} Both \eqref{list-01} and \eqref{list-02} hold if $T$ is a $(D_{\log}, \delta_1, \delta_{2, 3})$-CZZ operator satisfying the hypotheses \eqref{H1}--\eqref{H4} with $D_{\theta}$ replaced by $D_{\log}$.
\end{list} 
\end{theorem}

As shown above, the parameter $\theta$ is a delicate threshold in the weighted compactness. In the case \eqref{list-02}, there is a full weighted compactness in terms of the class of Zygmund weights $A_{p, \mathcal{Z}}$; while in the case \eqref{list-01}, the weighted compactness only holds in the class of strong weights $A_{p, \mathcal{R}}$, which is strictly smaller than $A_{p, \mathcal{Z}}$. Additionally, the weighted compactness is not true in the class of $A_{p, \mathcal{Z}}$ whenever $\theta \in (0, 1)$ because of the failure of the weighted boundedness in this scenario (cf. \cite[Proposition 2.18]{HLMV}).

Our second main result concerns the compactness of bilinear Calder\'{o}n--Zygmund operators associated with Zygmund dilations. Significantly, we successfully achieve the full range of exponents in the bilinear case.

\begin{theorem}\label{thm:Bcpt} 
Let $\vartheta \in (0, 2]$ and $\theta, \delta_1, \delta_{2, 3} \in (0, 1]$. Assume that $T$ is a $(\mathfrak{D}_{\vartheta}, D_{\theta}, \delta_1, \delta_{2, 3})$-BCZZ operator (cf. Definition {\rm \ref{def:BCZO}}) satisfying the hypotheses:
\begin{list}{\rm (\theenumi)}{\usecounter{enumi}\leftmargin=1.2cm \labelwidth=1cm \itemsep=0.2cm \topsep=0.2cm \renewcommand{\theenumi}{B\arabic{enumi}}}

\item\label{B1} $T$ admits the compact full kernel representation (cf. Definition $\ref{def:Bfull}$), 

\item\label{B2} $T$ admits the compact partial kernel representation (cf. Definition $\ref{def:Bpartial}$), 

\item\label{B3} $T$ satisfies the weak compactness property (cf. Definition $\ref{def:BWCP}$), 

\item\label{B4} $T$ satisfies the cancellation condition (cf. Definition $\ref{def:Bcancellation}$). 
\end{list} 
Then $T$ is compact from $L^{p_1}(\R^3) \times L^{p_2}(\R^3)$ to $L^p(\R^3)$ for all $p_1, p_2 \in (1, \infty)$, where $\frac1p = \frac{1}{p_1} + \frac{1}{p_2}$. 
\end{theorem}

It is well-known that the compactness is strictly stronger than the boundedness. The lack of weighted compactness in Theorem \ref{thm:Bcpt} is due to the absence of genuine multilinear weighted estimates in the entangled Zygmund setting. Once the latter is obtained, one can use the method in this current paper to conclude the weighted compactness of bilinear Calder\'{o}n--Zygmund operators associated with Zygmund dilations.

\subsection{Proof strategy of Theorem \ref{thm:cpt}}
Let us briefly outline how to prove Theorem \ref{thm:cpt}. Our proof relies on a compact dyadic representation in the Zygmund dilation setting as follows.

\begin{theorem}\label{thm:repre}
Let $T$ be a $(D_{\theta}, \delta_1, \delta_{2, 3})$-CZZ operator (cf. Definition {\rm \ref{def:CZO}}), where $\theta, \delta_1, \delta_{2, 3} \in (0, 1]$. Assume that $T$ satisfies the hypotheses \eqref{H1}--\eqref{H4}. Then $T$ admits a compact dyadic representation (cf. Definition {\rm \ref{def:repre}}). 
\end{theorem}

Compared with \cite{CYY}, Theorem \ref{thm:repre} is a new ingredient in the study of (weighted) compactness of multi-parameter singular integrals. It captures the precise dyadic structure behind compact Calder\'{o}n--Zygmund operators associated with Zygmund dilations. Such a dyadic representation was first established in \cite{Hyt} in order to show the sharp weighted norm inequality for Calder\'{o}n--Zygmund operators. A bi-parameter analogue was obtained in \cite{Mar}, which can be applied to two-weight commutators estimates \cite{HPW}. It was recently extended to the Zygmund dilation setting in \cite{ALM, HLMV}.

Let us explain why both approaches in \cite{CYY, Vil} are invalid in the Zygmund dilation setting. First, wavelet analysis was successfully applied to study compactness in \cite{Vil}, but there is no a wavelet theory in Zygmund setting. Additionally, it is impossible to parametrize the sums according to eccentricities and relative distances of rectangles because the third sidelength of a Zygmund rectangle depends on the first two sidelength. Second, although bi-parameter dyadic analysis in \cite{CYY} is much more natural and simpler than that in \cite{Vil}, the current Zygmund setting is actually something strictly in between bi-parameter and tri-parameter. If we perform the decomposition as in \cite[Section 3]{CYY}, then there holds
\begin{align*}
\langle Tf, g \rangle 
= \sum_{I, J \in \D_{\mathcal{Z}}} \langle T(\Delta_{I, \mathcal{Z}} f), \Delta_{J, \mathcal{Z}} g \rangle,
\end{align*}
where $\D_{\mathcal{Z}} := \{I = I_1 \times I_2 \times I_3 \in \D^1 \times \D^2 \times \D^3: \ell(I_3) = \ell(I_1) \ell(I_2)\}$ and $\D^1, \D^2, \D^3$ are dyadic grids on $\R$. Given $I, J \in \D_{\Z}$, the size relationship between them must satisfy one of the following: 
\begin{enumerate}
\item $\ell(I_1) \le \ell(J_1)$, $\ell(I_2) \le \ell(J_2)$ (hence, $\ell(I_3) \le \ell(J_3)$); 

\item $\ell(I_1) \le \ell(J_1)$, $\ell(I_2) > \ell(J_2)$, and $\ell(I_3) \le \ell(J_3)$;

\item $\ell(I_1) \le \ell(J_1)$, $\ell(I_2) > \ell(J_2)$, and $\ell(I_3) > \ell(J_3)$;

\item $\ell(I_1) > \ell(J_1)$, $\ell(I_2) \le \ell(J_2)$, and $\ell(I_3) \le \ell(J_3)$; 

\item $\ell(I_1) > \ell(J_1)$, $\ell(I_2) \le \ell(J_2)$, and $\ell(I_3) > \ell(J_3)$;

\item $\ell(I_1) > \ell(J_1)$, $\ell(I_2) > \ell(J_2)$ (hence, $\ell(I_3) > \ell(J_3)$).
\end{enumerate}
It is easy to see that the cases (1), (3), (4), and (6) have structure in the parameters $\{1\}$ and $\{2, 3\}$, while the cases (2) and (5) have structure in the parameters $\{2\}$ and $\{1, 3\}$. It indicates that the kernel estimates must hold with respect to both the parameter grouping $\{1\}$, $\{2, 3\}$ and the parameter grouping $\{2\}$, $\{1, 3\}$, if one would like to obtain similar estimates as in \cite{CYY}. This will greatly narrow the scope of operators to be considered. Therefore, we have to seek a new method that is different from those in \cite{CYY, Vil}. In order to overcome the difficulties aforementioned, we will borrow the idea of \cite{ALM} to have Zygmund resolution of operators with $\ell(I_i) = \ell(J_i)$, $i=1, 2, 3$.

With Theorem \ref{thm:repre} in hand, Theorem \ref{thm:cpt} is reduced to showing the following.

\begin{theorem}\label{thm:SD-cpt}
Let $k = (k_1, k_2, k_3) \in \N^3$. Suppose that the family $\{\mathbf{S}_{\D_{\w}}^k\}_{\w \in \Omega}$ is of the same type (cf. Definitions {\rm \ref{def:shift}}). Then $\E_{\w} \mathbf{S}_{\D_{\w}}^k$ is compact on $L^p(w)$ for all $p \in (1, \infty)$ and $w \in A_{p, \mathcal{Z}}$. 
\end{theorem}

To avoid troubles caused by Zygmund weights, in the proof of Theorem \ref{thm:SD-cpt}, our idea is to just obtain unweighted compactness instead of weighted compactness. For this purpose, we establish extrapolation of compactness inspired by \cite{CYY, HL}. The result below states that the weighted compactness of a linear operator $T$ in the full range follows from the unweighted compactness on one space, once the weighted boundedness of $T$ is shown on a certain Lebesgue space.

\begin{theorem}\label{thm:RdF-cpt}
Let $\mathcal{B} \in \{\mathcal{R}, \mathcal{Z}\}$. Assume that $T$ is a linear operator such that  
\begin{list}{\rm (\theenumi)}{\usecounter{enumi}\leftmargin=1.2cm \labelwidth=1cm \itemsep=0.2cm \topsep=0.2cm \renewcommand{\theenumi}{\alph{enumi}}}

\item\label{RdFcpt-1} $T$ is compact on $L^{p_0}(w_0)$ for some $p_0 \in [1, \infty)$ and for some $w_0 \in A_{p_0, \mathcal{B}}$. 

\item\label{RdFcpt-2} $T$ is bounded on $L^{p_1}(w_1)$ for some $p_1 \in [1, \infty)$ and for all $w_1 \in A_{p_1, \mathcal{B}}$. 
\end{list} 
Then $T$ is compact on $L^p(w)$ for all $p \in (1, \infty)$ and for all $w \in A_{p, \mathcal{B}}$.  
\end{theorem}

\subsection{Proof strategy of Theorem \ref{thm:Bcpt}}
Similarly to the arguments above, to show Theorem \ref{thm:Bcpt}, it suffices to prove a compact bilinear dyadic representation and the compactness of the average of bilinear Zygmund shifts as follows.

\begin{theorem}\label{thm:Brepre}
Let $T$ be a $(D_{\vartheta}, D_{\theta}, \delta_1, \delta_{2, 3})$-BCZZ operator (cf. Definition {\rm \ref{def:BCZO}}), where $\vartheta \in (0, 2]$ and $\theta, \delta_1, \delta_{2, 3} \in (0, 1]$. Assume that $T$ satisfies the hypotheses \eqref{B1}--\eqref{B4}. Then $T$ admits a compact bilinear dyadic representation (cf. Definition {\rm \ref{def:Brepre}}). 
\end{theorem}

\begin{theorem}\label{thm:BSD-cpt}
Let $k = (k_1, k_2, k_3) \in \N^3$ and $\bm{l} = (l_1, l_2, l_3)$ with $l_j = (l_j^1, l_j^2, l_j^3) \in \N^3$. Suppose that the family $\{\mathbf{S}_{\D_{\w}}^{k, \bm{l}}\}_{\w \in \Omega}$ is of the same type (cf. Definition {\rm \ref{def:Bshift}}). Then $\E_{\w} \mathbf{S}_{\D_{\w}}^{k, \bm{l}}$ is compact from $L^{p_1}(\R^3) \times L^{p_2}(\R^3)$ to $L^p(\R^3)$ for all $p_1, p_2 \in (1, \infty)$, where $\frac1p = \frac{1}{p_1} + \frac{1}{p_2}$. 
\end{theorem}

Although the quasi-Banach compactness has been given in Theorem \ref{thm:BSD-cpt}, Theorem \ref{thm:Brepre} enables us to only obtain the compactness in Banach range in Theorem \ref{thm:Bcpt}. Although extrapolation of bilinear compact operators allows one to achieve the full range, there is no need to do so in the unweighted setting. In fact, considering Theorem \ref{thm:BSD-cpt} and the boundedness in the full range, we could apply the following interpolation for multilinear compact operators (cf. \cite[Theorem 3.4]{COY}).

\begin{theorem}\label{thm:inter-cpt}
Let $p_0, q_0 \in (0, \infty)$ and $p_1, q_1, p_2, q_2 \in [1, \infty]$. Assume that $T$ is a bilinear operator such that
\begin{list}{\rm (\theenumi)}{\usecounter{enumi}\leftmargin=1.2cm \labelwidth=1cm \itemsep=0.2cm \topsep=0.2cm \renewcommand{\theenumi}{\alph{enumi}}}

\item\label{inter-1} $T$ is compact from $L^{p_1}(\Rn) \times L^{p_2}(\Rn)$ to $L^{p_0}(\Rn)$;

\item\label{inter-2} $T$ is bounded from $L^{q_1}(\Rn) \times L^{q_2}(\Rn)$ to $L^{q_0}(\Rn)$. 
\end{list} 
Then $T$ is compact from $L^{r_1}(\Rn) \times L^{r_2}(\Rn)$ to $L^{r_0}(\Rn)$, where $\frac{1}{r_j} = \frac{1-\eta}{p_j} + \frac{\eta}{q_j}$ and $\eta \in (0, 1)$, $j = 0, 1, 2$.
\end{theorem}

This paper is organized as follows. Section \ref{sec:pre} contains some preliminaries including notation,  definitions, and elementary estimates. In Section \ref{sec:CZZ}, we define Calder\'{o}n--Zygmund operators associated with Zygmund dilations, compact Zygmund shifts, and a compact dyadic representation. Section \ref{sec:aux} is devoted to presenting some auxiliary results, which will be used in Section \ref{sec:cdr} to show Theorem \ref{thm:repre}. Besides, in Section \ref{sec:Zd}, we first prove Theorems \ref{thm:SD-cpt} and \ref{thm:RdF-cpt}, then combine them with Theorem \ref{thm:repre} to justify Theorem \ref{thm:cpt}. After that, we introduce bilinear Calder\'{o}n--Zygmund operators associated with Zygmund dilations, compact bilinear Zygmund shifts, and a compact bilinear dyadic representation in Section \ref{sec:BCZZ}, and then give the proof of  \ref{thm:Brepre} in Section \ref{sec:Brepre}. Finally, in Section \ref{sec:Bcpt}, we demonstrate Theorem \ref{thm:BSD-cpt}, which along with Theorem \ref{thm:inter-cpt} implies Theorem \ref{thm:Bcpt}.

\section{Preliminaries}\label{sec:pre}

\subsection{Notation}
We begin with some useful notation. 
\begin{itemize}
\item For convenience, we always use the $\ell^{\infty}$ metric on $\Rn$.
 
\item Let $\N := \{0, 1, 2, \ldots,\}$ be the set of natural numbers. 

 \item Let $\N_+ := \{1, 2, \ldots,\}$ be the set of positive integers. 
 
\item Let $\I := [-\frac12, \frac12]^n$ and $\lambda \I := [-\frac{\lambda}{2}, \frac{\lambda}{2}]^n$ for any $\lambda>0$. 

\item Given a measurable set $E \subset \Rn$, let $|E|$ denote the Lebesgue measure of $E$. 

\item Given a cube $I \subset \Rn$, we denote its center by $c_I$ and its side length by $\ell(I)$. For any $\lambda>0$, we denote by $\lambda I$ the cube with the center $c_I$ and side length $\lambda \ell(I)$. 

\item Let $\D$ denote a generic dyadic grid on $\Rn$ (see Section \ref{sec:dyadic}). 

\item Given $I \in \D$ and $k \in \N$, set $\ch(I):=\{I' \in \D: I' \subset I, \ell(I')=\ell(I)/2\}$ and let $I^{(k)}$ denote the unique dyadic cube $J \in \D$ so that $I \subset J$ and $\ell(J)=2^k \ell(I)$. 

\item Define $\d(E, F) := \inf\{|x-y|: x \in E, y \in F\}$ for any nonempty subsets $E, F \subset \Rn$.

\item Given $f \in L^1_{\loc}(\Rn)$ and a measurable set $E \subset \Rn$ with $0<|E|<\infty$, we write $\langle f \rangle_E := \fint_E f \, dx = \frac{1}{|E|} \int_E f \, dx$.

\item Let $\rd(I, J) := 1 + \frac{\d(I, J)}{\max\{\ell(I), \, \ell(J)\}}$ be the relative distance between cubes $I$ and $J$ in $\Rn$.

\item Given rectangles $I=I_1 \times \cdots \times I_n$ and $J=J_1 \times \cdots \times J_n$ in $\Rn$, the relative distance between them is defined by $\rd(I, J) :=\max\{\rd(I_1, J_1), \ldots, \rd(I_n, J_n)\}$.

\item For any $N \in \N$, denote $\D(N) := \big\{I \in \D:  \, 2^{-N} \le \ell(I) \le 2^N, \, \rd(I, 2^N \I) \le N\}$. 

\item Let $\mathcal{R}$ be the collection of all rectangles in $\R^3$ with sides parallel to the coordinate axes. 

\item Let $\Z := \{I \in \mathcal{R}: \ell(I_3) = \ell(I_1) \ell(I_2)\}$ denote the collection of all Zygmund rectangles.

\item We shall use $a\lesssim b$ and $a \simeq b$ to mean, respectively, that $a\leq C b$ and $0<c\leq a/b\leq C$, where the constants $c$ and $C$ are harmless positive constants, not necessarily the same at each occurrence, which depend only on dimension and the constants appearing in the hypotheses of theorems. 
\end{itemize}

\subsection{Dyadic grids}\label{sec:dyadic}
Let $\D_0$ be the standard dyadic grid on $\Rn$:
\begin{align*}
\D_0 := \big\{2^{-k}([0, 1)^n + m): k \in \ZZ, \, m \in \ZZ^n \big\}.
\end{align*}
Let $\Omega := (\{0, 1\}^n)^{\ZZ}$ and let $\mathbb{P}_{\w}$ be the natural probability measure on $\Omega$: each component $\w_j$ has an equal probability $2^{-n}$ of taking any of the $2^n$ values in $\{0, 1\}^n$, and all components are independent of each other. Let $\E_{\w}$ denote the expectation over the random variables $\w_j$, $j \in \ZZ$. 

Given $\w=(\w_j)_{j \in \ZZ} \in \Omega$, the \emph{random dyadic grid} $\D_{\w}$ on $\Rn$ is defined by
\begin{align*}
\D_{\w} := \bigg\{Q+\w := Q + \sum_{j: 2^{-j}< \ell(Q)} 2^{-j} \w_j : Q \in \D_0\bigg\}.
\end{align*}
For any integer $k \ge 2$,  a cube $Q \in \D_{\w}$ is called \emph{$k$-\rm{good}}, denoted by $Q \in \D_{\w}^k$, if 
\begin{align*}
\d(Q, \partial Q^{(k)}) \ge \ell(Q^{(k)})/4,  
\end{align*}
where $Q^{(k)}$ is the $k$th ancestor of $Q$, i.e., $Q \subset Q^{(k)}$ and $\ell(Q^{(k)}) = 2^k \ell(Q)$. 
It was shown in \cite{GH} that 
\begin{align}\label{pigd}
\pi_{\text{good}} := \mathbb{P}_{\w}(Q + \w \in \D_{\w}^k) = 2^{-n}, 
\quad \forall \, Q \in \D_0 \text{ and } k \ge 2.
\end{align}
In what follows, by a \emph{dyadic grid} $\D$ we mean $\D = \D_{\w}$ for some $\w \in \Omega$.

\subsection{Haar functions}\label{sec:Haar}
Given a rectangle $I = I_1 \times \cdots \times I_n \subset \Rn$ and $\eta=(\eta_1, \ldots, \eta_n) \in \{0, 1\}^n$, the \emph{Haar function} $h_I^{\eta}$ is defined by 
\begin{align*}
h_I^{\eta} := h_{I_1}^{\eta_1} \otimes \cdots \otimes h_{I_n}^{\eta_n},
\end{align*} 
where $h_{I_i}^0 = |I_i|^{-\frac12} \mathbf{1}_{I_i}$ and $h_{I_i}^1 = |I_i|^{-\frac12}(\mathbf{1}_{I_i^-} - \mathbf{1}_{I_i^+})$ for every $i = 1, \ldots, n$. Here $I_i^-$ and $I_i^+$ are the left and right halves of the interval $I_i$ respectively. Note that for any $\eta \neq 0$, the Haar function is cancellative: $\int_{\Rn} h_I^{\eta} \, dx = 0$. We usually suppress the presence of $\eta$ and simply write $h_I$ for some $h_I^{\eta}$, $\eta \neq 0$. 

If we view $\Rn$ as the bi-parameter product space $\Rn = \R^{n_1} \times \R^{n_2}$, and $I_1 \times I_2 \subset \R^{n_1} \times \R^{n_2}$, then $h_{I_1 \times I_2}$ still refers to the one-parameter Haar function: $h_{I_1 \times I_2} = h^{\eta}_{I_1 \times I_2}$ for some $\eta \in \{0, 1\}^n \setminus \{0\}$. It is not the bi-parameter Haar function on the rectangle $I_1 \times I_2$, which should have the different form $h^{\alpha}_{I_1} \otimes h^{\beta}_{I_2}$ for some $\alpha \in \{0, 1\}^{n_1} \setminus \{0\}$ and $\beta \in \{0, 1\}^{n_2} \setminus \{0\}$. In what follows, when a product object is required, we write it explicitly with this tensor notation.

Given a dyadic grid $\D$ on $\Rn$ and $I \in \D$, we define 
\begin{align*}
E_I f := \langle f \rangle_I \mathbf{1}_I
\quad\text{ and }\quad 
\Delta_I f := \sum_{I' \in \ch(I)} \big(\langle f \rangle_{I'} - \langle f \rangle_I\big) \mathbf{1}_{I'}.
\end{align*}
Then one can see that 
\begin{align}\label{eq:EDF}
E_I f = \langle f, h_I^0 \rangle h_I^0 
\quad\text{ and }\quad 
\Delta_I f = \sum_{\eta \in \{0, 1\}^n \setminus \{0\}} \langle f, h_I^{\eta} \rangle h_I^{\eta}
=: \langle f, \, h_I \rangle \, h_I,
\end{align}
where we suppress the finite $\eta$ summation in the last step. Since all the cancellative Haar functions on $\Rn$ form an orthonormal basis of $L^2(\Rn)$, for each $f \in L^2(\Rn)$, we may write
\begin{align}\label{eq:fhh} 
f = \sum_{I \in \D}  \langle f, h_I \rangle \, h_I
=\sum_{I \in \D} \Delta_I f.
\end{align}

Given $k \in \ZZ$, if we define 
\begin{align}\label{def:Ek-1}
E_k f := \sum_{I \in \D : \, \ell(I)=2^k} \langle f \rangle_I \, \mathbf{1}_I 
\quad\text{ and }\quad 
\Delta_k f:= E_{2^{k-1}} f - E_{2^k} f,  
\end{align}
then 
\begin{align}\label{def:Ek-2}
E_k f = \sum_{I \in \D: \, \ell(I)>2^k} \Delta_I f 
\quad\text{ and }\quad  
\Delta_k f = \sum_{I \in \D: \, \ell(I)=2^k} \Delta_I f. 
\end{align}

\subsection{Resolution of functions on $\R^3$} 
We are working on $\R^3$ in this subsection. Given dyadic grids $\D^1$, $\D^2$, and $\D^3$ on $\R$, we define $\D := \D^1 \times \D^2 \times \D^3$ and 
\begin{align*}
\D_{\Z} := \big\{I=I_1 \times I_2 \times I_3 \in \D: \ell(I_3) = \ell(I_1) \ell(I_2) \big\}.
\end{align*} 
Given $I=I_1 \times I_2 \times I_3 =: I_1 \times I_{2, 3}\in \D_{\Z}$, the \emph{Zygmund martingale difference} is defined by 
\begin{align*}
\Delta_{I, \Z} f := \Delta_{I_1} \Delta_{I_{2, 3}} f, 
\quad \text{where } \, 
\Delta_{I_{2, 3}} f
:= \sum_{\substack{I'_2 \in \ch(I_2) \\ I'_3 \in \ch(I_3)}}
\big(\langle f \rangle_{I'_2 \times I'_3} - \langle f \rangle_{I_2 \times I_3} \big) \mathbf{1}_{I'_2 \times I'_3} .
\end{align*}
Note that $\Delta_{2, 3}$ is the one-parameter martingale difference operator on the rectangle $I_{2, 3}$, which is quite different from the bi-parameter martingale difference operator $\Delta_{I_2} \Delta_{I_3}$. In addition, it is easy to check that for all $I, J \in \D_{\Z}$, 
\begin{align*}
\Delta_{I, \Z} \Delta_{J, \Z} = \Delta_{I, \Z} \mathbf{1}_{\{I=J\}}, \qquad 
\int_{\R} \Delta_{I, \Z} f \, dx_1 = 0 
= \int_{\R^2} \Delta_{I, \Z} f \, dx_2 \, dx_3, 
\end{align*}
and 
\begin{align*}
\Delta_{I, \Z} f 
=\Delta_{I_1} \Delta_{I_{2, 3}} f 
= \langle f, h_{I_1} \otimes h_{I_{2, 3}} \rangle h_{I_1} \otimes h_{I_{2, 3}}
=: \langle f, h_{I, \Z} \rangle h_{I, \Z}. 
\end{align*}
For any $\lambda=2^k$ with $k \in \mathbb{Z}$, we define the dilated grid  
\begin{align*}
\D_{\lambda}^{2, 3} 
:= \big\{I_{2, 3} \in \D^{2, 3} = \D^2 \times \D^3: \ell(I_3) = \lambda \ell(I_2) \big\}.
\end{align*}
By means of these notation and \eqref{eq:fhh}, we may expand a function $f \in L^2(\R^3)$ as 
\begin{align}\label{eq:expand}
f = \sum_{I_1 \in \D^1} \Delta_{I_1} f 
= \sum_{I_1 \in \D^1} \sum_{I_{2, 3} \in \D^{2, 3}_{\ell(I_1)}} \Delta_{I_1}  \Delta_{I_{2, 3}} f 
=\sum_{I \in \D_{\Z}} \Delta_{I, \Z} f. 
\end{align}

\subsection{Weights and maximal operators}
Let $\mathcal{B}$ be a \emph{basis} of $\R^3$, that is, a collection of measurable sets in $\R^3$ such that $0 < |B| < \infty$ for every $B \in \mathcal{B}$ and $\R^3 = \bigcup_{B \in \mathcal{B}} B$. A measurable function $w$ on $\R^3$ is called a \emph{weight} if it is a.e. positive and locally integrable. Given $p \in (1, \infty)$ and a basis $\mathcal{B}$ of $\R^3$, we say that a weight $w$ on $\R^3$ belongs to the class $A_{p, \mathcal{B}}$ if
\begin{align*}
[w]_{A_{p, \mathcal{B}}} 
:= \sup_{B \in \mathcal{B}} \bigg(\fint_B w \, dx \bigg) \bigg(\fint_B w^{-\frac{1}{p-1}} dx \bigg)^{p-1} < \infty. 
\end{align*}

Given a basis $\mathcal{B}$ of $\R^3$, the associated maximal operator $M_{\mathcal{B}}$ is defined by 
\begin{align*}
M_{\mathcal{B}} f(x) 
:= \sup_{B \in \mathcal{B}} \langle |f| \rangle_B \, \mathbf{1}_B(x). 
\end{align*}
Let us give some examples of basis of $\R^3$: 
\begin{align*}
\Z & := \{I \in \mathcal{R}: \ell(I_3) = \ell(I_1) \ell(I_2)\}, 
\\ 
\D_{\Z} & := \{I \in \D: \ell(I_3) = \ell(I_1) \ell(I_2)\}, 
\\ 
\D_{\lambda} & := \{I \in \D: \ell(I_3) = \lambda \ell(I_1) \ell(I_2)\}, 
\\ 
\D_{\text{sub}} & := \{I \in \D: \ell(I_3) \le \ell(I_1) \ell(I_2)\},
\end{align*}
where $\lambda = 2^k$ with $k \in \ZZ$. Then 
\begin{align*}
M_{\D_{\Z}} \le M_{\D_{\text{sub}}} 
\quad \text{and} \quad 
M_{\D_{\lambda}} \le M_{\D_{\text{sub}}}, \quad \lambda = 2^k \text{ with } k \le 0.
\end{align*}
It follows from \cite[p. 453]{GR} and \cite[Theorem 1.2]{FP} that 
\begin{align}\label{MB}
\|M_{\mathcal{B}} f\|_{L^p(w)} 
\lesssim \|f\|_{L^p(w)}, \quad\forall p \in (1, \infty), \, w \in A_{p, \mathcal{B}}, \, 
\mathcal{B} \in \{\mathcal{R}, \mathcal{Z}\}. 
\end{align}
Moreover, it was shown in \cite[Section 5]{HLMV} that for any $p, r \in (1, \infty)$ and $w \in A_{p, \Z}$, 
\begin{align}
\label{MM-1}
&\bigg\| \bigg(\sum_{j \in \ZZ} |M_{\D_{\lambda}} f_j|^r \bigg)^{\frac1r}\bigg\|_{L^p(w)} 
\lesssim \max\{\lambda, 1\}^{1-\eta} \bigg\| \bigg(\sum_{j \in \ZZ} |f_j|^r \bigg)^{\frac1r}\bigg\|_{L^p(w)}, 
\\
\label{MM-2}
& \bigg\| \bigg(\sum_{j \in \ZZ} |M_{\D_{\rm{sub}}} f_j|^r \bigg)^{\frac1r}\bigg\|_{L^p(w)} 
\lesssim \bigg\| \bigg(\sum_{j \in \ZZ} |f_j|^r \bigg)^{\frac1r}\bigg\|_{L^p(w)},  
\\
\label{MM-3}
& \bigg\| \bigg(\sum_{K \in \D_{\lambda(k)}} |\mathcal{U}_{K, k} f|^2 \bigg)^{\frac12}\bigg\|_{L^p(w)} 
\lesssim (|k|+1) \|f\|_{L^p(w)},  
\end{align}
where $\eta = \eta(p, w) \in (0, 1)$, $\lambda(k) = 2^{-k_1-k_2+k_3}$, and 
\begin{align*}
\mathcal{U}_{K, k} f 
= \mathbf{1}_K \sum_{\substack{I \in \D_{\Z}: \, I_1 \subset K_1, \, \ell(I_1) \ge 2^{-k_1} \ell(K_1) 
\\ 2^{-k_1} \ell(K_2) \le 2^{k_2} \ell(I_2) \le 2^{\max\{k_2, k_3\}} \ell(K_2)}} \Delta_{I, \Z} f. 
\end{align*}

\section{Calder\'{o}n--Zygmund operators associated with Zygmund dilations}\label{sec:CZZ}
We are going to define Calder\'{o}n--Zygmund operators associated with Zygmund dilations. Let us first introduce some function spaces.

\begin{definition}\label{def:FF}
Let $\mathscr{F}$ consist of all triples $(F_1, F_2, F_3)$ of bounded functions $F_1, F_2, F_3: [0, \infty) \to [0, \infty)$ satisfying 
\begin{align*}
\lim_{t \to 0} F_1(t) 
=\lim_{t \to \infty} F_2(t) 
= \lim_{t \to \infty} F_3(t)
=0.  
\end{align*}
Let $\mathscr{F}_0$ be the collection of all bounded functions $F: \mathcal{I} \to [0, \infty)$ satisfying 
\begin{align*}
\lim_{|I| \to 0} F(I)
= \lim_{|I| \to \infty} F(I)
= \lim_{|c_I| \to \infty} F(I)
=0, 
\end{align*}
where $\mathcal{I}$ denotes the family of all intervals on $\R$. 
\end{definition}

Given a parameter $\theta \in (0, 1]$, denote 
\begin{align*}
D_{\theta}(x) := \bigg(\frac{|x_1| |x_2|}{|x_3|} + \frac{|x_3|}{|x_1| |x_2|} \bigg)^{-\theta} 
\quad \text{and} \quad 
D_{\log}(x) := D_1(x) \log \bigg(\frac{|x_1| |x_2|}{|x_3|} + \frac{|x_3|}{|x_1| |x_2|} \bigg)
\end{align*}
for any $x =(x_1, x_2, x_3) \in \R^3 \setminus \{x: x_1 x_2 x_3 = 0 \}$.

Let us proceed to give the compact full and partial kernel representations, and then define singular integral operators associated with Zygmund dilations. Let $\delta_1, \delta_{2, 3} \in (0, 1]$ be fixed numbers.

\begin{definition}\label{def:full}
A linear operator $T$ admits the \emph{compact full kernel representation} if the following hold. 
If $f = f_1 \otimes f_{2, 3}$ and $g = g_1 \otimes g_{2, 3}$ with $f_1, g_1 :\R \rightarrow \R$, $f_{2, 3}, g_{2, 3} :\R^2 \rightarrow \R$, 
$\supp(f_1) \cap \supp(g_1) = \emptyset$, and $\supp(f_{2, 3}) \cap \supp(g_{2, 3}) = \emptyset$, then 
\begin{align*}
\langle Tf, g \rangle 
= \int_{\R^3} \int_{\R^3} K(x, y) f(y) g(x) \, dx \, dy,
\end{align*}
where the kernel $K: (\R^3 \times \R^3) \setminus \big\{(x, y) \in \R^3 \times \R^3: x_1=y_1 \text{ or } x_2=y_2 \text{ or } x_3=y_3\big\} \rightarrow \C$ satisfies 
\begin{list}{\rm (\theenumi)}{\usecounter{enumi}\leftmargin=1.2cm \labelwidth=1cm \itemsep=0.2cm \topsep=.2cm \renewcommand{\theenumi}{\arabic{enumi}}} 

\item\label{full-1} the size condition: 
\begin{align*}
|K(x, y)| \leq D_{\theta}(x-y) \prod_{i=1}^3 \frac{F_i(x_i, y_i)}{|x_i - y_i|}. 
\end{align*}

\item\label{full-2} the H\"{o}lder conditions: 
\begin{align*}
& |K(x, y) - K((x_1, x'_2, x'_3), y) - K((x'_1, x_2, x_3), y) + K(x', y)| 
\\
& \leq \bigg(\frac{|x_1 - x'_1|}{|x_1 - y_1|}\bigg)^{\delta_1} 
\bigg(\frac{|x_2 - x'_2|}{|x_2 - y_2|} + \frac{|x_3 - x'_3|}{|x_3 - y_3|}\bigg)^{\delta_{2, 3}}
D_{\theta}(x-y) \prod_{i=1}^3 \frac{F_i(x_i, y_i)}{|x_i - y_i|} 
\end{align*}
whenever $|x_i - x'_i| \leq |x_i - y_i|/2$ for $i=1, 2, 3$, 
\begin{align*}
& |K(x, y) - K(x, (y_1, y'_2, y'_3)) - K(x, (y'_1, y_2, y_3)) + K(x, y')| 
\\
& \leq \bigg(\frac{|y_1 - y'_1|}{|x_1 - y_1|}\bigg)^{\delta_1} 
\bigg(\frac{|y_2 - y'_2|}{|x_2 - y_2|} + \frac{|y_3 - y'_3|}{|x_3 - y_3|}\bigg)^{\delta_{2, 3}}
D_{\theta}(x-y) \prod_{i=1}^3 \frac{F_i(x_i, y_i)}{|x_i - y_i|}
\end{align*}
whenever $|y_i - y'_i| \leq |x_i - y_i|/2$ for $i=1, 2, 3$, 
\begin{align*}
& |K(x, y) - K((x_1, x'_2, x'_3), y) - K(x, (y'_1, y_2, y_3)) + K(x_1, x'_2, x'_3, y'_1, y_2, y_3)| 
\\
& \leq \bigg(\frac{|y_1 - y'_1|}{|x_1 - y_1|}\bigg)^{\delta_1} 
\bigg(\frac{|x_2 - x'_2|}{|x_2 - y_2|} + \frac{|x_3 - x'_3|}{|x_3 - y_3|}\bigg)^{\delta_{2, 3}}
D_{\theta}(x-y) \prod_{i=1}^3 \frac{F_i(x_i, y_i)}{|x_i - y_i|}
\end{align*}
whenever $|y_1 - y'_1| \leq |x_1 - y_1|/2$ and $|x_i - x'_i| \leq |x_i - y_i|/2$ for $i=2, 3$, 
\begin{align*}
& |K(x, y) - K(x, (y_1, y'_2, y'_3)) - K((x'_1, x_2, x_3), y) + K((x'_1, x_2, x_3), (y_1, y'_2, y'_3))| 
\\
& \leq \bigg(\frac{|x_1 - x'_1|}{|x_1 - y_1|}\bigg)^{\delta_1} 
\bigg(\frac{|y_2 - y'_2|}{|x_2 - y_2|} + \frac{|y_3 - y'_3|}{|x_3 - y_3|}\bigg)^{\delta_{2, 3}}
D_{\theta}(x-y) \prod_{i=1}^3 \frac{F_i(x_i, y_i)}{|x_i - y_i|}
\end{align*}
whenever $|x_1 - x'_1| \leq |x_1 - y_1|/2$ and $|y_i - y'_i| \leq |x_i - y_i|/2$ for $i=2, 3$.

\item\label{full-3} the mixed size-H\"{o}lder conditions:  
\begin{align*}
|K(x, y) - K((x'_1, x_2, x_3), y)|
\leq \bigg(\frac{|x_1 - x'_1|}{|x_1 - y_1|}\bigg)^{\delta_1} 
D_{\theta}(x-y) \prod_{i=1}^3 \frac{F_i(x_i, y_i)}{|x_i - y_i|}, 
\end{align*}
whenever $|x_1 - x'_1| \leq |x_1-y_1|/2$,  
\begin{align*}
|K(x, y) - K((x_1, x'_2, x'_3), y)|
\le \bigg(\frac{|x_2 - x'_2|}{|x_2 - y_2|} + \frac{|x_3 - x'_3|}{|x_3 - y_3|}\bigg)^{\delta_{2, 3}}
D_{\theta}(x-y) \prod_{i=1}^3 \frac{F_i(x_i, y_i)}{|x_i - y_i|}, 
\end{align*}
whenever $|x_2 - x'_2| \leq |x_2 - y_2|/2$ and $|x_3 - x'_3| \leq |x_3 - y_3|/2$, 
\begin{align*}
|K(x, y) - K(x, (y'_1, y_2, y_3))|
\leq \bigg(\frac{|y_1 - y'_1|}{|x_1 - y_1|}\bigg)^{\delta_1} 
D_{\theta}(x-y) \prod_{i=1}^3 \frac{F_i(x_i, y_i)}{|x_i - y_i|}, 
\end{align*}
whenever $|y_1 - y'_1| \leq |x_1-y_1|/2$, and 
\begin{align*}
|K(x, y) - K(x, (y_1, y'_2, y'_3))|
\le \bigg(\frac{|y_2 - y'_2|}{|x_2 - y_2|} + \frac{|y_3 - y'_3|}{|x_3 - y_3|}\bigg)^{\delta_{2, 3}}
D_{\theta}(x-y) \prod_{i=1}^3 \frac{F_i(x_i, y_i)}{|x_i - y_i|}, 
\end{align*}
whenever $|y_2 - y'_2| \leq |x_2 - y_2|/2$ and $|y_3 - y'_3| \leq |x_3 - y_3|/2$. 

\item\label{full-4} the function $F_i$ above is given by 
\begin{align*}
F_i(x_i, y_i) := F_{i, 1}(|x_i - y_i|) F_{i, 2}(|x_i - y_i|) F_{i, 3}(|x_i + y_i|),  
\end{align*}
where $(F_{i, 1}, F_{i, 2}, F_{i, 3}) \in \mathscr{F}$, $i=1, 2, 3$.
\end{list}
\end{definition}

\begin{definition}\label{def:partial-1}
A linear operator $T$ admits a \emph{compact partial kernel representation on the parameter $\{1\}$} if the following hold. 
If $f = f_1 \otimes f_{2, 3}$ and $g = g_1 \otimes g_{2, 3}$ with $\supp(f_1) \cap \supp(g_1) = \emptyset$, then 
\begin{align*}
\langle Tf, g \rangle 
= \int_{\R} \int_{\R} K_{f_{2, 3}, g_{2, 3}}(x_1, y_1) f_1(y_1) g_1(x_1) \, dx_1 \, dy_1, 
\end{align*}
where the kernel $K_{f_{2, 3}, g_{2, 3}}: (\R \times \R) \setminus \big\{(x_1, y_1) \in \R \times \R: x_1 = y_1 \big\} \rightarrow \C$ satisfies  
\begin{list}{\rm (\theenumi)}{\usecounter{enumi}\leftmargin=1.2cm \labelwidth=1cm \itemsep=0.2cm \topsep=.2cm \renewcommand{\theenumi}{\arabic{enumi}}} 

\item\label{partial-11} the size condition:
\begin{align*}
|K_{f_{2, 3}, g_{2, 3}}(x_1,y_1)| 
\leq C(f_{2, 3}, g_{2, 3}) \frac{F_1(x_1, y_1)}{|x_1-y_1|}. 
\end{align*}

\item\label{partial-12} the H\"{o}lder conditions:
\begin{align*}
|K_{f_{2, 3}, g_{2, 3}}(x_1, y_1) - K_{f_{2, 3}, g_{2, 3}}(x'_1, y_1)|
\leq C(f_{2, 3}, g_{2, 3}) F_1(x_1, y_1) \frac{|x_1-x'_1|^{\delta_1}}{|x_1-y_1|^{1+\delta_1}}
\end{align*}
whenever $|x_1-x'_1| \leq |x_1-y_1|/2$, and 
\begin{align*}
|K_{f_{2, 3}, g_{2, 3}}(x_1, y_1) - K_{f_{2, 3}, g_{2, 3}}(x_1, y'_1)|
\leq C(f_{2, 3}, g_{2, 3}) F_1(x_1, y_1) \frac{|y_1-y'_1|^{\delta_1}}{|x_1-y_1|^{1+\delta_1}}
\end{align*}
whenever $|y_1-y'_1| \leq |x_1-y_1|/2$.

\item\label{partial-13} the function $F_1$ above is given by 
\begin{align*}
F_1(x_1, y_1) := F_{1, 1}(|x_1 - y_1|) F_{1, 2}(|x_1 - y_1|) F_{1, 3}(|x_1 + y_1|), 
\end{align*}
where $(F_{1, 1}, F_{1, 2}, F_{1, 3}) \in \mathscr{F}$.

\item\label{partial-14} the bound $C(f_{2, 3}, g_{2, 3})$ above verifies  
\begin{align*}
C(\mathbf{1}_{I_{2, 3}}, \mathbf{1}_{I_{2, 3}}) 
+ C(\mathbf{1}_{I_{2, 3}}, a_{I_{2, 3}}) 
+ C(a_{I_{2, 3}}, \mathbf{1}_{I_{2, 3}}) 
\le F_2(I_2) \, |I_2| \, F_3(I_3) \, |I_3|,  
\end{align*}
for all rectangles $I_{2, 3} \subset \R^2$ and all functions $a_{I_{2, 3}}$ satisfying $\supp(a_{I_{2, 3}}) \subset I_{2, 3}$, $|a_{I_{2, 3}}| \le 1$, and $\int_{\R^2} a_{I_{2, 3}} \, dx_{2, 3}=0$, where $F_2, F_3 \in \mathscr{F}_0$. 
\end{list} 
\end{definition}

\begin{definition}\label{def:partial-2}
A linear operator $T$ admits a \emph{compact partial kernel representation on the parameter $\{2, 3\}$} if the following hold. 
If $f = f_1 \otimes f_{2, 3}$ and $g = g_1 \otimes g_{2, 3}$ with $\supp(f_{2, 3}) \cap \supp(g_{2, 3}) = \emptyset$, then 
\begin{align*}
\langle Tf, g \rangle 
= \int_{\R^2} \int_{\R^2} K_{f_1, g_1}(x_{2, 3}, y_{2, 3}) f_{2, 3}(y_{2, 3}) g_{2, 3}(x_{2, 3}) \, dx_{2, 3} \, dy_{2, 3}, 
\end{align*}
where the kernel $K_{f_1, g_1}: (\R^2 \times \R^2) \setminus \big\{(x_{2, 3}, y_{2, 3}) \in \R^2 \times \R^2: x_2 = y_2 \text{ or } x_3=y_3 \big\} \rightarrow \C$ satisfies  
\begin{list}{\rm (\theenumi)}{\usecounter{enumi}\leftmargin=1.2cm \labelwidth=1cm \itemsep=0.2cm \topsep=.2cm \renewcommand{\theenumi}{\arabic{enumi}}} 

\item\label{partial-21} the size condition:
\begin{align*}
|K_{f_1, g_1}(x_{2, 3}, y_{2, 3})| 
\leq C(f_1, g_1) D_{\theta}(t, x_{2, 3} - y_{2, 3}) 
\prod_{i=2}^3 \frac{F_i(x_i, y_i)}{|x_i-y_i|}, 
\end{align*}
where $t:= |\supp f_1 \cup \supp g_1|$.

\item\label{partial-22} the H\"{o}lder conditions: 
\begin{align*}
& |K_{f_1, g_1}(x_{2, 3}, y_{2, 3}) - K_{f_1, g_1}(x'_{2, 3}, y_{2, 3})|
\\ 
& \leq C(f_1, g_1) \bigg(\frac{|x_2 - x'_2|}{|x_2 - y_2|} + \frac{|x_3 - x'_3|}{|x_3 - y_3|}\bigg)^{\delta_{2, 3}} 
D_{\theta}(t, x_{2, 3} - y_{2, 3}) \prod_{i=2}^3 \frac{F_i(x_i, y_i)}{|x_i-y_i|} 
\end{align*}
whenever $x'_{2, 3} =(x'_2, x'_3)$ satisfies $|x_2-x'_2| \leq |x_2-y_2|/2$ and $|x_3-x'_3| \leq |x_3-y_3|/2$, and 
\begin{align*}
& |K_{f_1, g_1}(x_{2, 3}, y_{2, 3}) - K_{f_1, g_1}(x_{2, 3}, y'_{2, 3})|
\\ 
& \leq C(f_1, g_1) \bigg(\frac{|y_2 - y'_2|}{|x_2 - y_2|} + \frac{|y_3 - y'_3|}{|x_3 - y_3|}\bigg)^{\delta_{2, 3}} 
D_{\theta}(t, x_{2, 3} - y_{2, 3}) \prod_{i=2}^3 \frac{F_i(x_i, y_i)}{|x_i-y_i|} 
\end{align*}
whenever $y'_{2, 3} =(y'_2, y'_3)$ satisfies $|y_2-y'_2| \leq |x_2-y_2|/2$ and $|y_3-y'_3| \leq |x_3-y_3|/2$. 

\item\label{partial-23} the function $F_i$ above is given by 
\begin{align*}
F_i(x_i, y_i) := F_{i, 1}(|x_i - y_i|) F_{i, 2}(|x_i - y_i|) F_{i, 3}(|x_i + y_i|), 
\end{align*}
where $(F_{i, 1}, F_{i, 2}, F_{i, 3}) \in \mathscr{F}$, $i=2, 3$.

\item\label{partial-24} the bound $C(f_1, g_1)$ above verifies  
\begin{align*}
C(\mathbf{1}_{I_1}, \mathbf{1}_{I_1}) 
+ C(\mathbf{1}_{I_1}, a_{I_1}) 
+ C(a_{I_1}, \mathbf{1}_{I_1}) 
\le F_1(I_1) \, |I_1|, 
\end{align*}
for all intervals $I_1 \subset \R$ and all functions $a_{I_1}$ satisfying $\supp(a_{I_1}) \subset I_1$, $|a_{I_1}| \le 1$, and $\int_{\R} a_{I_1} \, dx_1=0$, where $F_1 \in \mathscr{F}_0$. 
\end{list} 
\end{definition}

\begin{definition}\label{def:partial}
A linear operator $T$ admits the \emph{compact partial kernel representation} if it admits the compact partial kernel representation on the parameters $\{1\}$ and $\{2, 3\}$. 
\end{definition}

Moreover, we would like to define the compactness and cancellation assumptions for singular integrals associated with Zygmund dilations.

\begin{definition}\label{def:WCP}
We say that $T$ satisfies the \emph{weak compactness property} if 
\begin{align*}
|\langle T\mathbf{1}_I, \mathbf{1}_I \rangle|
\leq F_1(I_1) F_2(I_2) F_3(I_3) \, |I|, 
\end{align*}
for all Zygmund rectangles $I=I_1 \times I_2 \times I_3$, where $F_1, F_2, F_3 \in \mathscr{F}_0$. If the functions $F_1, F_2, F_3$ above are replaced by a uniform constant $C \ge 1$, we say that $T$ satisfies the \emph{weak boundedness property}.
\end{definition}

\begin{definition}\label{def:cancellation}
We say that $T$ satisfies the \emph{cancellation condition} if   
\begin{align*}
& \langle T(1 \otimes \mathbf{1}_{I_{2, 3}}), h_{J_1} \otimes \mathbf{1}_{J_{2, 3}} \rangle
= \langle T(\mathbf{1}_{I_1} \otimes 1), \mathbf{1}_{J_1} \otimes h_{J_{2, 3}} \rangle
= 0, 
\\
& \langle T(h_{I_1} \otimes \mathbf{1}_{I_{2, 3}}), 1 \otimes \mathbf{1}_{J_{2, 3}} \rangle 
= \langle T(\mathbf{1}_{I_1} \otimes h_{I_{2, 3}}), \mathbf{1}_{J_1} \otimes 1 \rangle
= 0, 
\end{align*}
for all rectangles $I = I_1 \times I_2 \times I_3=: I_1 \times I_{2, 3}$ and $J = J_1 \times J_2 \times J_3=: J_1 \times J_{2, 3}$.
\end{definition}

\begin{definition}\label{def:CZO}
Let $T$ be a linear operator. 
\begin{itemize}

\item We say that $T$ admits the \emph{full kernel representation} if the functions $F_1, F_2, F_3$ in the compact full kernel representation are replaced by a uniform constant $C \ge 1$. 

\item We say that $T$ admits the \emph{partial kernel representation} if the functions $F_1, F_2, F_3$ in the compact partial kernel representation are replaced by a uniform constant $C \ge 1$. 

\item $T$ is called a \emph{$(D_{\theta}, \delta_1, \delta_{2, 3})$-CZZ operator} if $T$ admits the full and partial kernel representations, and satisfies the weak boundedness property and the cancellation condition. 

\item As above, $T$ is called a \emph{$(D_{\log}, \delta_1, \delta_{2, 3})$-CZZ operator} if $D_{\theta}$ is replaced by $D_{\log}$. 
\end{itemize}
\end{definition}

Next, let us introduce dyadic shifts associated with Zygmund dilations, which own simple dyadic structure and are compact operators on $L^p(\R^3)$, $1 < p < \infty$. This is crucial to show the compactness of singular integral operators.

\begin{definition}\label{def:shift}
Given $k = (k_1, k_2, k_3) \in \N^3$ and $\D = \D^1 \times \D^2 \times \D^3$ with $\D^1, \D^2, \D^3$ being dyadic grids on $\R$, a \emph{compact Zygmund shift $\mathbf{S}_{\D}^k$ of complexity $k$} is an operator of the form  
\begin{align*}
\mathbf{S}_{\D}^k f
:= \sum_{K \in \D_{\lambda(k)}} 
\sum_{\substack{I, J \in \D_{\Z} \\ I^{(k)} = J^{(k)} = K}}   
a_{IJK} \, \langle f, \widetilde{h}_{I, \Z} \rangle \, \widetilde{h}_{J, \Z}, 
\end{align*}
where $\lambda(k) := 2^{-k_1 - k_2 + k_3}$, $\D_{\lambda(k)} := \{K \in \D: \ell(K_3) = \lambda(k) \ell(K_1) \ell(K_2)\}$, $I^{(k)} := I_1^{(k_1)} \times I_2^{(k_2)} \times I_3^{(k_3)}$, and 
\begin{equation*}
(\widetilde{h}_{I, \Z}, \widetilde{h}_{J, \Z}) 
\in 
\left\{
\begin{aligned}
& (h_{I, \Z}, h_{J, \Z}), 
\\ 
& (h_{I_1} \otimes H_{I_{2, 3}, J_{2, 3}}, \, h_{J, \Z}), \qquad\qquad \, 
(h_{I, \Z}, h_{J_1} \otimes H_{I_{2, 3}, J_{2, 3}}), 
\\ 
& (H_{I_1, J_1} \otimes h_{I_{2, 3}}, \, h_{J, \Z}), \qquad\qquad \, \, \, \, \, 
(h_{I, \Z}, H_{I_1, J_1} \otimes h_{J_{2, 3}}), 
\\
& (H_{I_1, J_1} \otimes H_{I_{2, 3}, J_{2, 3}}, \, h_{J, \Z}), \qquad \quad
(h_{I, \Z}, \, H_{I_1, J_1} \otimes H_{I_{2, 3}, J_{2, 3}}), 
\\
& (H_{I_1, J_1} \otimes h_{I_{2, 3}}, \, h_{J_1} \otimes H_{I_{2, 3}, J_{2, 3}}), \, \, 
(h_{I_1} \otimes H_{I_{2, 3}, J_{2, 3}}, \, H_{I_1, J_1} \otimes h_{J_{2, 3}})
\end{aligned}
\right\}
\end{equation*} 
with $H_{I_1, J_1} := h_{I_1}^0 - h_{J_1}^0$ and $H_{I_{2, 3}, J_{2, 3}} := h_{I_{2, 3}}^0 - h_{J_{2, 3}}^0$. Moreover, the coefficients $a_{IJK}$ satisfy  
\begin{align*}
|a_{IJK}| 
\le \mathbf{F}(K) \frac{|I|^{\frac12} |J|^{\frac12}}{|K|}
\end{align*}
with  
\begin{align*}
\mathbf{F}(K) \le 1 
\quad\text{and}\quad 
\lim_{N \to \infty} \mathbf{F}_N 
:= \lim_{N \to \infty} \sup_{\D} \sup_{K \not\in \D_{\lambda(k)}^N} 
\mathbf{F}(K) = 0,  
\end{align*}
where $K \notin \D_{\lambda(k)}^N$ means $K \in \D_{\lambda(k)} \setminus \D_{\lambda(k)}^N$, and 
\begin{align*}
\D_{\lambda(k)}^N 
:= \big\{K \in \D_{\lambda(k)}:  \, 
&2^{-N} \le \ell(K_1) \le 2^N, \, \rd(K_1, 2^N \I) \le N, 
\\ \nonumber 
&2^{-N} \le \ell(K_2) \le 2^N, \, \rd(K_2, 2^N  \I) \le N, \, 
\rd(K_3, 2^{2N} \I) \le 2N \big\}. 
\end{align*}
\end{definition}

Finally, we introduce the compact dyadic representation of Calder\'{o}n--Zygmund operators associated with Zygmund dilations. 

\begin{definition}\label{def:repre}
Given $\theta, \delta_1, \delta_{2, 3} \in (0, 1]$, we say that a $(D_{\theta}, \delta_1, \delta_{2, 3})$-CZZ operator $T$ admits a \emph{compact dyadic representation} if there exists a constant $C_0 = C_0(T) \in (0,  \infty)$ such that for all compactly supported and bounded functions $f, g$ on $\R^3$, 
\begin{align*}
\langle Tf, g \rangle
&= C_0 \, \E_{\w} 
\sum_{\substack{k = (k_1, k_2, k_2) \\ k_1, k_2, k_3 \ge 2}} 
\sum_{m=1}^{m_0} \varphi(k) 
\langle \mathbf{S}_{m, \D_{\w}}^k f, g \rangle, 
\end{align*} 
where 
\begin{align*}
& \text{$m_0$ is an absolute constant}, 
\\
& \varphi(k) := 2^{-k_1 \delta_1} 
2^{- k_2 \min\{\delta_{2, 3}, \, \theta\}} 
2^{- \max\{k_3 - k_1 - k_2, \, 0\} \theta}, \text{ and} 
\\
& \text{$\mathbf{S}_{m, \D_{\w}}^k$ is a compact Zygmund shift of complexity $k$ on $\D_{\w}$}.
\end{align*}
\end{definition}

\section{Auxiliary results}\label{sec:aux} 
This section is devoted to presenting some auxiliary estimates to show Theorem \ref{thm:repre}.

\subsection{Some elementary estimates}
We begin with the convergence of a series as follows.

\begin{lemma}\label{lem:kkk}
Let $\rho \ge 1$, $\alpha_2 > 0$, and $\alpha_0, \alpha_1 > \alpha_3 \ge 0$. Then for $i=1, 2, 3$ and $A \ge 0$, there holds 
\begin{align*}
\mathscr{I}_i(A) 
&:= \sum_{\substack{k_1, k_2, k_3 \ge 0 \\ k_i \ge A}}  
(k_1 + k_2 + k_3 + 1)^{\rho} \, 
\frac{2^{-k_1 \alpha_1} 2^{- k_2 \alpha_2} \, 2^{\max\{k_3 - k_2, \, 0\} \alpha_3}}{2^{\max\{k_3 - k_1 - k_2, \, 0\} \alpha_0}}
\\
& \lesssim (A+1)^{3\rho+4} \, 2^{-A\min\{\alpha_0 - \alpha_3, \, \alpha_1 - \alpha_3, \, \alpha_2\}/2},  
\end{align*}
where the implicit constant is independent of $A$.
\end{lemma}

\begin{proof}
For each $i=1, 2, 3$, we split 
\begin{align}\label{kkk-1}
\mathscr{I}_i(A) 
&\lesssim \sum_{\substack{k_1, k_2, k_3 \ge 0 \\ k_3 \ge k_1 + k_2 \\ k_i \ge A}}  
(k_1 + k_2 + k_3 + 1)^{\rho}
\frac{2^{-k_1 \alpha_1} 2^{-k_2 \alpha_2} 2^{(k_3 - k_2) \alpha_3}}{2^{(k_3 - k_1 - k_2) \alpha_0}} 
\\ \nonumber 
&\quad + \sum_{\substack{k_1, k_2, k_3 \ge 0 \\ k_2 \le k_3 < k_1 + k_2 \\ k_i \ge A}}   
(k_1 + k_2 + k_3 + 1)^{\rho} 2^{-k_1 \alpha_1} 2^{-k_2 \alpha_2} 2^{(k_3 - k_2) \alpha_3} 
\\ \nonumber 
&\quad + \sum_{\substack{k_1, k_2, k_3 \ge 0 \\ k_3 < k_2 \\ k_i \ge A}}    
(k_1 + k_2 + k_3 + 1)^{\rho} 2^{-k_1 \alpha_1} 2^{-k_2 \alpha_2}.   
\end{align}
Note that for all $k_1, k_2, k_3 \ge 0$, 
\begin{align}\label{kkk-2}
(k_1 + k_2 + k_3 + 1)^{\rho} 
\le (k_1 + k_2 + 1)^{\rho} (k_3 + 1)^{\rho}
\le (k_1 + 1)^{\rho} (k_2 + 1)^{\rho} (k_3 + 1)^{\rho},
\end{align} 
and 
\begin{align}\label{kkk-3}
\sum_{k_3 \ge k_1 + k_2} \frac{(k_3+1)^{\rho} \, 2^{(k_3 - k_2) \alpha_3}}{2^{(k_3 - k_1 - k_2) \alpha_0}} 
\lesssim 2^{k_1 \alpha_3} (k_1+k_2+1)^{\rho}. 
\end{align}
Thus, by \eqref{kkk-1}--\eqref{kkk-3},
\begin{align*}
\mathscr{I}_1(A) 
&\lesssim \sum_{\substack{k_1 \ge A \\ k_2 \ge 0}}
(k_1 + k_2 + 1)^{2 \rho +1} 2^{-k_1 (\alpha_1 - \alpha_3)} 2^{-k_2 \alpha_2}
\\
&\lesssim \sum_{k_1 \ge A} (k_1 + 1)^{2 \rho + 1} 2^{-k_1 (\alpha_1 - \alpha_3)} 
\sum_{k_2 \ge 0} (k_2 + 1)^{2 \rho + 1} 2^{-k_2 \alpha_2} 
\\
& \lesssim (A+1)^{2\rho + 1} 2^{-A (\alpha_1 - \alpha_3)},  
\end{align*}
and    
\begin{align*}
\mathscr{I}_2(A) 
&\lesssim \sum_{\substack{k_1 \ge 0 \\ k_2 \ge A}}
(k_1 + k_2 + 1)^{2 \rho +1} 2^{-k_1 (\alpha_1 - \alpha_3)} 2^{-k_2 \alpha_2}
\\
&\lesssim \sum_{k_1 \ge 0} (k_1 + 1)^{2 \rho + 1} 2^{-k_1 (\alpha_1 - \alpha_3)} 
\sum_{k_2 \ge A} (k_2 + 1)^{2 \rho + 1} 2^{-k_2 \alpha_2} 
\\
& \lesssim (A+1)^{2\rho + 1} 2^{-A \alpha_2}.
\end{align*}

To bound $\mathscr{I}_3(A) $, we present two primary estimates:  
\begin{align}\label{kkk-5}
\sum_{k_3 \ge A} \frac{(k_3 + 1)^{\rho} \, 2^{(k_3 - k_2) \alpha_3}}{2^{(k_3 - k_1 - k_2) \alpha_0}} 
\lesssim (A+1)^{\rho} \, 2^{-A (\alpha_0 - \alpha_3)} \, 2^{k_1 \alpha_0 + k_2(\alpha_0 - \alpha_3)},
\end{align}
and
\begin{align}\label{kkk-6}
\mathscr{J}(A) 
& := \sum_{\substack{k_1, k_2 \ge 0 \\ k_1 + k_2 < A}}
(k_1 + k_2 + 1)^{\rho + 1} 2^{-k_1 \alpha_1} 2^{-k_2 \alpha_2}  
 2^{-A (\alpha_0 - \alpha_3)} 2^{k_1 \alpha_0 + k_2(\alpha_0 - \alpha_3)} 
\\ \nonumber
& \lesssim (A+1)^{2\rho+4} \, 2^{-A \min\{\alpha_0 - \alpha_3, \, \alpha_1 - \alpha_3, \, \alpha_2\}}. 
\end{align}
Indeed, it is enough to treat \eqref{kkk-6} in the case $A \ge 2$. If $\alpha_0 - \alpha_3 < \alpha_2$, then 
\begin{align*}
\mathscr{J}(A) 
&\le 2^{-A (\alpha_0 - \alpha_3)} \sum_{0 \le k_1 < A} 
(k_1 + 1)^{\rho + 1} 2^{-k_1 (\alpha_1 - \alpha_0)}
\sum_{0 \le k_2 < A - k_1} (k_2 + 1)^{\rho + 1} 2^{-k_2(\alpha_2 - \alpha_0 + \alpha_3)}
\\ 
& \lesssim 2^{-A (\alpha_0 - \alpha_3)} 
\sum_{0 \le k_1 < A} (k_1 + 1)^{\rho + 1} 2^{-k_1 (\alpha_1 - \alpha_0)}, 
\end{align*}
which further gives that in the case $\alpha_0 < \alpha_1$, $\mathscr{J}(A) \lesssim 2^{-A (\alpha_0 - \alpha_3)}$; and in the case $\alpha_1 \le \alpha_0$, 
\begin{align*}
\mathscr{J}(A) 
\lesssim 2^{-A (\alpha_0 - \alpha_3)} A (A+1)^{\rho+1}\, 2^{A(\alpha_0 - \alpha_1)} 
\le (A+1)^{\rho+2} \, 2^{-A (\alpha_1 - \alpha_3)}.
\end{align*} 
If $\alpha_2 \le \alpha_0 - \alpha_3$, then 
\begin{align*}
\mathscr{J}(A) 
&\le 2^{-A (\alpha_0 - \alpha_3)} \sum_{0 \le k_1 < A} 
(k_1 + 1)^{\rho + 1} 2^{-k_1 (\alpha_1 - \alpha_0)} 
A (A+1)^{\rho + 1}\, 2^{(A - k_1) (\alpha_0 - \alpha_3 - \alpha_2)}
\\
&\le (A + 1)^{\rho + 2} \, 2^{-A \alpha_2} 
\sum_{0 \le k_1 < A} (k_1 + 1)^{\rho + 1} 2^{-k_1 (\alpha_1 - \alpha_3 - \alpha_2)}, 
\end{align*}
which yields that in the case $\alpha_2 < \alpha_1 - \alpha_3$, $\mathscr{J}(A) \lesssim (A + 1)^{\rho + 2} \, 2^{-A \alpha_2}$; and in the case $\alpha_1 - \alpha_3 \le \alpha_2$, 
\begin{align*}
\mathscr{J}(A) 
\le (A+1)^{\rho+2}\, 2^{-A \alpha_2} A (A+1)^{\rho+1}\, 2^{A(\alpha_2 - \alpha_1 + \alpha_3)} 
\le (A+1)^{2\rho + 4} \, 2^{-A (\alpha_1 - \alpha_3)}.
\end{align*} 
Hence, it follows from \eqref{kkk-1}--\eqref{kkk-6} that 
\begin{align*}
\mathscr{I}_3(A) 
& \lesssim \sum_{\substack{k_1, k_2 \ge 0 \\ k_1 + k_2 < A}} 
(k_1 + k_2 + 1)^{\rho+1} 2^{-k_1 \alpha_1} 2^{-k_2 \alpha_2} 
(A+1)^{\rho} \, 2^{-A (\alpha_0 - \alpha_3)} \, 2^{k_1 \alpha_0 + k_2(\alpha_0 - \alpha_3)}
\\ 
&\qquad + \sum_{\substack{k_1, k_2 \ge 0 \\ k_1 + k_2 \ge A}}   
(k_1 + k_2 + 1)^{2\rho} \, 2^{-k_1 (\alpha_1 - \alpha_3)} 2^{-k_2 \alpha_2} 
\\ 
&\qquad + \sum_{\substack{k_1, k_2 \ge 0 \\ k_1 + k_2 \ge A}}   
(k_1 + k_2 + 1)^{\rho+1} \, 2^{-k_1 (\alpha_1 - \alpha_3)} 2^{-k_2 \alpha_2} 
\\
&\qquad+ \sum_{\substack{k_1 \ge 0 \\ k_2 \ge A}}
(k_1 + k_2 + 1)^{\rho + 1} \, 2^{-k_1 \alpha_1} 2^{-k_2 \alpha_2}  
\\
&\lesssim (A+1)^{3\rho + 4} \, 2^{-A \min\{\alpha_0 - \alpha_3, \alpha_1 - \alpha_3, \alpha_2\}}
+ \sum_{\substack{k_1 \ge A/2 \\ k_2 \ge 0}}
(k_1 + k_2 + 1)^{2\rho} 2^{-k_1 (\alpha_1 - \alpha_3)} 2^{-k_2 \alpha_2} 
\\
&\qquad+ \sum_{\substack{k_1 \ge 0 \\ k_2 \ge A/2}}
(k_1 + k_2 + 1)^{2\rho} \, 2^{-k_1 (\alpha_1 - \alpha_3)} 2^{-k_2 \alpha_2}  
\\
&\lesssim (A+1)^{3\rho+4} \, 2^{-A \min\{\alpha_0 - \alpha_3, \, \alpha_1 - \alpha_3, \, \alpha_2\}} 
+ (A+1)^{2\rho} 2^{-A \min\{\alpha_1 - \alpha_3, \, \alpha_2\}/2}
\\
&\lesssim (A+1)^{3\rho+4} \, 2^{-A \min\{\alpha_0 - \alpha_3, \, \alpha_1 - \alpha_3, \, \alpha_2\}/2}. 
\end{align*}
This completes the proof. 
\end{proof}

Let us recall a technical lemma from \cite[Lemma 2.3]{CYY}.

\begin{lemma}\label{lem:improve}
Given $(F_{i, 1}, F_{i, 2}, F_{i, 3}) \in \F$, denote 
\begin{align*}
F_i(x_i, y_i) := F_{i, 1}(|x_i - y_i|) F_{i, 2}(|x_i - y_i|) F_{i, 3}(|x_i + y_i|),  \quad i=1, 2, 3. 
\end{align*}
\begin{list}{\rm (\theenumi)}{\usecounter{enumi}\leftmargin=1.2cm \labelwidth=1cm \itemsep=0.2cm \topsep=0.2cm \renewcommand{\theenumi}{\roman{enumi}}} 

\item\label{size-imp} For each $i=1, 2, 3$, there exists a triple $(F'_{i, 1}, F'_{i, 2}, F'_{i, 3}) \in \F$ such that $F'_{i, 1}$ is monotone increasing, $F'_{i, 2}$ and $F'_{i, 3}$ are monotone decreasing, and 
\begin{align*}
F_i(x_i, y_i) 
\le F'_i(x_i, y_i) 
:= F'_{i, 1}(|x_i-y_i|) F'_{i, 2}(|x_i-y_i|) F'_{i, 3} \bigg(1 + \frac{|x_i+y_i|}{1+|x_i-y_i|}\bigg). 
\end{align*}

\item\label{Holder-imp}
Assume that $F_{1, 1}$, $F_{2, 1}$, and $F_{3, 1}$ are monotone increasing. Then for each $i=1, 2, 3$, there exist $\delta'_i \in (0, \delta_i)$ and $(F'_{i, 1}, F'_{i, 2}, F'_{i, 3}) \in \F$ such that $F'_{i, 1}$ is monotone increasing, $F'_{i, 2}$ and $F'_{i, 3}$ are monotone decreasing, and 
\begin{align*}
F_i(x_i, y_i) \frac{|y_i-y'_i|^{\delta_i}}{|x_i-y_i|^{1 + \delta_i}} 
\le F'_i(x_i, y_i) \frac{|y_i-y'_i|^{\delta'_i}}{|x_i-y_i|^{1 + \delta'_i}}, 
\end{align*}
whenever $|y_i-y'_i| \leq |x_i-y_i|/2$, where 
\begin{align*}
F'_i(x_i, y_i) 
:= F'_{i, 1}(|y_i-y'_i|) F'_{i, 2}(|x_i-y_i|) F'_{i, 3} \bigg(1 + \frac{|x_i+y_i|}{1+|x_i-y_i|}\bigg). 
\end{align*}
\end{list}
\end{lemma}

The following result is just a particular case of \cite[Lemma 2.12]{CYY}. 
 
\begin{lemma}\label{lem:diag}
Let $r>1$ and $1<s<2$. Assume that $F$ is monotone decreasing. Then for all intervals $I, J \subset \R$ with $\ell(I) = \ell(J)$, $I \cap J = \emptyset$, and $\d(I, J)=0$, 
\begin{align*}
\bigg(\fint_I \fint_J \frac{1}{|x - y|^s} \, dx \, dy  \bigg)^{\frac1s} 
\lesssim |I|^{-1}, 
\end{align*}
and 
\begin{align*}
\bigg(\fint_I \fint_J F(|x - y|)^r \, dx \, dy \bigg)^{\frac1r} 
\lesssim \widetilde{F}(\ell(I)),  
\end{align*}
where $\widetilde{F}(t) := \sum_{k=0}^{\infty} 2^{-k/r} F(2^{-k} t)$. 
\end{lemma}

Next, let us present a useful estimate, which will be used frequently below.  

\begin{lemma}\label{lem:IJD} 
Let $F$ be a monotone decreasing function such that $\lim\limits_{t \to \infty} F(t)=0$. Then for all dyadic intervals $I, J \subset \R$ with $\ell(I) = \ell(J)$ and $\rd(I, J) = 1$ and for any $t>0$, if we define 
\begin{align*}
\mathscr{I}_{\theta}^t(I, J) 
:= \fint_I \fint_J \bigg(t |x-y| + \frac{1}{t|x-y|}\bigg)^{-\theta} \frac{F(|x-y|)}{|x-y|} \, dx \, dy,  
\end{align*}
then the following hold: 
\begin{enumerate} 
\item\label{TT-1} If $\theta=0$ and $I \cap J = \emptyset$, then
\begin{align*}
\mathscr{I}_{\theta}^t(I, J) 
\lesssim |I|^{-1} \widetilde{F}(\ell(I)). 
\end{align*}

\item\label{TT-2} If $\theta \in (0, 1)$, then 
\begin{align*}
\mathscr{I}_{\theta}^t(I, J) 
\lesssim |I|^{-1} \big(t|I| + t^{-1} |I|^{-1} \big)^{-\theta} \widetilde{F}(\ell(I)),  
\end{align*}
where $\widetilde{F}(t) := \sum_{k=0}^{\infty} 2^{-k \eta} F(2^{-k} t)$ for some small number $\eta \in (0, 1)$. 
\end{enumerate} 
\end{lemma}

\begin{proof}
Since the proof of part \eqref{TT-1} is contained in that of part \eqref{TT-2}, we only show the latter. Let $\theta \in (0, 1)$,  $I, J \subset \R$ be dyadic intervals with $\ell(I) = \ell(J)$ and $\rd(I, J) = 1$. We begin with the case $I \cap J = \emptyset$. Note that $\d(I, J)=0$. 

If $t|I| < 1$, we take $r>1$ so that $r'(1-\theta)<2$, and use H\"{o}lder's inequality to obtain 
\begin{align*}
\mathscr{I}_{\theta}^t(I, J) 
&\le t^{\theta} \fint_I \fint_J \frac{F(|x-y|)}{|x-y|^{1-\theta}} dx \, dy
\\
&\le t^{\theta} \bigg(\fint_I \fint_J F(|x-y|)^r dx \, dy \bigg)^{\frac1r}
\bigg(\fint_I \fint_J \frac{dx \, dy}{|x-y|^{r'(1-\theta)}} \bigg)^{\frac{1}{r'}}
\\
&\lesssim t^{\theta} \widetilde{F}(\ell(I)) |I|^{-1+\theta}
\simeq |I|^{-1} \widetilde{F}(\ell(I)) (t|I| + t^{-1} |I|^{-1})^{-\theta}. 
\end{align*}
where Lemma \ref{lem:diag} was used in the second-to-last step. Similarly, if $t|I| \ge 1$, picking $r>1$ so that $r'(1+\theta)<2$, we apply H\"{o}lder's inequality to arrive at 
\begin{align*}
\mathscr{I}_{\theta}^t(I, J) 
&\le t^{-\theta} \fint_I \fint_J \frac{F(|x-y|)}{|x-y|^{1+\theta}} dx \, dy
\\
&\le t^{-\theta} \bigg(\fint_I \fint_J F(|x-y|)^r dx \, dy \bigg)^{\frac1r}
\bigg(\fint_I \fint_J \frac{dx \, dy}{|x-y|^{r'(1+\theta)}} \bigg)^{\frac{1}{r'}}
\\
&\lesssim t^{-\theta} \widetilde{F}(\ell(I)) |I|^{-1-\theta}
\simeq |I|^{-1} \widetilde{F}(\ell(I)) (t|I| + t^{-1} |I|^{-1})^{-\theta}. 
\end{align*}

Now assume $I \cap J \not=\emptyset$, which implies $I=J$. Let $I_t :=\{tx: x \in I\}$. By a change of variables, we are reduced to showing  
\begin{align}\label{IJD-1}
t \fint_{I_t} \fint_{I_t} \frac{F(t^{-1} |x-y|)}{|x-y|^{1-\theta}} dx \, dy 
\lesssim \widetilde{F}(|I|) |I|^{-1} (t|I| + t^{-1} |I|^{-1})^{-\theta}, 
\end{align}
since $\mathscr{I}^t_{\theta}(I, I)$ is controlled by the left hand side of \eqref{IJD-1}. Note that the inequality \eqref{IJD-1} is invariant under the translation and dilation. Hence, we may assume $I_t =[0, 1]$ and it suffices to prove 
\begin{align}\label{IJD-2}
\int_0^1 \int_0^1 \frac{F(s |x-y|)}{|x-y|^{1-\theta}} dx \, dy 
\lesssim \widetilde{F}(s), \quad \forall \, s>0. 
\end{align}
Indeed, letting $x \in (0, 1)$, it follows from the monotonicity of $F$ that 
\begin{align*}
\int_0^1 \frac{F(s |x-y|)}{|x-y|^{1-\theta}} \, dy 
&= \int_0^x \frac{F(s |x-y|)}{|x-y|^{1-\theta}} \, dy 
+ \int_x^1 \frac{F(s |x-y|)}{|x-y|^{1-\theta}} \, dy 
\\
&\le \int_0^{x-2^{-k_0+1}} + \sum_{k \ge k_0} \int_{x-2^{-k+1}}^{x-2^{-k}} 
+ \sum_{k \ge k_1} \int_{x+2^{-k}}^{x+2^{-k+1}} 
+ \int_{x+2^{-k_1+1}}^1
\\
&\lesssim \sum_{k=0}^{\infty} F(2^{-k} s) \big[2^{(-k+1) \theta} - 2^{-k \theta} \big]
\lesssim \sum_{k=0}^{\infty} 2^{-k \theta} F(2^{-k} s) 
=: \widetilde{F}(s),  
\end{align*}
where $k_0 := [\log_2 \frac{2}{x}] +1$ if $\log_2 \frac{2}{x} \not\in \N$, otherwise, $k_0 := \log_2 \frac{2}{x}$, $k_1 := [\log_2 \frac{2}{1-x}] +1$ if $\log_2 \frac{2}{1-x} \not\in \N$, otherwise, $k_0 := \log_2 \frac{2}{x}$, and the implicit constants are independent of $x$. This immediately implies \eqref{IJD-2}. It is trivial that $\widetilde{F}$ is bounded and decreasing, and $\lim\limits_{t \to \infty} \widetilde{F}(t)=0$. The proof is complete. 
\end{proof}

\subsection{Functions acting on cubes}

For each $i \in \{1, 2, 3\}$, let $F_i \in \F_0$ and $(F_{i, 1}, F_{i, 2}, F_{i, 3}) \in \mathscr{F}$ so that $F_{i, 1}$ is increasing, $F_{i, 2}$ and $F_{i, 3}$ are decreasing. Define  
\begin{align}
\label{def:FIK-1} F_i(I_i, J_i)
& := F_{i, 1}(\ell(I_i \vee J_i)) 
\widetilde{F}_{i, 2}(\ell(I_i \wedge J_i)) 
\widetilde{F}_{i, 3}(I_i \cup J_i), 
\\
\label{def:FIK-2} \widehat{F}_i(I_i, J_i)  
& := F_i(I_i, J_i) + \sum_{I'_i \in \ch(I_i)} F_i(I'_i), 
\end{align}
for all intervals $I_i, J_i \subset \R$, where 
\begin{align*}
I \vee J & := 
\begin{cases}
I, & \ell(I) \ge \ell(J), 
\\
J, & \ell(I) < \ell(J), 
\end{cases}
\qquad\quad \, \, 
I \wedge J := 
\begin{cases}
J, & \ell(I) \ge \ell(J), 
\\
I, & \ell(I) < \ell(J), 
\end{cases}
\\ 
\widetilde{F}_{i, 2}(t) 
& := \sum_{k=0}^{\infty} 2^{-k \eta} F_{i, 2}(2^{-k} t), 
\qquad
\widetilde{F}_{i, 3}(I_i)
:=\sum_{k=0}^{\infty} 2^{-k \eta} F_{i, 3}(\rd(2^k I_i, \I)). 
\end{align*}
The auxiliary parameter $\eta \in (0, 1)$ is harmless and small enough. 

The notation in \eqref{def:FIK-1}--\eqref{def:FIK-2} will be used throughout this paper. For any harmless constant $\lambda \in (0, \infty)$, we always suppress the presence of $\lambda$ in  $F_{i, j}(\lambda t)$ and $F_i(\lambda I_i)$ because they enjoy the same properties as $F_{i, j}(t)$ and $F_i(I_i)$ respectively.

\begin{lemma}\label{lem:FF}
For each $i=1, 2, 3$, let $I_i, J_i \subset \R$ be intervals. Then the following  hold: 
\begin{list}{\rm (\theenumi)}{\usecounter{enumi}\leftmargin=1.2cm \labelwidth=1cm \itemsep=0.2cm \topsep=0.2cm \renewcommand{\theenumi}{\alph{enumi}}} 

\item\label{list-FF01} If $I_i \subset J_i$, then $F_i(I_i, I_i) \le F_i(I_i, J_i)$ and $\max\{F_i(I_i, J_i), \widehat{F}_i(I_i, I_i) \} \le  \widehat{F}_i(I_i, J_i)$.

\item\label{list-FF02} $F_i(I_i, J_i) \simeq F_i(I_i, I_i) \simeq F_i(J_i, J_i)$ whenever $\d(I_i, J_i) \lesssim 1$ and $\ell(I_i) \simeq \ell(J_i)$.

\item\label{list-FF2} For any $\varepsilon>0$, there exists $N_0 \ge 1$ so that for all $N \ge N_0$, $I_i \in \D$, and $J_i \not\in \D(2N)$ with $\rd(I_i, J_i) \le 2$, there holds $F_i(I_i, J_i) < \varepsilon$ or $\ell(I_i)/\ell(J_i) \le N^{-\frac12}$ or $\ell(J_i)/\ell(I_i) \le N^{-\frac12}$.

\black
\item\label{list-FF3} For any $k_i \ge 0$, one has 
\begin{align*}
\lim_{N \to \infty} \sup_{\D} \sup_{K_i \not\in \D(N)} \mathbf{F}_i (K_i) = 0, 
\end{align*}
where 
\begin{align*}
\mathbf{F}_i(K_i)
:= F_{i, 1}(\ell(K_i)) \widetilde{F}_{i, 2}(2^{-k_i} \ell(K_i)) \widetilde{F}_{i, 3}(K_i) 
+ \sum_{I_i^{(k_i+1)} = K_i} F_i(I_i). 
\end{align*}
\end{list}
\end{lemma}

\begin{proof}
By the monotonicity of $F_{i, 3}$, we see that for all intervals $I \subset J$, 
\begin{align}\label{Fi31}
\rd(I, \I) \ge \rd(J, \I), \quad 
F_{i, 3}(\rd(I, \I)) \le F_{i, 3}(\rd(J, \I)), \quad\text{and}\quad 
\widetilde{F}_{i, 3}(I) \le \widetilde{F}_{i, 3}(J). 
\end{align} 
Hence, $F_i(I_i, I_i) \le F_i(I_i, J_i)$, which implies $\max\{F_i(I_i, J_i), \widehat{F}_i(I_i, I_i) \} \le  \widehat{F}_i(I_i, J_i)$. This shows part \eqref{list-FF01}.  

To show part \eqref{list-FF02}, let $\d(I_i, J_i) \lesssim 1$ and $\ell(I_i) \simeq \ell(J_i)$. Then there exists some constant $\lambda \ge 1$ such that $I_i \subset \lambda J_i$ and $J_i \subset \lambda I_i$. Note that for any interval $I \subset \R$, 
\begin{align}\label{Fi32}
\widetilde{F}_{i, 3}(\lambda I)
\le (2\lambda)^{\eta} \, \widetilde{F}_{i, 3}(I), \qquad\forall \lambda \ge 1. 
\end{align}
Then it follows from \eqref{Fi31} and \eqref{Fi32} that 
\begin{align*}
\widetilde{F}_{i, 3}(I_i \cup J_i) 
\le \widetilde{F}_{i, 3}(\lambda J_i) 
\lesssim \widetilde{F}_{i, 3}(J_i) 
\le \widetilde{F}_{i, 3}(\lambda I_i) 
\lesssim \widetilde{F}_{i, 3}(I_i) 
\le \widetilde{F}_{i, 3}(I_i \cup J_i), 
\end{align*} 
which together with part \eqref{list-FF01} implies 
\begin{align*}
& F_i(I_i, I_i) 
\le F_i(I_i, \lambda J_i) 
\lesssim F_i(I_i, J_i) 
\\
& \le F_{i, 1}(\lambda \ell(I_i)) 
\widetilde{F}_{i, 2}(\lambda^{-1} \ell(I_i)) 
\widetilde{F}_{i, 3}(I_i \cup J_i) 
\\
&\lesssim F_{i, 1}(\ell(I_i)) 
\widetilde{F}_{i, 2}(\ell(I_i)) 
\widetilde{F}_{i, 3}(I_i) 
= F_i(I_i, I_i). 
\end{align*} 
Similarly, $F_i(I_i, J_i) \simeq F_i(J_i, J_i)$. 

To prove part \eqref{list-FF2}, fix $I_i \in \D$ and $J_i \not\in \D(2N)$ with $\rd(I_i, J_i) \le 1$. We begin with the case $\ell(J_i) < 2^{-2N}$. If $\ell(I_i \vee J_i) \le 2^{-N}$, we have $F_i(I_i, J_i) \lesssim F_{i, 1} (2^{-N}) \to 0$ as $N \to \infty$. If $\ell(I_i \vee J_i) > 2^{-N}$, then $\ell(I_i) > 2^{-N}$, and hence, $\ell(J_i)/\ell(I_i) \le 2^{-N} \le N^{-\frac12}$ provided $N$ sufficiently large. 

We then treat the case $\ell(J_i)>2^{2N}$. If $\ell(I_i \wedge J_i) \ge 2^N$, we have $F_i(I_i, J_i) \lesssim \widetilde{F}_{i, 2}(2^N) \to 0$ as $N \to \infty$. If $\ell(I_i \wedge J_i) < 2^N$, then $\ell(I_i) < 2^N$, and hence, $\ell(I_i)/\ell(J_i) \le 2^{-N} \le N^{-\frac12}$ provided $N$ large enough.  

Now let $2^{-2N} \le \ell(J_i) \le 2^{2N}$ and $\rd(J_i, 2^{2N} \I) > 2N$, then $\d(J_i, \I) \ge \d(J_i, 2^{2N} \I) > 2N 2^{2N}$. Consider first that $\ell(I_i) > \ell(J_i)$. If $\ell(I_i) > N^{\frac12} \ell(J_i)$, then $\ell(J_i)/\ell(I_i) \le N^{-\frac12}$. If $\ell(I_i) \le N^{\frac12} \ell(J_i)$, then $I_i \subset 5 N^{\frac12} J_i$. Hence, for any $k \ge 0$,
\begin{align}\label{kcd-1}
\d(2^k \cdot 5 N^{\frac12} J_i, \I)
\ge \d(J_i, \I) - 2^k \cdot 5 N^{\frac12} \ell(J_i)/2
> 2N 2^{2N} - 2^k \cdot 5  N^{\frac12} \cdot 2^{2N}/2,
\end{align}
which in turn implies that
\begin{align}\label{kcd-2}
\rd(2^k(I_i \cup J_i), \I)
\ge \rd(2^k \cdot 5 N^{\frac12} J_i, \I)
\ge \frac{\d(2^k \cdot 5 N^{\frac12} J_i, \I)}{2^k \cdot 5 N^{\frac12} \cdot 2^{2N}}
>\frac25 2^{-k} N^{\frac12} - \frac12.
\end{align}
Let $N \ge 81$ and $k_N := [\log_2 N^{\frac14}]$. Thus, for all $0 \le k \le k_N$,
\begin{align}\label{kcd-3}
\rd(2^k(I_i \cup J_i), \I) 
\ge \frac25 N^{\frac14} - \frac12 
\ge \frac15 N^{\frac14}, 
\end{align}
which along with the monotonicity of $F_{i, 3}$ yields  
\begin{align}\label{kcd-4}
F_i(I_i, J_i)
\lesssim \sum_{k=0}^{k_N} 2^{-k \delta_i} F_{i, 3}(N^{\frac14}/5)
+ \sum_{k=k_N+1}^{\infty} 2^{-k \delta_i}
\lesssim F_{i, 3}(N^{\frac14}) + 2^{-k_N \delta_i} \to 0,
\end{align}
as $N \to \infty$. Furthermore, in the case $\ell(I_i) < \ell(J_i)$, there holds $I_i \subset 5J_i \subset 5N^{\frac12} J_i$. This means that \eqref{kcd-1}--\eqref{kcd-4} still hold in this scenario. Consequently, we conclude part \eqref{list-FF2}.

Finally, to obtain part \eqref{list-FF3}, let $\varepsilon>0$ be an arbitrary number. Fix an integer $k_i \ge 0$. Since $F_i \in \F_0$ and 
\begin{align*}
\lim_{t \to 0} F_{i, 1}(t)
= \lim_{t \to \infty} F_{i, 2}(t)
= \lim_{t \to \infty} F_{i, 3}(t)
= \lim_{N \to \infty} \sum_{k>N} 2^{-k \eta}
= 0,  
\end{align*}
there exists $N_0 \ge 1$ such that 
\begin{align}
\label{FT-1} & 0 \le F_{i, 1}(t) < \varepsilon, \quad\forall t < 2^{- N_0}, 
\\
\label{FT-2} & 0 \le F_{i, 2}(t) < \varepsilon, \quad\forall t > 2^{N_0}, 
\\
\label{FT-3} & 0 \le F_{i, 3}(t) < \varepsilon,  \quad\forall t > N_0, 
\\
\label{FT-4} & \sum_{k>N} 2^{-k \eta} < \varepsilon, \, \, \quad\forall N > N_0, 
\\
\label{FT-5} & F(I_i) < \varepsilon, \, \quad\forall I_i: \ell(I_i) < 2^{-N_0}, 
\\
\label{FT-6} & F(I_i) < \varepsilon, \, \quad\forall I_i: \ell(I_i) > 2^{N_0}, 
\\
\label{FT-7} & F(I_i) < \varepsilon, \, \quad\forall I_i: |c_{I_i}| > 2^{N_0} N_0. 
\end{align}
Let $N > 2^{N_0+2} (N_0 + k_i)$ and $K_i \not\in \D(N)$. If $\ell(K_i) < 2^{-N}$, then by \eqref{FT-1} and \eqref{FT-5}, 
\begin{align*}
\mathbf{F}_i(K_i) 
\lesssim F_{i, 1}(\ell(K_i)) + \sum_{I_i^{(k_i+1)} = K_i} F_i(I_i)
\lesssim \varepsilon.
\end{align*}
If $\ell(K_i) > 2^N$, then the monotonicity of $F_{i, 2}$, \eqref{FT-2}, \eqref{FT-4}, and \eqref{FT-6} imply 
\begin{align*}
\mathbf{F}_i(K_i) 
\lesssim \sum_{k \le N_0} 2^{-k \eta} F_{i, 2}(2^{-k_i - N_0 + N})  
+ \sum_{k > N_0} 2^{-k \eta} \|F_{i, 2}\|_{L^{\infty}} 
+ \sum_{I_i^{(k_i+1)} = K_i} F_i(I_i)
\lesssim \varepsilon, 
\end{align*}
provided that $\ell(I_i) = 2^{-k_i} \ell(K_i) \ge 2^{-k_i+N} > 2^{N_0}$. If $2^{-N} \le \ell(K_i) \le 2^N$ and $\rd(K_i, 2^N \I)>N$, then $|c_{K_i}| \ge 2^N N \ge 4N_0 2^{N_0+N} \ge 4N_0 \ell(2^{N_0} K_i)$, which implies that $2^{N_0} K_i \cap \I = \emptyset$ and 
\begin{align*}
\d(2^{N_0} K_i, \I)
&=|c_{2^{N_0} K_i}| - \ell(2^{N_0} K_i)/2 - \ell(\I)/2
\ge |c_{K_i}| - 2^{N_0} \ell(K_i) - \ell(\I)
\\
&\ge 2^N N  - 2^{N_0 + N} - 1 
\ge 2^{N_0 + N} N_0.
\end{align*}
The latter in turn gives that for every $k \le N_0$, 
\begin{align}\label{FT-8}
\rd(2^k K_i, \I)
\ge \rd(2^{N_0} K_i, \I)
= 1+ \frac{\d(2^{N_0} K_i, \I)}{\max\{2^{N_0} \ell(K_i), 1\}} 
> N_0. 
\end{align}
Since $F_3$ is monotone decreasing,  the estimates \eqref{FT-3}, \eqref{FT-4}, and \eqref{FT-7} yield 
\begin{align*}
\mathbf{F}_i(K_i) 
\lesssim \sum_{0 \le k \le N_0} 2^{-k \eta} F_{i, 3}(N_0)
+ \sum_{k > N_0} 2^{-k \eta} \|F_{i, 3}\|_{L^{\infty}}
+ \sum_{I_i^{(k_i+1)} = K_i} F_i(I_i)
\lesssim \varepsilon, 
\end{align*}
provided that $|c_{I_i}| \ge |c_{K_i}| - \ell(K_i)/2 \ge 2^N N - 2^N/2 \ge 2^N N/2 > 2^{N_0} N_0$. This completes the proof. 
\end{proof}

\subsection{Some integral estimates}
The following lemmas naturally arises from the proof of the compact dyadic representation in Section \ref{sec:cdr}. 

\begin{lemma}\label{lem:PP} 
For each $i \in \{1, 2, 3\}$, let $(F_{i, 1}, F_{i, 2}, F_{i, 3}) \in \mathscr{F}$ be such that $F_{i, 1}$ is increasing, $F_{i, 2}$ and $F_{i, 3}$ are decreasing. Let $I_i, J_i, K_i \subset \R$ be dyadic intervals satisfying $\ell(I_i) = \ell(J_i)$ and $I_i \cup J_i \subset K_i$, $i=1, 2, 3$. Then the following statements hold: 
\begin{enumerate}

\item[(i)] If $\rd(I_i, J_i) > 1$, then 
\begin{align}\label{def:PP-1}
\mathscr{P}_i(I_i, J_i)
:= \int_{I_i} \int_{J_i} \frac{F_i(x_i, y_i)}{|x_i - y_i|}  dx_i \, dy_i 
\lesssim \frac{F_i(I_i, K_i)}{\rd(I_i, J_i)} |I_i|,  
\end{align}
where 
\begin{align*}
F_i(x_i, y_i)
:= F_{i, 1}(|x_i - c_{J_i}|) F_{i, 2}(|x_i - y_i|) F_{i, 3}\bigg(1 + \frac{|x_i + y_i|}{1 + |x_i - y_i|}\bigg). 
\end{align*}

\item[(ii)] If $\rd(I_i, J_i) \le 1$ and $I_i \cap J_i = \emptyset$, then 
\begin{align}\label{def:PP-2}
\mathscr{P}_i(I_i, J_i)
:= \int_{I_i} \int_{J_i} \frac{F_i(x_i, y_i)}{|x_i - y_i|}  dx_i \, dy_i 
\lesssim F_i(I_i, I_i) |I_i|, 
\end{align}
where 
\begin{align*}
F_i(x_i, y_i)
&:= F_{i, 1}(|x_i - y_i|) F_{i, 2}(|x_i - y_i|) F_{i, 3}\bigg(1+\frac{|x_i+y_i|}{1+|x_i-y_i|}\bigg). 
\end{align*}

\item[(iii)] If $\rd(I_i, J_i) \le 1$, then for any $t>0$ and $\theta \in (0, 1)$, 
\begin{align}\label{def:QQ}
\mathscr{Q}_i^t(I_i, J_i) 
&:= \int_{I_i} \int_{J_i} \bigg(t |x_i - y_i| + \frac{1}{t|x_i - y_i|}\bigg)^{-\theta} 
\frac{F_i(x_i, y_i)}{|x_i - y_i|} \, dx_i \, dy_i 
\\ \nonumber
&\lesssim |I_i| \big(t |I_i| + t^{-1} |I_i|^{-1} \big)^{-\theta} F_i(I_i, I_i), 
\end{align}
where the implicit constant is independent of $t$, $I_i$, and $J_i$.
\end{enumerate}

\end{lemma}

\begin{proof}
Suppose that $\rd(I_i, J_i) > 1$. Then for all $x_i \in J_i$ and $y_i \in I_i$, it is easy to see that 
\begin{align}\label{eq:PP-1}
\ell(I_i) = \ell(J_i) \le \d(I_i, J_i)
\le |x_i - y_i| \le \ell(K_i), \qquad 
(x_i + y_i)/2 \in K_i, 
\end{align}
and
\begin{align*}
1 + \frac{|x_i + y_i|}{1 + |x_i - y_i|}
\ge 1 + \frac{2\d(K_i, \I)}{1+ \ell(K_i)}
\ge 1 + \frac{\d(K_i, \I)}{\max\{\ell(K_i), 1\}}
= \rd(K_i, \I), 
\end{align*}
which gives 
\begin{align}\label{eq:PP-2}
F_i(x_i, y_i)
\le F_{i, 1}(\ell(I_i)) F_{i, 2}(\ell(I_i)) F_{i, 3}(\rd(K_i, \I))
\le F_i(I_i, K_i).
\end{align}
Hence, it follows from \eqref{eq:PP-1} and \eqref{eq:PP-2} that 
\begin{align*}
\mathscr{P}_i(I_i, J_i)
\lesssim \frac{F_i(I_i, K_i)}{\d(I_i, J_i)} |I_i| |J_i| 
= \frac{F_i(I_i, K_i)}{\rd(I_i, J_i)} |I_i|,   
\end{align*}
which shows \eqref{def:PP-1}.

To prove \eqref{def:PP-2}, let $\rd(I_i, J_i) \le 1$, $x_i \in J_i$, and $y_i \in I_i$. Then, $J_i \in 3I_i$, $|x_i - y_i| \le 3 \ell(I_i)$, and
\begin{align*}
& 2 \bigg(1 + \frac{|x_i + y_i|}{1 + |x_i - y_i|} \bigg)
\ge 1 + \frac{|x_i + y_i| + |x_i - y_i|}{1 + |x_i - y_i|}
\\
&\ge 1 + \frac{2 |y_i|}{1 + 3 \ell(I_i)}
\ge 1 + \frac{\d(I_i, \I)}{2\max\{\ell(I_i), 1\}}
\ge \frac12 \rd(I_i, \I),
\end{align*}
which implies
\begin{align}\label{eq:PP-3}
F_i(x_i, y_i)
\le F_{i, 1}(\ell(I_i)) F_{i, 2}(|x_i - y_i|)  F_{i, 3}(\rd(I_i, \I)).
\end{align}
Thus, by Lemma \ref{lem:IJD} part \eqref{TT-1} and \eqref{eq:PP-3} 
\begin{align*}
\mathscr{P}_i(I_i, J_i)
&\lesssim F_{i, 1}(\ell(I_i)) F_{i, 3}(\rd(I_i, \I))
\int_{I_i} \int_{J_i} \frac{F_{i, 2}(|x_i - y_i|)}{|x_i - y_i|}  dx_i \, dy_i
\\
&\lesssim F_{i, 1}(\ell(I_i)) \widetilde{F}_{i, 2}(\ell(I_i)) F_{i, 3}(\rd(I_i, \I)) |I_i|
\le  F_i(I_i, I_i) |I_i|.
\end{align*}
Similarly, we use Lemma \ref{lem:IJD} part \eqref{TT-2} and \eqref{eq:PP-3} to conclude \eqref{def:QQ} as desired. 
\end{proof}

\begin{lemma}\label{lem:RR} 
Let $\theta \in (0, 1)$ and $t>0$. Let $i \in \{1, 2, 3\}$, $(F_{i, 1}, F_{i, 2}, F_{i, 3}) \in \mathscr{F}$ be such  that $F_{i, 1}$ is increasing, $F_{i, 2}$ and $F_{i, 3}$ are decreasing. Let $I_i, J_i \subset \R$ be dyadic intervals with $\ell(I_i) = \ell(J_i)$, $i=1, 2, 3$. Then the following hold:
\begin{enumerate} 
\item[(i)] If $\rd(I_1, J_1) \le 1$ and $\rd(I_2, J_2) \le 1$, then 
\begin{align}\label{def:R12}
\mathscr{R}_{1, 2}^t(I_{1, 2}, J_{1, 2}) 
&:= \int_{I_{1, 2}} \int_{J_{1, 2}} D_{\theta}(x_{1, 2}-y_{1, 2}, t)
\prod_{i = 1, 2} \frac{F_i(x_i, y_i)}{|x_i - y_i|} \, dx_{1, 2} \, dy_{1, 2}
\\ \nonumber
&\lesssim F_1(I_1, I_1) |I_1| \, F_2(I_2, I_2) |I_2| \, 
\bigg(\frac{|I_1| |I_2|}{t} + \frac{t}{|I_1| |I_2|} \bigg)^{-\theta}. 
\end{align}

\item[(ii)] If $\rd(I_1, J_1) \le 1$ and $\rd(I_3, J_3) \le 1$, then 
\begin{align}\label{def:R13}
\mathscr{R}_{1, 3}^t(I_{1, 3}, J_{1, 3}) 
&:= \int_{I_{1, 3}} \int_{J_{1, 3}} D_{\theta}(x_1 - y_1, t, x_3 - y_3)
\prod_{i = 1, 3} \frac{F_i(x_i, y_i)}{|x_i - y_i|} \, dx_{1, 3} \, dy_{1, 3}
\\ \nonumber
&\lesssim F_1(I_1, I_1) |I_1| \, F_3(I_3, I_3) |I_3| 
\bigg(\frac{t |I_1|}{|I_3|} + \frac{|I_3|}{t |I_1|} \bigg)^{-\theta}. 
\end{align}

\item[(iii)] If $\rd(I_2, J_2) \le 1$ and $\rd(I_3, J_3) \le 1$, then 
\begin{align}\label{def:R23}
\mathscr{R}_{2, 3}^t(I_{2, 3}, J_{2, 3}) 
&:= \int_{I_{2, 3}} \int_{J_{2, 3}} D_{\theta}(t, x_{2, 3}-y_{2, 3})
\prod_{i = 2, 3} \frac{F_i(x_i, y_i)}{|x_i - y_i|} \, dx_{2, 3} \, dy_{2, 3}
\\ \nonumber
&\lesssim F_2(I_2, I_2) |I_2| \, F_3(I_3, I_3) |I_3| 
\bigg(\frac{t |I_2|}{|I_3|} + \frac{|I_3|}{t |I_2|} \bigg)^{-\theta}.
\end{align}

\item[(iv)] If $\rd(I_i, J_i) \le 1$ for each $i=1, 2, 3$, then 
\begin{align}\label{def:RIJ}
\mathscr{R}(I, J)
:= \int_I \int_J D_{\theta}(x-y) 
\prod_{i=1}^3 \frac{F_i(x_i, y_i)}{|x_i - y_i|} \, dx \, dy
\lesssim \prod_{i=1}^3 F_i(I_i, I_i) |I_i|. 
\end{align}
\end{enumerate}
\end{lemma}

\begin{proof}
To show \eqref{def:R12}, fix $x_1 \in J_1$ and $y_1 \in I_1$. Writing $t_2 := \frac{|x_1-y_1|}{t}$ and using \eqref{def:QQ}, we obtain 
\begin{align*}
\mathscr{Q}_2^{t_2}(I_2, J_2) 
&\lesssim |I_2| (t_2 |I_2| + t_2^{-1} |I_2|^{-1})^{-\theta} F_2(I_2, I_2)
\\
&= \bigg(\frac{|I_2|}{t} |x_1 - y_1| + \frac{t}{|I_2|} \frac{1}{|x_1 - y_1|}\bigg)^{-\theta} 
F_2(I_2, I_2) |I_2|,  
\end{align*}
which together with \eqref{def:QQ} applied to $t_1 := \frac{|I_2|}{t}$ implies 
\begin{align*}
\mathscr{R}_{1, 2}^t(I_{1, 2}, J_{1, 2}) 
&=\int_{I_1} \int_{J_1} \mathscr{Q}_2^{t_2}(I_2, J_2) \frac{F_1(x_1, y_1)}{|x_1 - y_1|} dx_1 \, dy_1 
\\
&\lesssim \mathscr{Q}_1^{t_1}(I_1, J_1) F_2(I_2, I_2) |I_2| 
\\
&\lesssim F_1(I_1, I_1) |I_1| F_2(I_2, I_2) |I_2| 
\bigg(\frac{|I_1| |I_2|}{t} + \frac{t}{|I_1| |I_2|}\bigg)^{-\theta}.
\end{align*}
This justifies \eqref{def:R12}. Much as in the same way, one can prove \eqref{def:R13} and \eqref{def:R23}.  

Observe that if $t = |x_3 - y_3|$, then 
\begin{align*}
\mathscr{R}(I, J)
=\int_{I_3} \int_{J_3} \mathscr{R}_{1, 2}^t(I_{1, 2}, J_{1, 2}) \frac{F_3(x_3, y_3)}{|x_3-y_3|} \, dx_3 \, dy_3. 
\end{align*}
Setting $t_3=\frac{1}{|I_1| |I_2|}$, it follows from \eqref{def:R12} and \eqref{def:QQ} that 
\begin{align*}
\mathscr{R}(I, J)
&\lesssim \int_{I_3} \int_{J_3} \bigg[t_3 |x_3-y_3| + \frac{1}{t_3 |x_3-y_3|}\bigg]^{-\theta} 
\frac{F_3(x_3, y_3)}{|x_3-y_3|} \, dx_3 \, dy_3
\prod_{i=1}^2 F_i(I_i, I_i) |I_i| 
\\
&=\mathscr{Q}_3^{t_3}(I_3, J_3) \prod_{i=1}^2 F_i(I_i, I_i) |I_i|
\lesssim \prod_{i=1}^3 F_i(I_i, I_i) |I_i|, 
\end{align*}
which shows \eqref{def:RIJ}. 
\end{proof}

\section{A compact dyadic representation}\label{sec:cdr}
The goal of this section is to prove Theorem \ref{thm:repre}. To this end, we combine the ideas in \cite{CYY} and \cite{HLMV}.

\subsection{Dyadic reductions}  
Given $\w = (\w_1, \w_2, \w_3) \in (\{0, 1\}^{\ZZ})^3$ and $k = (k_1, k_2, k_3) \in \{2, 3, \ldots \}^3$, we define 
\begin{align*}
\D_{\w} & := \D_{\w_1} \times \D_{\w_2} \times \D_{\w_3}, \quad 
\D_{\w}^k := \D_{\w_1}^{k_1} \times \D_{\w_2}^{k_2} \times \D_{\w_3}^{k_3}, 
\\
\D_{\w, \Z} & := \big\{I=I_1 \times I_2 \times I_3 \in \D_{\w}: \ell(I_3) = \ell(I_1) \ell(I_2) \big\}, 
\\ 
\D_{\w, \Z}^k & := \big\{I=I_1 \times I_2 \times I_3 \in \D_{\w}^k: \ell(I_3) = \ell(I_1) \ell(I_2) \big\}.
\end{align*} 
It follows from \eqref{eq:fhh}, \eqref{def:Ek-1}, and \eqref{def:Ek-2} that 
\begin{align}\label{TDE-1}
\langle Tf, g \rangle 
& = \sum_{\substack{I_1, J_1 \in \D_1 \\ \ell(I_1) > \ell(J_1)}} 
+ \sum_{\substack{I_1, J_1 \in \D_1 \\ \ell(I_1) < \ell(J_1)}} 
+ \sum_{\substack{I_1, J_1 \in \D_1 \\ \ell(I_1) = \ell(J_1)}}
\langle T \Delta_{I_1} f, \Delta_{J_1} g \rangle
\\ \nonumber
& = \sum_{\substack{I_1, J_1 \in \D_1 \\ \ell(I_1) = \ell(J_1)}} 
\big[\langle T E_{I_1} f, \Delta_{J_1} g \rangle 
+ \langle T \Delta_{I_1} f, E_{J_1} g \rangle 
+ \langle T \Delta_{I_1} f, \Delta_{J_1} g \rangle \big]. 
\end{align}
Furthermore, by definition, 
\begin{align*}
\sum_{\substack{I_{2, 3} \in \D_{\lambda}^{2, 3} \\ \ell(I_2) > 2^k}} 
\Delta_{I_{2, 3}} f 
= \sum_{\substack{I_{2, 3} \in \D_{\lambda}^{2, 3} \\ \ell(I_2) = 2^k}} 
E_{I_{2, 3}} f,
\end{align*}
which gives  
\begin{align}\label{TDE-2}
\langle T f, g \rangle 
& = \sum_{\substack{I_{2, 3}, J_{2, 3} \in \D_{\ell(I_1)}^{2, 3} \\ \ell(I_2) > \ell(J_2)}}
+ \sum_{\substack{I_{2, 3}, J_{2, 3} \in \D_{\ell(I_1)}^{2, 3} \\ \ell(I_2) < \ell(J_2)}} 
+ \sum_{\substack{I_{2, 3}, J_{2, 3} \in \D_{\ell(I_1)}^{2, 3} \\ \ell(I_2) = \ell(J_2)}} 
\langle T \Delta_{I_{2, 3}} f, \Delta_{J_{2, 3}} g \rangle
\\ \nonumber
& = \sum_{\substack{I_{2, 3}, J_{2, 3} \in \D_{\ell(I_1)}^{2, 3} \\ \ell(I_2) = \ell(J_2)}}  
\big[ \langle T E_{I_{2, 3}} f, \Delta_{J_{2, 3}} g \rangle
+ \langle T \Delta_{I_{2, 3}} f, E_{J_{2, 3}} g \rangle 
+ \langle T \Delta_{I_{2, 3}} f, \Delta_{J_{2, 3}} g \rangle \big]. 
\end{align}
Then invoking \eqref{TDE-2} applied to resulting terms in \eqref{TDE-1}, we obtain 
\begin{align}\label{TS-1}
\langle Tf, g \rangle 
= \sum_{j=1}^9 \S_j(\w), 
\end{align}
where 
\begin{align*}
\S_1(\w) & := \sum_{\substack{I, J \in \D_{\w, \Z} \\ \ell(I) = \ell(J)}}  
\langle T E_{I_1} E_{I_{2, 3}} f, \Delta_{J_1} \Delta_{J_{2, 3}} g \rangle, \quad
\S_2(\w) := \sum_{\substack{I, J \in \D_{\w, \Z} \\ \ell(I) = \ell(J)}}  
\langle T \Delta_{I_1} \Delta_{I_{2, 3}} f, E_{J_1} E_{J_{2, 3}} g \rangle, 
\\
\S_3(\w) & := \sum_{\substack{I, J \in \D_{\w, \Z} \\ \ell(I) = \ell(J)}}  
\langle T E_{I_1} \Delta_{I_{2, 3}} f, \Delta_{J_1} E_{J_{2, 3}} g \rangle, \quad 
\S_4(\w) := \sum_{\substack{I, J \in \D_{\w, \Z} \\ \ell(I) = \ell(J)}}  
\langle T \Delta_{I_1} E_{I_{2, 3}} f, E_{J_1} \Delta_{J_{2, 3}} g \rangle,
\\
\S_5(\w) & := \sum_{\substack{I, J \in \D_{\w, \Z} \\ \ell(I) = \ell(J)}}  
\langle T E_{I_1} \Delta_{I_{2, 3}} f, \Delta_{J_1} \Delta_{J_{2, 3}} g \rangle, \quad 
\S_6(\w) := \sum_{\substack{I, J \in \D_{\w, \Z} \\ \ell(I) = \ell(J)}}  
\langle T \Delta_{I_1} \Delta_{I_{2, 3}} f, E_{J_1} \Delta_{J_{2, 3}} g \rangle, 
\\
\S_7(\w) & := \sum_{\substack{I, J \in \D_{\w, \Z} \\ \ell(I) = \ell(J)}}  
\langle T \Delta_{I_1} E_{I_{2, 3}} f, \Delta_{J_1} \Delta_{J_{2, 3}} g \rangle, \quad
\S_8(\w) := \sum_{\substack{I, J \in \D_{\w, \Z} \\ \ell(I) = \ell(J)}}  
\langle T \Delta_{I_1} \Delta_{I_{2, 3}} f, \Delta_{J_1} E_{J_{2, 3}} g \rangle, 
\\
\S_9(\w) & := \sum_{\substack{I, J \in \D_{\w, \Z} \\ \ell(I) = \ell(J)}}  
\langle T \Delta_{I_1} \Delta_{I_{2, 3}} f, \Delta_{J_1} \Delta_{J_{2, 3}} g \rangle.
\end{align*}
Then taking an expectation $\E_{\w}$ of \eqref{TS-1} gives 
\begin{align}\label{TS-2}
\langle Tf, g \rangle 
= \sum_{j=1}^9 \E_{\w} \S_j(\w). 
\end{align}
To show Theorem \ref{thm:repre}, it suffices to focus on $\S_1(\w)$: 
\begin{align*}
\S_1(\w) 
= \sum_{\substack{I, J \in \D_{\w, \Z} \\ \ell(I) = \ell(J)}}  
\langle T (h_{I_1}^0 \otimes h_{I_{2, 3}}^0) , h_{J_1} \otimes h_{J_{2, 3}} \rangle
\langle f, h_{I_1}^0 \otimes h_{I_{2, 3}}^0 \rangle 
\langle g, h_{J_1} \otimes h_{J_{2, 3}} \rangle, 
\end{align*}
because our general strategy and techniques are contained in the estimate for $\S_1(w)$. 

Using the decomposition
\begin{align*}
h_{I_1}^0 
& = (h_{I_1}^0 - h_{J_1}^0) + h_{J_1}^0 
=: H_{I_1, J_1} + h_{J_1}^0, 
\\ 
h_{I_{2, 3}}^0 
& = (h_{I_{2, 3}}^0 - h_{J_{2, 3}}^0) + h_{J_{2, 3}}^0 
=: H_{I_{2, 3}, J_{2, 3}} + h_{J_{2, 3}}^0, 
\end{align*}
we have 
\begin{align}\label{S6}
\S_1(\w) 
= \sum_{i=1}^4 \S_1^i(\w), 
\end{align}
where 
\begin{align*}
\S_1^1(\w) 
& := \sum_{\substack{I, J \in \D_{\w, \Z} \\ \ell(I) = \ell(J)}}  
\langle T (h_{I_1}^0 \otimes h_{I_{2, 3}}^0) , h_{J_1} \otimes h_{J_{2, 3}} \rangle
\langle f, H_{I_1, J_1} \otimes H_{I_{2, 3}, J_{2, 3}} \rangle 
\langle g, h_{J_1} \otimes h_{J_{2, 3}} \rangle, 
\\ 
\S_1^2(\w) 
& := \sum_{\substack{I, J \in \D_{\w, \Z} \\ \ell(I) = \ell(J)}}  
\langle T (h_{I_1}^0 \otimes h_{I_{2, 3}}^0) , h_{J_1} \otimes h_{J_{2, 3}} \rangle
\langle f, H_{I_1, J_1} \otimes h_{J_{2, 3}}^0 \rangle 
\langle g, h_{J_1} \otimes h_{J_{2, 3}} \rangle, 
\\
\S_1^3(\w) 
& := \sum_{\substack{I, J \in \D_{\w, \Z} \\ \ell(I) = \ell(J)}}  
\langle T (h_{I_1}^0 \otimes h_{I_{2, 3}}^0) , h_{J_1} \otimes h_{J_{2, 3}} \rangle
\langle f, h_{J_1}^0 \otimes H_{I_{2, 3}, J_{2, 3}} \rangle 
\langle g, h_{J_1} \otimes h_{J_{2, 3}} \rangle, 
\\
\S_1^4(\w) 
& := \sum_{\substack{I, J \in \D_{\w, \Z} \\ \ell(I) = \ell(J)}}  
\langle T (h_{I_1}^0 \otimes h_{I_{2, 3}}^0) , h_{J_1} \otimes h_{J_{2, 3}} \rangle
\langle f, h_{J_1}^0 \otimes h_{J_{2, 3}}^0 \rangle 
\langle g, h_{J_1} \otimes h_{J_{2, 3}} \rangle.
\end{align*}
The cancellation assumption (cf. Definition \ref{def:cancellation}) implies  
\begin{align*}
\langle T (\mathbf{1}_{I_1} \otimes 1) , h_{J_1} \otimes h_{J_{2, 3}} \rangle
= \sum_{J'_1 \in \ch(J_1)} 
\langle T (\mathbf{1}_{I_1} \otimes 1) , \mathbf{1}_{J'_1} \otimes h_{J_{2, 3}} \rangle 
\langle h_{J_1} \rangle_{J'_1} 
= 0, 
\end{align*}
and 
\begin{align*}
\langle T (1 \otimes \mathbf{1}_{I_{2, 3}}) , h_{J_1} \otimes h_{J_{2, 3}} \rangle
= \sum_{J'_{2, 3} \in \ch(J_{2, 3})} 
\langle T (1 \otimes \mathbf{1}_{I_{2, 3}}) , h_{J_1} \otimes \mathbf{1}_{J'_{2, 3}} \rangle 
\langle h_{J_{2, 3}} \rangle_{J'_{2, 3}} 
= 0. 
\end{align*}
Hence, 
\begin{align*}
\S_1^2(\w) 
&= \sum_{\substack{J \in \D_{\w, \Z}, \, I \in \D^1 \\ \ell(I_1) = \ell(J_1)}}  
\sum_{\substack{I_{2, 3} \in \D_{\ell(I_1)}^{2, 3} \\ \ell(I_{2, 3}) = \ell(J_{2, 3})}} |I_1|^{-\frac12} 
\langle T (\mathbf{1}_{I_1} \otimes \mathbf{1}_{I_{2, 3}}) , h_{J_1} \otimes h_{J_{2, 3}} \rangle
\langle f, H_{I_1, J_1} \otimes \frac{\mathbf{1}_{J_{2, 3}}}{|J_{2, 3}|} \rangle 
\langle g, h_{J_1} \otimes h_{J_{2, 3}} \rangle
\\ 
&= \sum_{\substack{J \in \D_{\w, \Z}, \, I \in \D^1 \\ \ell(I_1) = \ell(J_1)}} |I_1|^{-\frac12} 
\langle T (\mathbf{1}_{I_1} \otimes 1) , h_{J_1} \otimes h_{J_{2, 3}} \rangle
\langle f, H_{I_1, J_1} \otimes \frac{\mathbf{1}_{J_{2, 3}}}{|J_{2, 3}|} \rangle 
\langle g, h_{J_1} \otimes h_{J_{2, 3}} \rangle
= 0, 
\end{align*}

\begin{align*}
\S_1^3(\w) 
&= \sum_{\substack{J \in \D_{\w, \Z}, \, I_{2, 3} \in \D_{\ell(J_1)}^{2, 3} \\ \ell(I_{2, 3}) = \ell(J_{2, 3})}}
\sum_{\substack{I \in \D^1 \\ \ell(I_1) = \ell(J_1)}}  
\langle T (\mathbf{1}_{I_1} \otimes \mathbf{1}_{I_{2, 3}}) , h_{J_1} \otimes h_{J_{2, 3}} \rangle
\langle f, \frac{\mathbf{1}_{J_1}}{|J_1|} \otimes H_{I_{2, 3}, J_{2, 3}} \rangle 
\langle g, h_{J_1} \otimes h_{J_{2, 3}} \rangle
\\ 
&= \sum_{\substack{J \in \D_{\w, \Z}, \, I_{2, 3} \in \D_{\ell(J_1)}^{2, 3} \\ \ell(I_{2, 3}) = \ell(J_{2, 3})}}
\langle T (1 \otimes \mathbf{1}_{I_{2, 3}}) , h_{J_1} \otimes h_{J_{2, 3}} \rangle
\langle f, \frac{\mathbf{1}_{J_1}}{|J_1|} \otimes H_{I_{2, 3}, J_{2, 3}} \rangle 
\langle g, h_{J_1} \otimes h_{J_{2, 3}} \rangle
= 0, 
\end{align*}
and 
\begin{align*}
\S_1^4(\w) 
&= \sum_{J \in \D_{\w, \Z}} 
\sum_{\substack{I_{2, 3} \in \D_{\ell(J_1)}^{2, 3} \\ \ell(I_{2, 3}) = \ell(J_{2, 3})}}
\sum_{\substack{I \in \D^1 \\ \ell(I_1) = \ell(J_1)}}  
\langle T (\mathbf{1}_{I_1} \otimes \mathbf{1}_{I_{2, 3}}) , h_{J_1} \otimes h_{J_{2, 3}} \rangle
\langle f \rangle_J 
\langle g, h_{J_1} \otimes h_{J_{2, 3}} \rangle
\\ 
&= \sum_{J \in \D_{\w, \Z}} 
\sum_{\substack{I_{2, 3} \in \D_{\ell(J_1)}^{2, 3} \\ \ell(I_{2, 3}) = \ell(J_{2, 3})}}
\langle T (1 \otimes \mathbf{1}_{I_{2, 3}}) , h_{J_1} \otimes h_{J_{2, 3}} \rangle
\langle f \rangle_J 
\langle g, h_{J_1} \otimes h_{J_{2, 3}} \rangle
= 0. 
\end{align*}
These and \eqref{S6} yield 
\begin{align*}
\S_1(\w) = \S_1^1(\w) 
= \sum_{\substack{k=(k_1, k_2, k_3) \in \N^3 \\ k_1, k_2, k_3 \ge 2}}  
\sum_{\substack{I, J \in \D_{0, \Z}: \, \ell(J) = \ell(I) \\ 2^{k_i-3} < \rd(I_i, J_i) \le 2^{k_i-2}}} 
\Lambda(I + \w, J + \w),   
\end{align*}
where 
\begin{align*}
\Lambda(I, J) 
:= \langle T (h_{I_1}^0 \otimes h_{I_{2, 3}}^0) , h_{J_1} \otimes h_{J_{2, 3}} \rangle
\langle f, H_{I_1, J_1} \otimes H_{I_{2, 3}, J_{2, 3}} \rangle 
\langle g, h_{J_1} \otimes h_{J_{2, 3}} \rangle. 
\end{align*}
Recall that $\w = (\w_1, w_2, \w_3)$ with $\w_i = (\w_i^{j_i})_{j_i \in \ZZ}$, $i=1, 2, 3$. Observe that $\mathbf{1}_{\{I + \w \in \D_{\w, \Z}^k\}}(\w)$ depends on $\w_i^{j_i}$ for $\ell(I_i) \le 2^{-j_i} < \ell(I_i^{(k_i)})$, $i=1, 2, 3$; while $\Lambda(I + \w, J + \w)$ depends on $\w_i^{j_i}$ for $2^{-j_i} < \ell(I_i) = \ell(J_i)$, $i=1, 2, 3$. Thus, by independence and \eqref{pigd}, we arrive at 
\begin{align}\label{indep-1}
\E_{\w} \S_1(\w) 
& = \sum_{k_1, k_2, k_3 \ge 2}  
\sum_{\substack{I, J \in \D_{0, \Z}: \, \ell(J) = \ell(I) \\ 2^{k_i-3} < \rd(I_i, J_i) \le 2^{k_i-2}}} 
8 \, \E_{\w} \mathbf{1}_{\{I + \w \in \D_{\w, \Z}^k\}}(\w) \, 
\E_{\w} \Lambda(I + \w, J + \w)
\\ \nonumber
& = \sum_{k_1, k_2, k_3 \ge 2}  
\sum_{\substack{I, J \in \D_{0, \Z}: \, \ell(J) = \ell(I) \\ 2^{k_i-3} < \rd(I_i, J_i) \le 2^{k_i-2}}} 
8 \, \E_{\w} \mathbf{1}_{\{I + \w \in \D_{\w, \Z}^k\}}(\w) \Lambda(I + \w, J + \w)
\\ \nonumber
& = 8 \, \E_{\w} \sum_{k_1, k_2, k_3 \ge 2}  
\sum_{I \in \D_{\w, \Z}^k} 
\sum_{\substack{J \in \D_{\w, \Z}: \, \ell(J) = \ell(I) \\ 2^{k_i-3} < \rd(I_i, J_i) \le 2^{k_i-2}}} 
\Lambda(I, J). 
\end{align}
Given $k=(k_1, k_2, k_3) \in \N^3$ with $k_1, k_2, k_3 \ge 2$, by the $k$-goodness, we have
\begin{equation}\label{indep-2}
\begin{aligned}
&\text{for any $I \in \D_{\w, \Z}^k$ and $J \in \D_{\w, \Z}$ with $\ell(I) = \ell(J)$ and $\rd(I_i, J_i) \le 2^{k_i-2}$, $i=1, 2, 3$,}
\\
&\text{there must hold $J_i \subset I^{(k_i)}$ (hence, $I_i^{(k_i)} = J_i^{(k_i)}$), $i=1, 2, 3$.}
\end{aligned}
\end{equation}
This leads to 
\begin{align*}
\E_{\w} \S_1(\w)
& = 8 C_0 \, \E_{\w} \sum_{\substack{k = (k_1, k_2, k_3) \\ k_1, k_2, k_3 \ge 2}} 
\varphi(k) 
\sum_{K \in \D_{\lambda(k)}} 
\sum_{\substack{I, J \in \D_{\w, \Z} \\ I^{(k)} = J^{(k)} = K}} 
\\
&\qquad \times a_{IJK}
\langle f, H_{I_1, J_1} \otimes H_{I_{2, 3}, J_{2, 3}} \rangle 
\langle g, h_{J_1} \otimes h_{J_{2, 3}} \rangle, 
\end{align*}
where $\lambda(k) := 2^{-k_1 - k_2 + k_3}$ and 
\begin{align}\label{AAG}
a_{IJK}
:= \frac{\langle T (h_{I_1}^0 \otimes h_{I_{2, 3}}^0) , h_{J_1} \otimes h_{J_{2, 3}} \rangle}{C_0 \, \varphi(k)}
=: \frac{\G_{IJ}}{C_0 \, \varphi(k)}
\end{align}
if $I \in \D_{\w, \Z}^k$ and $J \in \D_{\w, \Z}$ with $\ell(I) = \ell(J)$ and $2^{k_i-3} < \rd(I_i, J_i) \le 2^{k_i-2}$ for each $i=1, 2, 3$, and $a_{IJK} = 0$ otherwise. 

It remains to show that there exists a uniform constant $C_0 \in (1, \infty)$ such that 
\begin{align}\label{aIJK-1}  
|a_{IJK}| 
\le \mathbf{F}(K) \frac{|I|^{\frac12} |J|^{\frac12}}{|K|}
\end{align}
with 
\begin{align}\label{aIJK-2}  
\mathbf{F}(K) \le 1 
\quad \text{and} \quad 
\lim_{N \to \infty} \sup_{\D} \sup_{K \not\in \D_{\lambda(k)}^N} \mathbf{F}(K) 
= 0.
\end{align}
We claim that for all $I, J \in \D_{\w, \Z}$ with $\ell(I) = \ell(J)$ and $I, J \subset K$, there holds
\begin{equation}\label{GGIJ}
|\G_{IJ}| \lesssim 
\begin{cases}
{\displaystyle \frac{1}{\rd(I_1, J_1)^{\delta_1}}
\prod_{i=1}^3 \frac{\widehat{F}_i(I_i, K_i)}{\rd(I_i, J_i)}, 
\quad \text{if $\rd(I_1, J_1) > 1$ and $I_{2, 3} = J_{2, 3}$}},
\vspace{0.2cm}
\\ 
{\displaystyle \frac{\min\{\rd(I_2, J_2), \rd(I_3, J_3)\}^{-\delta_{2, 3}}}{\rd(I_1, J_1)^{\delta_1} D_{\theta}(I, J)} \prod_{i=1}^3 \frac{\widehat{F}_i(I_i, K_i)}{\rd(I_i, J_i)}, \quad\text{otherwise}},
\end{cases}
\end{equation} 
where $\widehat{F}_i(I_i, K_i)$ is defined in \eqref{def:FIK-2}, and 
\begin{align*}
D_{\theta}(I, J) := \bigg[\frac{\rd(I_1, J_1) \rd(I_2, J_2)}{\rd(I_3, J_3)} 
+ \frac{\rd(I_3, J_3)}{\rd(I_1, J_1) \rd(I_2, J_2)} \bigg]^{\theta}. 
\end{align*}
Assume that \eqref{GGIJ}  holds momentarily. Note that 
\begin{equation}\label{KDN}
K \not\in \D_{\lambda(k)}^N 
\Longrightarrow 
\begin{cases}
K_1 \notin \D_{\w_1}(N) \, \, \, \text{ or } \, \, \,
K_2 \notin \D_{\w_2}(N) \, \, \, \text{ or }  
\\
2^{-2N} \lesssim \ell(K_3) \lesssim 2^{2N} \text{ and } \rd(K_3, 2^{2N} \I) > 2N,   
\end{cases}
\end{equation}
where the last condition states the fact $K_3 \notin \D_{\w_3}(N)$ in some sense. Then picking 
\begin{align}\label{BFK} 
\mathbf{F}(K) := \mathbf{F}_1(K_1) \, \mathbf{F}_2(K_2) \, \mathbf{F}_3(K_3), 
\end{align}
where 
\begin{align*}
\mathbf{F}_i(K_i)
:= F_{i, 1}(\ell(K_i)) \widetilde{F}_{i, 2}(2^{-k_i} \ell(K_i)) \widetilde{F}_{i, 3}(K_i) 
+ \sum_{I_i^{(k_i+1)} = K_i} F_i(I_i), 
\qquad i=1, 2, 3, 
\end{align*}
we conclude \eqref{aIJK-1} and \eqref{aIJK-2} from \eqref{AAG}, \eqref{KDN}, Lemma \ref{lem:FF}, and an elementary estimate 
\begin{align}\label{MKK}
- \min\{k_2, \, k_3\} \delta_{2, 3} - |k_3 - k_1 - k_2| \theta
\le - k_2 \min\{\delta_{2, 3}, \theta\} - \max\{k_3 - k_1 - k_2, \, 0\} \theta.
\end{align}

The rest of this section is devoted to showing \eqref{GGIJ}. In what follows, fix an arbitrary $\w = (\w_1, \w_2, \w_3)$ and let $I, J \in \D_{\w, \Z}$ with $\ell(I) = \ell(J)$ and $I, J \subset K$. The proof of \eqref{GGIJ} will be given by distinguishing the relative positions between $I_1$ and $J_2$ and between $I_{2, 3}$ and $J_{2, 3}$ respectively. In each case, we will obtain improved estimates than \eqref{GGIJ}. In addition, we may assume that $\theta \in (0, 1)$ because $D_1(x) \le D_{\theta}(x)$ for all $x \in \Rn$ and $\theta \in (0, 1)$.

\subsection{Refined functions}\label{sec:refined}
Before proving \eqref{GGIJ}, we would like to improve some properties of functions appearing in hypotheses in order to simplify our argument below.

\begin{list}{\rm (\theenumi)}{\usecounter{enumi}\leftmargin=1.2cm \labelwidth=1cm \itemsep=0.2cm \topsep=0.2cm \renewcommand{\theenumi}{P\arabic{enumi}}} 

\item\label{list:P1} Assume that for each $i=1, 2$ and $j=1, 2, 3$, the function $F_{i, j}$ in Definitions \ref{def:full} is the same as the ones in Definitions \ref{def:partial-1}--\ref{def:partial-2}. Indeed, given $(F_{i, 1}^1, F_{i, 2}^1, F_{i, 3}^1) \in \F$ in Definition \ref{def:full}, $(F_{i, 1}^2, F_{i, 2}^2, F_{i, 3}^2) \in \F$ in Definition \ref{def:partial-1}, and $(F_{i, 1}^3, F_{i, 2}^3, F_{i, 3}^3)  \in \F$ in Definition \ref{def:partial-2}, if we define 
\begin{align*}
F_{i, j} := \max\{F_{i, j}^1, F_{i, j}^2, F_{i, j}^3\}, \qquad j=1, 2, 3, 
\end{align*}
then the compact full and partial kernel representations hold for $(F_{i, 1}, F_{i, 2}, F_{i, 3}) \in \F$.

\item\label{list:P2} For each $i=1,2$ and $(F_{i, 1}, F_{i, 2}, F_{i, 3}) \in \F$ in Definitions \ref{def:full}--\ref{def:partial-2}, we may assume that $F_{i, 1}$ is monotonically increasing while $F_{i, 2}$ and $F_{i, 3}$ are monotone decreasing. Indeed, if we define 
\begin{align*}
F_{i, 1}^*(t) := \sup_{0 \le s \le t} F_{i, 1}(s), \quad 
F_{i, 2}^*(t) := \sup_{s \ge t} F_{i, 2}(s), \quad\text{and}\quad 
F_{i, 3}^*(t) := \sup_{s \ge t} F_{i, 3}(s),
\end{align*}
then $F^*_{i, 1}$ is monotone increasing, $F^*_{i, 2}$ and $F^*_{i, 3}$ are monotone decreasing, and the compact full and partial kernel representations hold for $(F_{i, 1}^*, F_{i, 2}^*, F_{i, 3}^*) \in \F$.

\item\label{list:P3} Since any dilation of functions in $\F$ and $\F_0$ still belong to the original space, we will often omit all universal constants appearing in those functions.  

\item\label{list:P4} In view of Lemma \ref{lem:improve}, we will use alternative estimates for kernels in Definitions \ref{def:full}--\ref{def:partial-2}. That is, in Definition \ref{def:full}, the size condition is replaced by 
\begin{align*}
|K(x, y)| 
\leq D_{\theta}(x-y) \prod_{i=1}^3 \frac{F_i(x_i, y_i)}{|x_i - y_i|},  
\end{align*}
where 
\begin{align*}
F_i(x_i, y_i) := F_{i, 1}(|x_i - y_i|) F_{i, 2}(|x_i - y_i|) F_{i, 3} \bigg(1+ \frac{|x_i + y_i|}{1+|x_i - y_i|} \bigg);  
\end{align*}
while the first H\"{o}lder condition is replaced by 
\begin{align*}
& |K(x, y) - K((x_1, x'_2, x'_3), y) - K((x'_1, x_2, x_3), y) + K(x', y)| 
\\
& \leq \bigg(\frac{|x_1 - x'_1|}{|x_1 - y_1|}\bigg)^{\delta_1} 
\bigg(\frac{|x_2 - x'_2|}{|x_2 - y_2|} + \frac{|x_3 - x'_3|}{|x_3 - y_3|}\bigg)^{\delta_{2, 3}}
D_{\theta}(x-y) \prod_{i=1}^3 \frac{F_i(x_i, y_i)}{|x_i - y_i|}
\end{align*}
whenever $|x_i - x'_i| \leq |x_i - y_i|/2$ for $i=1, 2, 3$, where 
\begin{align*}
F_i(x_i, y_i) := F_{i, 1}(|x_i - x'_i|) F_{i, 2}(|x_i - y_i|) F_{i, 3} \bigg(1+ \frac{|x_i + y_i|}{1+|x_i - y_i|} \bigg). 
\end{align*} 
Other conditions can be formulated in a similar way. 
\end{list}

\subsection{Separated/Separated}\label{sec:SS}
We begin with the case $\rd(I_1, J_1) > 1$ and $\rd(I_{2, 3}, J_{2, 3}) > 1$, which will be treated in three subcases below. 

\subsubsection{\bf Case 1: $\rd(I_i, J_i) > 1$, $i=1, 2, 3$} 
Let $x_i \in J_i$ and $y_i \in I_i$, $i=1, 2, 3$. Then,  
\begin{align}\label{DT-0}
\frac{|x_i - c_{J_i}|}{|x_i - y_i|} 
\le \frac{\ell(J_i)/2}{\d(I_i, J_i)} 
=\frac{1/2}{\rd(I_i, J_i)},
\end{align}
and
\begin{align}\label{DT-1}
|x_i - y_i| 
\le \d(I_i, J_i) + \ell(I_i) + \ell(J_i) 
\le 3\d(I_i, J_i) \le 3|x_i - y_i|, 
\end{align}
which implies  
\begin{align}\label{DT-2}
D_{\theta}(x-y) 
&= \Bigg[ \frac{\frac{|x_1 - y_1|}{\ell(I_1)} \frac{|x_2 - y_2|}{\ell(I_2)}}{\frac{|x_3 - y_3|}{\ell(I_3)}} 
+ \frac{\frac{|x_1 - y_1|}{\ell(I_1)} \frac{|x_2 - y_2|}{\ell(I_2)}}{\frac{|x_3 - y_3|}{\ell(I_3)}} \Bigg]^{-\theta}
\\ \nonumber
&\simeq \Bigg[ \frac{\frac{\d(I_1, J_1)}{\ell(I_1)} \frac{\d(I_2, J_2)}{\ell(I_2)}}{\frac{\d(I_3, J_3)}{\ell(I_3)}} 
+ \frac{\frac{\d(I_1, J_1)}{\ell(I_1)} \frac{\d(I_2, J_2)}{\ell(I_2)}}{\frac{\d(I_3, J_3)}{\ell(I_3)}} \Bigg]^{-\theta}
= D_{\theta}(I, J)^{-1}. 
\end{align}
Write
\begin{align*}
\mathbf{K}(x, y)
:= K(x, y) - K((x_1, c_{J_{2, 3}})) - K((c_{J_1}, x_{2, 3})) + K(c_J, y).
\end{align*}
In light of \eqref{DT-0} and \eqref{DT-2}, we use the compact full kernel representation, the cancellation of $h_{J_1}$ and $h_{J_{2, 3}}$, and the H\"{o}lder condition of $K$ to deduce that 
\begin{align*}
|\mathscr{G}_{I, J}| 
&=\bigg|\int_{I_1 \times I_{2, 3}} \int_{J_1 \times J_{2, 3}}
\mathbf{K}(x, y) h_{I_1}^0 \otimes h_{I_{2, 3}}^0(y) 
h_{J_1} \otimes h_{J_{2, 3}}(x) \, dx \, dy \bigg|
\\
&\lesssim \int_{I_1 \times I_{2, 3}} \int_{J_1 \times J_{2, 3}} 
\bigg(\frac{|x_1 - c_{J_1}|}{|x_1 - y_1|}\bigg)^{\delta_1} 
\bigg(\frac{|x_2 - c_{J_2}|}{|x_2 - y_2|} + \frac{|x_3 - c_{J_3}|}{|x_3 - y_3|}\bigg)^{\delta_{2, 3}}
\\
&\quad\times D_{\theta}(x-y) \prod_{i=1}^3 \frac{F_i(x_i, y_i)}{|x_i - y_i|} 
|h_{I_1}^0 \otimes h_{I_{2, 3}}^0(y)| |h_{J_1} \otimes h_{J_{2, 3}}(x)| \, dx \, dy
\\
&\lesssim \frac{1}{\rd(I_1, J_1)^{\delta_1}} 
\bigg[\frac{1}{\rd(I_2, J_2)} + \frac{1}{\rd(I_3, J_3)}\bigg]^{\delta_{2, 3}}
\frac{1}{D_{\theta}(I, J)} \prod_{i=1}^3 \mathscr{P}_i(I_i, J_i) |I_i|^{-1}
\\ 
&\lesssim  \frac{\min\{\rd(I_2, J_2), \rd(I_3, J_3)\}^{-\delta_{2, 3}}}{\rd(I_1, J_1)^{\delta_1} 
D_{\theta}(I, J)}
\prod_{i=1}^3 \frac{F_i(I_i, K_i)}{\rd(I_i, J_i)}, 
\end{align*}
where \eqref{def:PP-1} was used in the last inequality.
\qed

\subsubsection{\bf Case 2: $\rd(I_1, J_1) > 1$, $\rd(I_2, J_2) > 1$, and $\rd(I_3, J_3) = 1$} 
In this case, we invoke \eqref{DT-0}--\eqref{DT-1} to obtain that 
\begin{align}\label{HH21}
\frac{|x_1 - c_{J_1}|}{|x_1 - y_1|} \le \frac{1/2}{\rd(I_1, J_1)}
\quad \text{ and }\quad 
|x_i - y_i| \simeq \d(I_i, J_i), \quad i=1, 2, 
\end{align}
which gives 
\begin{align}\label{HH22}
D_{\theta}(x-y) 
&= \Bigg[ \frac{\frac{|x_1 - y_1|}{\ell(I_1)} \frac{|x_2 - y_2|}{\ell(I_2)}}{\frac{|x_3 - y_3|}{\ell(I_3)}} 
+ \frac{\frac{|x_3 - y_3|}{\ell(I_3)}}{\frac{|x_1 - y_1|}{\ell(I_1)} \frac{|x_2 - y_2|}{\ell(I_2)}} \Bigg]^{-\theta}
\\ \nonumber 
&\simeq \Bigg[ \frac{\frac{\d(I_1, J_1)}{\ell(I_1)} \frac{\d(I_2, J_2)}{\ell(I_2)}}{\frac{|x_3 - y_3|}{\ell(I_3)}} 
+ \frac{\frac{|x_3 - y_3|}{\ell(I_3)}}{\frac{\d(I_1, J_1)}{\ell(I_1)} \frac{\d(I_2, J_2)}{\ell(I_2)}} \Bigg]^{-\theta}
\\ \nonumber 
&= \bigg[ \frac{\rd(I_1, J_1) \rd(I_2, J_2)}{|x_3 - y_3|/\ell(I_3)} 
+ \frac{|x_3 - y_3|/\ell(I_3)}{\rd(I_1, J_1) \rd(I_2, J_2)}  \bigg]^{-\theta}. 
\end{align}
Denote 
\begin{align*}
D_{\theta}^{1, 2}(I, J)
:= \bigg[\rd(I_1, J_1) \rd(I_2, J_2)  + \frac{1}{\rd(I_1, J_1) \rd(I_2, J_2)} \bigg]^{\theta}. 
\end{align*} 
If we set 
\begin{align*}
t_3 := \frac{1}{\rd(I_1, J_1) \rd(I_2, J_2) \ell(I_3)}
\quad\text{ and }\quad 
\mathbf{K}(x, y) := K(x, y) - K((c_{J_1}, x_{2, 3}), y), 
\end{align*}
then by the compact full kernel representation, the cancellation of $h_{J_1}$, the mixed size-H\"{o}lder condition of $K$, and \eqref{HH21}--\eqref{HH22}, 
\begin{align}\label{HH23}
|\mathscr{G}_{I, J}| 
&=\bigg|\int_{I_1 \times I_{2, 3}} \int_{J_1 \times J_{2, 3}}
\mathbf{K}(x, y) h_{I_1}^0 \otimes h_{I_{2, 3}}^0(y) 
h_{J_1} \otimes h_{J_{2, 3}}(x) \, dx \, dy \bigg|
\\ \nonumber 
&\lesssim \int_{I_1 \times I_{2, 3}} \int_{J_1 \times J_{2, 3}} 
\bigg(\frac{|x_1 - c_{J_1}|}{|x_1 - y_1|}\bigg)^{\delta_1} D_{\theta}(x-y) 
\\ \nonumber 
&\quad\times \prod_{i=1}^3 \frac{F_i(x_i, y_i)}{|x_i - y_i|} 
|h_{I_1}^0 \otimes h_{I_{2, 3}}^0(y)| |h_{J_1} \otimes h_{J_{2, 3}}(x)| \, dx \, dy
\\ \nonumber 
&\lesssim \frac{1}{\rd(I_1, J_1)^{\delta_1}} \prod_{i=1}^2 \mathscr{P}_i(I_i, J_i) |I_i|^{-1} 
\times \mathscr{Q}_3^{t_3}(I_3, J_3) |I_3|^{-1} 
\\ \nonumber 
&\lesssim \frac{F_1(I_1, K_1)}{\rd(I_1, J_1)^{1+\delta_1}} 
\frac{F_2(I_2, K_2)}{\rd(I_2, J_2)} 
\frac{F_3(I_3, I_3)}{D_{\theta}^{1, 2}(I, J)}, 
\end{align}
where \eqref{def:PP-1} and \eqref{def:QQ} were used in the last step. 
\qed

\subsubsection{\bf Case 3: $\rd(I_1, J_1) > 1$, $\rd(I_2, J_2) = 1$, and $\rd(I_3, J_3) > 1$} 
The condition $\rd(I_1, J_1) > 1$ and $\rd(I_3, J_3) > 1$ implies 
\begin{align}\label{HH31}
\frac{|x_1 - c_{J_1}|}{|x_1 - y_1|} \le \frac{1/2}{\rd(I_1, J_1)}
\quad \text{ and }\quad 
|x_i - y_i| \simeq \d(I_i, J_i), \quad i=1, 3, 
\end{align}
which immediately yields  
\begin{align}\label{HH32}
D_{\theta}(x-y) 
&= \Bigg[ \frac{\frac{|x_1 - y_1|}{\ell(I_1)} \frac{|x_2 - y_2|}{\ell(I_2)}}{\frac{|x_3 - y_3|}{\ell(I_3)}} 
+ \frac{\frac{|x_3 - y_3|}{\ell(I_3)}}{\frac{|x_1 - y_1|}{\ell(I_1)} \frac{|x_2 - y_2|}{\ell(I_2)}} \Bigg]^{-\theta}
\\ \nonumber
&\simeq \Bigg[ \frac{\frac{\d(I_1, J_1)}{\ell(I_1)} \frac{|x_2 - y_2|}{\ell(I_2)}}{\frac{\d(I_3, J_3)}{\ell(I_3)}} 
+ \frac{\frac{\d(I_3, J_3)}{\ell(I_3)}}{\frac{\d(I_1, J_1)}{\ell(I_1)} \frac{|x_2 - y_2|}{\ell(I_2)}} \Bigg]^{-\theta}
\\ \nonumber
&= \bigg[ \frac{\rd(I_1, J_1)}{\rd(I_3, J_3)} \frac{|x_2 - y_2|}{\ell(I_2)} 
+ \frac{\rd(I_3, J_3)}{\rd(I_1, J_1)} \frac{\ell(I_2)}{|x_2 - y_2|} \bigg]^{-\theta}. 
\end{align}
Let 
\begin{align*}
t_2 := \frac{\rd(I_1, J_1)}{\ell(I_2) \rd(I_3, J_3)} 
\quad\text{ and }\quad 
\mathbf{K}(x, y) := K(x, y) - K((c_{J_1}, x_{2, 3}), y). 
\end{align*}
In view of \eqref{HH31}--\eqref{HH32}, similarly to \eqref{HH23}, we invoke \eqref{def:PP-1} and \eqref{def:QQ} to obtain 
\begin{align*}
|\mathscr{G}_{I, J}| 
&\lesssim \frac{1}{\rd(I_1, J_1)^{\delta_1}} \prod_{i=1, 3} \mathscr{P}_i(I_i, J_i) |I_i|^{-1} 
\times \mathscr{Q}_2^{t_2}(I_2, J_2) |I_2|^{-1}  
\\ 
&\lesssim \frac{F_1(I_1, K_1)}{\rd(I_1, J_1)^{1+\delta_1}} 
\frac{F_2(I_2, I_2)}{D_{\theta}^{1, 3}(I, J)}
\frac{F_3(I_3, K_3)}{\rd(I_3, J_3)},  
\end{align*}
where 
\begin{align*}
D_{\theta}^{1, 3}(I, J)
:= \bigg[ \frac{\rd(I_1, J_1)}{\rd(I_3, J_3)}  + \frac{\rd(I_3, J_3)}{\rd(I_1, J_1)} \bigg]^{\theta}. 
\end{align*}

\subsection{Separated/Adjacent}\label{sec:SA}
We are going to handle the case $\rd(I_1, J_1) > 1$, $\rd(I_{2, 3}, J_{2, 3}) = 1$, and $I_{2, 3} \cap J_{2, 3} = \emptyset$. As before, it is easy to check that for any $x_1 \in J_1$ and $I_1$, 
\begin{align}\label{HH4}
|x_1-y_1| \simeq \d(I_1, J_1) =: t \quad\text{and}\quad 
D_{\theta}(x-y) 
\simeq \bigg[\frac{t |x_2 - y_2|}{|x_3 - y_3|} 
+ \frac{|x_3 - y_3|}{t |x_2 - y_2|} \bigg]^{-\theta}. 
\end{align}
Writing 
\begin{align*}
\mathbf{K}(x, y) := K(x, y) - K((c_{J_1}, x_{2, 3}), y), 
\end{align*}
we apply the compact full kernel representation, the cancellation of $h_{J_1}$, the mixed size-H\"{o}lder condition of $K$, and \eqref{HH4} to arrive at 
\begin{align}\label{GHH4}
|\G_{IJ}| 
&=\bigg|\int_{I_1 \times I_{2, 3}} \int_{J_1 \times J_{2, 3}}
\mathbf{K}(x, y) h_{I_1}^0 \otimes h_{I_{2, 3}}^0(y) h_{J_1} \otimes h_{J_{2, 3}}(x) \, dx \, dy \bigg|
\\ \nonumber 
&\lesssim \int_{I_1 \times I_{2, 3}} \int_{J_1 \times J_{2, 3}} 
\bigg(\frac{|x_1 - c_{J_1}|}{|x_1 - y_1|}\bigg)^{\delta_1} D_{\theta}(x-y) 
\\ \nonumber
&\quad\times \prod_{i=1}^3 \frac{F_i(x_i, y_i)}{|x_i - y_i|} 
|h_{I_1}^0 \otimes h_{I_{2, 3}}^0(y)| |h_{J_1} \otimes h_{J_{2, 3}}(x)| \, dx \, dy
\\ \nonumber
&\lesssim \frac{1}{\rd(I_1, J_1)^{\delta_1}} \mathscr{P}_1(I_1, J_1) |I_1|^{-1} 
\mathscr{R}_{2, 3}^t(I_{2, 3}, J_{2, 3}) |I_{2, 3}|^{-1}
\\ \nonumber
&\lesssim \frac{F_1(I_1, K_1)}{\rd(I_1, J_1)^{1+\delta_1}} 
\bigg[ \rd(I_1, J_1) + \frac{1}{\rd(I_1, J_1)} \bigg]^{-\theta}
F_2(I_2, I_2) F_3(I_3, I_3), 
\end{align}
where \eqref{def:PP-1} and \eqref{def:R23} were used in the second-to-last inequality. 
\qed

\subsection{Separated/Identical}\label{sec:SI}
Let us deal with the case $\rd(I_1, J_1) > 1$ and $I_{2, 3} = J_{2, 3}$. We rewrite 
\begin{align}\label{HH50}
\G_{IJ}
= \sum_{I'_{2, 3}, \, I''_{2, 3} \in \ch(I_{2, 3})} |I_{2, 3}|^{-\frac12} 
\langle h_{I_{2, 3}} \rangle_{I''_{2, 3}}
\langle T(h_{I_1}^0 \otimes \mathbf{1}_{I'_{2, 3}}), 
h_{J_1} \otimes \mathbf{1}_{I''_{2, 3}} \rangle.
\end{align}
Let us first consider the case $I'_{2, 3}=I''_{2, 3}$. Denote 
\begin{align*}
\mathbf{K}(x_1, y_1) 
:= K_{I'_{2, 3}, I'_{2, 3}}(x_1, y_1) - K_{I'_{2, 3}, I'_{2, 3}}(c_{J_1}, y_1). 
\end{align*}
Then invoking the compact partial kernel representation, the cancellation of $h_{J_1}$, the H\"{o}lder condition of $K_{I'_{2, 3}, I'_{2, 3}}$, and \eqref{def:PP-1}, we obtain  
\begin{align}\label{HH51}
& |\langle T(h_{I_1}^0 \otimes \mathbf{1}_{I'_{2, 3}}), h_{J_1} \otimes \mathbf{1}_{I'_{2, 3}} \rangle| 
 = \bigg|\int_{I_1} \int_{J_1} \mathbf{K}(x_1, y_1) h_{I_1}^0(y_1) h_{J_1}(x_1) \, dx_1 \, dy_1 \bigg|
\\ \nonumber
& \lesssim \frac{\mathscr{P}_1(I_1, J_1)}{\rd(I_1, J_1)^{\delta_1}} |I_1|^{-1}
C(\mathbf{1}_{I'_{2, 3}}, \mathbf{1}_{I'_{2, 3}})
\lesssim \frac{F_1(I_1, K_1)}{\rd(I_1, J_1)^{1 + \delta_1}} F_2(I'_2) |I_2| F_3(I'_3) |I_3|. 
\end{align}
To handle the case $I'_{2, 3} \neq I''_{2, 3}$, observe that $\rd(I'_{2, 3}, I''_{2, 3}) = 1$ and $I'_{2, 3} \cap I''_{2, 3} = \emptyset$, which is similar to the situation of  Section \ref{sec:SA}. Hence, by \eqref{GHH4} and the fact $F_2(I'_2, I'_2) = F_2(I'_2, I'_2)$ and $F_3(I'_3, I'_3) = F_3(I_3, I_3)$, 
\begin{align}\label{HH52}
|\langle T(h_{I_1}^0 \otimes \mathbf{1}_{I'_{2, 3}}), h_{J_1} \otimes \mathbf{1}_{I''_{2, 3}} \rangle| 
\lesssim \frac{F_1(I_1, K_1)}{\rd(I_1, J_1)^{1 + \delta_1}} 
F_2(I_2, I_2) |I_2| \, F_3(I_3, I_3) |I_3|. 
\end{align}
Therefore, it follows from \eqref{HH50}--\eqref{HH52} that 
\begin{align*}
|\G_{IJ}|
\lesssim \frac{F_1(I_1, K_1)}{\rd(I_1, J_1)^{1 + \delta_1}} 
\widehat{F}_2(I_2, I_2) \widehat{F}_3(I_3, I_3). 
\end{align*}
\qed

\subsection{Adjacent/Separated}\label{sec:AS}
Let us turn to the case $\rd(I_1, J_1) = 1$, $I_1 \cap J_1 = \emptyset$, and $\rd(I_{2, 3}, J_{2, 3}) > 1$.

\subsubsection{\bf Case 1: $\rd(I_1, J_1) = 1$, $I_1 \cap J_1 = \emptyset$, $\rd(I_2, J_2) > 1$, and $\rd(I_3, J_3) > 1$}\label{sec:ASS} 
In the current scenario, for each $i=2, 3$ and for all $x_i \in J_i$ and $y_i \in I_i$, we have 
\begin{align}\label{HH61}
|x_i - y_i| \simeq \d(I_i, J_i), \qquad 
\frac{|x_i - c_{J_i}|}{|x_i - y_i|} \le \frac{1/2}{\rd(I_i, J_i)}, 
\end{align}
and 
\begin{align}\label{HH62}
D_{\theta}(x-y) 
&= \Bigg[\frac{\frac{|x_1 - y_1|}{\ell(I_1)} \frac{|x_2 - y_2|}{\ell(I_2)}}{\frac{|x_3 - y_3|}{\ell(I_3)}} 
+ \frac{\frac{|x_3 - y_3|}{\ell(I_3)}}{\frac{|x_1 - y_1|}{\ell(I_1)} \frac{|x_2 - y_2|}{\ell(I_2)}} \Bigg]^{-\theta}
\\ \nonumber
&\simeq \bigg[\frac{|x_1 - y_1|}{\ell(I_1)} \frac{\rd(I_2, J_2)}{\rd(I_3, J_3)}
+ \frac{\ell(I_1)}{|x_1 - y_1|} \frac{\rd(I_3, J_3)}{\rd(I_2, J_2)} \bigg]^{-\theta}. 
\end{align}
Write 
\begin{align*}
& \mathbf{K}(x, y) := K(x, y) - K((x_1, c_{J_2}, c_{J_3}), y), \quad 
t := \frac{\rd(I_2, J_2)}{\ell(I_1) \rd(I_3, J_3)}, 
\\
&\text{and } \quad 
D_{\theta}^{2, 3}(I, J) 
:= \bigg[\frac{\rd(I_2, J_2)}{\rd(I_3, J_3)} + \frac{\rd(I_3, J_3)}{\rd(I_2, J_2)}\bigg]^{\theta}.
\end{align*}
By \eqref{HH61}--\eqref{HH62}, the compact full kernel representation, the cancellation of $h_{J_{2, 3}}$, and the mixed size-H\"{o}lder condition of $K$, we conclude  
\begin{align}\label{GASS}
|\mathscr{G}_{I, J}|
&=\bigg|\int_{I_1 \times I_{2, 3}} \int_{J_1 \times J_{2, 3}}
\mathbf{K}(x, y) h_{I_1}^0 \otimes h_{I_{2, 3}}^0(y) 
h_{J_1} \otimes h_{J_{2, 3}}(x) \, dx \, dy \bigg|
\\ \nonumber 
&\lesssim \int_{I_1 \times I_{2, 3}} \int_{J_1 \times J_{2, 3}} 
\bigg(\frac{|x_2 - c_{J_2}|}{|x_2 - y_2|} 
+ \frac{|x_3 - c_{J_3}|}{|x_3 - y_3|}\bigg)^{\delta_{2, 3}} D_{\theta}(x-y) 
\\ \nonumber 
&\quad\times \prod_{i=1}^3 \frac{F_i(x_i, y_i)}{|x_i - y_i|} 
|h_{I_1}^0 \otimes h_{I_{2, 3}}^0(y)| |h_{J_1} \otimes h_{J_{2, 3}}(x)| \, dx \, dy
\\ \nonumber 
&\lesssim \bigg[\frac{1}{\rd(I_2, J_2)} + \frac{1}{\rd(I_3, J_3)}\bigg]^{\delta_{2, 3}} 
\mathscr{Q}_1^t(I_1, J_1) |I_1|^{-1} \prod_{i=2}^3 \mathscr{P}_i(I_i, J_i) |I_i|^{-1}
\\ \nonumber 
&\lesssim \frac{\min\{\rd(I_2, J_2), \, \rd(I_3, J_3)\}^{-\delta_{2, 3}}}{D_{\theta}^{2, 3}(I, J)} 
F_1(I_1, I_1) \prod_{i=2}^3 \frac{F_i(I_i, K_i)}{\rd(I_i, J_i)}, 
\end{align}
where we have used \eqref{def:PP-1} and \eqref{def:QQ} in the last inequality. 
\qed

\subsubsection{\bf Case 2: $\rd(I_1, J_1) = 1$, $I_1 \cap J_1 = \emptyset$, $\rd(I_2, J_2) > 1$, and $\rd(I_3, J_3) = 1$} \label{sec:ASE}
Note that for any $x_2 \in J_2$ and $y_2 \in I_2$, 
\begin{align}\label{HH7}
|x_2 - y_2| \simeq \d(I_2, J_2) =: t 
\quad\text{ and }\quad 
D_{\theta}(x-y) 
\simeq \bigg[\frac{t |x_1 - y_1|}{|x_3 - y_3|} 
+ \frac{|x_3 - y_3|}{t |x_1 - y_1|} \bigg]^{-\theta}. 
\end{align}
Then in view of the compact full kernel representation and the size condition of $K$, the estimate \eqref{HH7} yields 
\begin{align}\label{GASE}
|\mathscr{G}_{I, J}| 
&=\bigg|\int_{I_1 \times I_{2, 3}} \int_{J_1 \times J_{2, 3}}
K(x, y) h_{I_1}^0 \otimes h_{I_{2, 3}}^0(y) 
h_{J_1} \otimes h_{J_{2, 3}}(x) \, dx \, dy \bigg|
\\ \nonumber 
&\lesssim |I_1|^{-1} |I_2|^{-1} |I_3|^{-1} 
\int_{I_1 \times I_{2, 3}} \int_{J_1 \times J_{2, 3}} D_{\theta}(x-y) 
\prod_{i=1}^3 \frac{F_i(x_i, y_i)}{|x_i - y_i|} \, dx \, dy
\\ \nonumber 
&\lesssim  \mathscr{P}_2(I_2, J_2) |I_2|^{-1}
\mathscr{R}_{1, 3}^t(I_{1, 3}, J_{1, 3}) |I_1|^{-1} |I_3|^{-1} 
\\ \nonumber 
&\lesssim F_1(I_1, I_1) 
\frac{F_2(I_2, K_2)}{\rd(I_2, J_2)} 
\bigg[ \rd(I_2, J_2) + \frac{1}{\rd(I_2, J_2)} \bigg]^{-\theta}
F_3(I_3, I_3), 
\end{align}
provided \eqref{def:PP-1}, \eqref{def:R13}, and that $\ell(I_1) \d(I_2, J_2)/\ell(I_3) =\rd(I_2, J_2)^{-1}$.

\subsubsection{\bf Case 3: $\rd(I_1, J_1) = 1$, $I_1 \cap J_1 = \emptyset$, $\rd(I_2, J_2) = 1$, and $\rd(I_3, J_3) > 1$} \label{sec:AIS}
The condition $\rd(I_3, J_3) > 1$ implies that for all $x_3 \in J_3$ and $y_3 \in I_3$, 
\begin{align}\label{HH8}
|x_3 - y_3| \simeq \d(I_3, J_3) =: t 
\quad\text{ and }\quad
D_{\theta}(x-y) 
\simeq D_{\theta}(x_{1, 2} - y_{1, 2}, t) . 
\end{align}
Then it follows from the compact full kernel representation, the size condition of $K$, and \eqref{HH8} that 
\begin{align}\label{GAES}
|\mathscr{G}_{I, J}| 
&=\bigg|\int_{I_1 \times I_{2, 3}} \int_{J_1 \times J_{2, 3}}
K(x, y) h_{I_1}^0 \otimes h_{I_{2, 3}}^0(y) h_{J_1} \otimes h_{J_{2, 3}}(x) \, dx \, dy \bigg|
\\ \nonumber 
&\lesssim |I_1|^{-1} |I_2|^{-1} |I_3|^{-1} \int_{I_1 \times I_{2, 3}} \int_{J_1 \times J_{2, 3}} 
D_{\theta}(x-y) \prod_{i=1}^3 \frac{F_i(x_i, y_i)}{|x_i - y_i|} \, dx \, dy
\\ \nonumber 
&\lesssim \mathscr{R}_{1, 2}^t(I_{1, 2}, J_{1, 2}) |I_1|^{-1} |I_2|^{-1}
\mathscr{P}_3(I_3, J_3) |I_3|^{-1}
\\ \nonumber 
&\lesssim F_1(I_1, I_1) F_2(I_2, I_2) 
\frac{F_3(I_3, K_3)}{\rd(I_3, J_3)}
\bigg[ \rd(I_3, J_3) + \frac{1}{\rd(I_3, J_3)} \bigg]^{-\theta},  
\end{align}
provided \eqref{def:PP-1}, \eqref{def:R12}, and that $\ell(I_1) \ell(I_2)/\d(I_3, J_3) = \rd(I_3, J_3)^{-1}$.
\qed

\subsection{Adjacent/Adjacent}\label{sec:AA}
In the case $\rd(I_1, J_1) = 1$, $I_1 \cap J_1 = \emptyset$, $\rd(I_{2, 3}, J_{2, 3}) = 1$, and $I_{2, 3} \cap J_{2, 3} = \emptyset$, we apply the compact full kernel representation, the size condition of $K$, and \eqref{def:RIJ} to deduce  
\begin{align}\label{eq:HH-9}
|\mathscr{G}_{I, J}| 
\lesssim \mathscr{R}(I, J) \, |I_1|^{-1} |I_2|^{-1} |I_3|^{-1} 
\lesssim \prod_{i=1}^3 F_i(I_i, I_i). 
\end{align}

\subsection{Adjacent/Identical}\label{sec:AI}
In the case $\rd(I_1, J_1) = 1$, $I_1 \cap J_1 = \emptyset$, and $I_{2, 3} = J_{2, 3}$, we perform the decomposition
\begin{align}\label{HH100}
\mathscr{G}_{I, J}
= \sum_{I'_{2, 3}, I''_{2, 3} \in \ch(I_{2, 3})} 
|I_{2, 3}|^{-\frac12} \langle h_{I_{2, 3}} \rangle_{I''_{2, 3}}
\langle T(h_{I_1}^0 \otimes \mathbf{1}_{I'_{2, 3}}), h_{J_1} \otimes \mathbf{1}_{I''_{2, 3}} \rangle. 
\end{align}
If $I'_{2, 3}=I''_{2, 3}$, the the compact partial kernel representation, the size condition of $K_{\mathbf{1}_{I'_{2, 3}}, \mathbf{1}_{I'_{2, 3}}}$, and \eqref{def:PP-2} imply 
\begin{align}\label{HH101}
|\langle T(h_{I_1}^0 \otimes \mathbf{1}_{I'_{2, 3}}), h_{J_1} \otimes \mathbf{1}_{I''_{2, 3}} \rangle|
& \lesssim \mathscr{P}_1(I_1, J_1) |I_1|^{-1} 
C(\mathbf{1}_{I'_{2, 3}}, \mathbf{1}_{I'_{2, 3}}) 
\\ \nonumber
& \lesssim F_1(I_1, I_1) F_2(I'_2) F_3(I'_3) \, |I_2| |I_3|. 
\end{align}
If $I'_{2, 3} \neq I''_{2, 3}$, we see that $\rd(I_1, J_1) = 1$, $I_1 \cap J_1 = \emptyset$, $\rd(I'_{2, 3}, I''_{2, 3}) = 1$, and $I'_{2, 3} \neq I''_{2, 3} = \emptyset$. This is similar to the situation in Section \ref{sec:AA}, which gives 
\begin{align}\label{HH102}
|\langle T(h_{I_1}^0 \otimes \mathbf{1}_{I'_{2, 3}}), h_{J_1} \otimes \mathbf{1}_{I''_{2, 3}} \rangle|
&\lesssim F_1(I_1, I_1) F_2(I'_2, I'_2) F_3(I'_3, I'_3) \, |I_2| |I_3|. 
\end{align}
Thus, by \eqref{HH100}--\eqref{HH102} and the fact $F_i(I'_i, I'_i) \simeq F_2(I_i, I_i)$, $i=2, 3$, 
\begin{align}\label{HH10}
|\mathscr{G}_{I, J}| 
\lesssim F_1(I_1, I_1) \widehat{F}_2(I_2, I_2) \widehat{F}_3(I_3, I_3). 
\end{align}

\subsection{Identical/Separated}\label{sec:IS}
In the case $I_1 = J_1$ and $\rd(I_{2, 3}, J_{2, 3}) > 1$, we rewrite 
\begin{align}\label{GES}
\G_{IJ}
= \sum_{I'_1, \, I''_1 \in \ch(I_1)} |I_1|^{-\frac12} \langle h_{I_1} \rangle_{I''_1}
\langle T(\mathbf{1}_{I'_1} \otimes h_{I_{2, 3}}^0), \mathbf{1}_{I''_1} \otimes h_{J_{2, 3}} \rangle.
\end{align}

\subsubsection{\bf Case 1: $I_1 = J_1$, $\rd(I_2, J_2) > 1$, and $\rd(I_3, J_3) > 1$} 
If $I'_1 \neq I''_1$, then $\rd(I'_1, I''_1) = 1$, $I'_1 \cap I''_1 = \emptyset$, $\rd(I_2, J_2) > 1$, and $\rd(I_3, J_3) > 1$. This is similar to the case of Section \ref{sec:ASS}. Hence, by \eqref{GASS} and the fact that $F_1(I'_1, I'_1) \simeq F_1(I_1, I_1)$, 
\begin{align}\label{HH111}
|\langle T(\mathbf{1}_{I'_1} \otimes h_{I_{2, 3}}^0), \mathbf{1}_{I''_1} \otimes h_{J_{2, 3}} \rangle|
\lesssim \frac{\min\{\rd(I_2, J_2), \rd(I_3, J_3)\}^{-\delta_{2, 3}}}{D_{\theta}^{2, 3}(I, J)} 
F_1(I_1, I_1) \prod_{i=2}^3 \frac{F_i(I_i, K_i)}{\rd(I_i, J_i)}. 
\end{align}
To handle the case $I'_1=I''_1$, observe that for all $x_{2, 3} \in J_{2, 3}$ and $y_{2, 3} \in I_{2, 3}$, 
\begin{align}\label{HH112}
|x_i - y_i| \simeq \d(I_i, J_i), \qquad 
\frac{|x_i - c_{J_i}|}{|x_i - y_i|} \le \frac{1/2}{\rd(I_i, J_i)}, \quad i=2, 3, 
\end{align}
and 
\begin{align}\label{HH113}
D_{\theta}(\ell(I_1), x_{2, 3} - y_{2, 3}) 
\simeq \bigg[\frac{\rd(I_2, J_2)}{\rd(I_3, J_3)} 
+ \frac{\rd(I_3, J_3)}{\rd(I_2, J_2)} \bigg]^{-\theta}
=D_{\theta}^{2, 3}(I, J)^{-1}. 
\end{align}
If we let 
\begin{align*}
\mathbf{K}(x_{2, 3}, y_{2, 3})
:= K_{\mathbf{1}_{I'_1}, \mathbf{1}_{I'_1}}(x_{2, 3}, y_{2, 3}) 
- K_{\mathbf{1}_{I'_1}, \mathbf{1}_{I'_1}}(c_{J_{2, 3}}, y_{2, 3}), 
\end{align*}
then the compact partial kernel representation, the cancellation of $h_{J_{2, 3}}$,  the H\"{o}lder condition of $K_{\mathbf{1}_{I'_1}, \mathbf{1}_{I'_1}}$, and \eqref{HH112}--\eqref{HH113} imply 
\begin{align}\label{HH114}
& |\langle T(\mathbf{1}_{I'_1} \otimes h_{I_{2, 3}}^0), \mathbf{1}_{I'_1} \otimes h_{J_{2, 3}} \rangle|
\\ \nonumber
& = \bigg|\int_{I_{2, 3}} \int_{J_{2, 3}} \mathbf{K}(x_{2, 3}, y_{2, 3}) 
h_{I_{2, 3}}^0(y_{2, 3}) h_{J_{2, 3}}(x_{2, 3}) \, dx_{2, 3} \, dy_{2, 3}\bigg|
\\ \nonumber
& \lesssim C(\mathbf{1}_{I'_1}, \mathbf{1}_{I'_1}) |I_2|^{-1} |I_3|^{-1} \int_{I_{2, 3}} \int_{J_{2, 3}} 
\bigg(\frac{|y_2 - c_{I_2}|}{|x_2 - y_2|} + \frac{|y_3 - c_{I_3}|}{|x_3 - y_3|}\bigg)^{\delta_{2, 3}} 
\\ \nonumber
& \quad \times D_{\theta}(\ell(I_1), x_{2, 3} - y_{2, 3}) 
\prod_{i=2}^3 \frac{F_i(x_i, y_i)}{|x_i - y_i|} \, dx_{2, 3} \, dy_{2, 3} 
\\ \nonumber
&\lesssim F_1(I'_1) |I'_1| \, \bigg(\frac{1}{\rd(I_2, J_2)} + \frac{1}{\rd(I_3, J_3)}\bigg)^{\delta_{2, 3}} 
\frac{1}{D_{\theta}^{2, 3}(I, J)} \prod_{i=2}^3 \mathscr{P}_i(I_i, J_i) |I_i|^{-1} 
\\ \nonumber
&\lesssim F_1(I'_1) |I_1| 
\frac{\min\{\rd(I_2, J_2), \rd(I_3, J_3)\}^{-\delta_{2, 3}}}{D_{\theta}^{2, 3}(I, J)} 
\prod_{i=2}^3 \frac{F_i(I_i, K_i)}{\rd(I_i, J_i)}, 
\end{align}
where \eqref{def:PP-1} was used in the last inequality. Consequently, it follows from \eqref{GES}, \eqref{HH111}, and \eqref{HH114} that 
\begin{align}\label{HH11}
|\G_{IJ}|
\lesssim \frac{\min\{\rd(I_2, J_2), \rd(I_3, J_3)\}^{-\delta_{2, 3}}}{D_{\theta}^{2, 3}(I, J)}  
\widehat{F}_1(I_1, I_1) \prod_{i=2}^3 \frac{F_i(I_i, K_i)}{\rd(I_i, J_i)}. 
\end{align}

\subsubsection{\bf Case 2: $I_1 = J_1$, $\rd(I_2, J_2) > 1$, and $\rd(I_3, J_3) = 1$} 
If $I'_1 \neq I''_1$, then $\rd(I'_1, I''_1) = 1$, $I'_1 \cap I''_1 = \emptyset$, $\rd(I_2, J_2) > 1$, and $\rd(I_3, J_3) = 1$, which is similar to the case of Section \ref{sec:ASE}. Thus, \eqref{GASE} and the fact $F_1(I'_1, I'_1) \simeq F_1(I_1, I_1)$ give 
\begin{align}\label{HH121}
& |\langle T(\mathbf{1}_{I'_1} \otimes h_{I_{2, 3}}^0), \mathbf{1}_{I''_1} \otimes h_{J_{2, 3}} \rangle|
\\ \nonumber 
&\lesssim F_1(I_1, I_1) |I_1| 
\frac{F_2(I_2, K_2)}{\rd(I_2, J_2)} 
\bigg[ \rd(I_2, J_2) + \frac{1}{\rd(I_2, J_2)} \bigg]^{-\theta}
F_3(I_3, I_3).
\end{align}
To proceed, note that for all $x_2 \in J_2$ and $y_2 \in I_2$, 
\begin{align}\label{HH122}
D_{\theta}(\ell(I_1), x_{2, 3} - y_{2, 3}) 
\simeq \bigg[t|x_3 - y_3| + \frac{1}{t|x_3 - y_3|}\bigg]^{-\theta}, 
\end{align}
where $t := \frac{1}{\ell(I_1) \d(I_2, J_2)}$. Then, together with \eqref{HH122}, the compact partial kernel representation and the size condition of $K_{\mathbf{1}_{I'_1}, \mathbf{1}_{I'_1}}$ imply 
\begin{align}\label{HH123}
& |\langle T(\mathbf{1}_{I'_1} \otimes h_{I_{2, 3}}^0), \mathbf{1}_{I'_1} \otimes h_{J_{2, 3}} \rangle|
\\ \nonumber
&\lesssim C(\mathbf{1}_{I'_1}, \mathbf{1}_{I'_1}) |I_{2, 3}|^{-1} 
\int_{I_{2, 3}} \int_{J_{2, 3}} D_{\theta}(\ell(I_1), x_{2, 3} - y_{2, 3}) 
\prod_{i=2}^3 \frac{F_i(x_i, y_i)}{|x_i - y_i|} \, dx_{2, 3} \, dy_{2, 3} 
\\ \nonumber
&\lesssim F_1(I'_1) |I'_1| \, \mathscr{P}_2(I_2, J_2) |I_2|^{-1} \mathscr{Q}_3^t(I_3, J_3) |I_3|^{-1} 
\\ \nonumber
&\lesssim F_1(I'_1) |I_1| \, \frac{F_2(I_2, K_2)}{\rd(I_2, J_2)}   
\bigg[ \rd(I_2, J_2) + \frac{1}{\rd(I_2, J_2)} \bigg]^{-\theta}
F_3(I_3, I_3),
\end{align}
where \eqref{def:PP-1}, \eqref{def:QQ}, and $\frac{\ell(I_3)}{\ell(I_1) \d(I_2, J_2)} = \frac{1}{\rd(I_2, J_2)}$ were used in the last step. Hence, a consequence of \eqref{GES}, \eqref{HH121}, and \eqref{HH123} is that 
\begin{align}\label{HH12}
|\G_{IJ}|
\lesssim \widehat{F}_1(I_1, I_1)
\frac{F_2(I_2, K_2)}{\rd(I_2, J_2)}
\bigg[ \rd(I_2, J_2) + \frac{1}{\rd(I_2, J_2)} \bigg]^{-\theta}
F_3(I_3, I_3). 
\end{align}

\subsubsection{\bf Case 3: $I_1 = J_1$, $\rd(I_2, J_2) = 1$, and $\rd(I_3, J_3) > 1$} 
If $I'_1 \neq I''_1$, then $\rd(I'_1, I''_1) = 1$, $I'_1 \cap I''_1 = \emptyset$, $\rd(I_2, J_2) = 1$, and $\rd(I_3, J_3) > 1$. This is similar to the case of Section \ref{sec:AIS}, which along with \eqref{GAES} and the fact $F_1(I'_1, I'_1) \simeq F_1(I_1, I_1)$ yields    
\begin{align}\label{HH131}
& |\langle T(\mathbf{1}_{I'_1} \otimes h_{I_{2, 3}}^0), \mathbf{1}_{I''_1} \otimes h_{J_{2, 3}} \rangle|
\\ \nonumber 
&\lesssim F_1(I_1, I_1) |I_1| \, F_2(I_2, I_2)
\frac{F_3(I_3, K_3)}{\rd(I_3, J_3)}
\bigg[ \rd(I_3, J_3) + \frac{1}{\rd(I_3, J_3)} \bigg]^{-\theta}. 
\end{align}
Set $t := \frac{\ell(I_1)}{\d(I_3, J_3)}$. Analogously to \eqref{HH122} and \eqref{HH123}, we have 
\begin{align*}
D_{\theta}(\ell(I_1), x_{2, 3} - y_{2, 3}) 
\simeq \bigg(t|x_2 - y_2| + \frac{1}{t|x_2 - y_2|}\bigg)^{-\theta}, 
\qquad x_2 \in J_2, \, y_2 \in I_2, 
\end{align*}
and 
\begin{align}\label{HH132}
|\langle T(\mathbf{1}_{I'_1} \otimes h_{I_{2, 3}}^0), \mathbf{1}_{I'_1} \otimes h_{J_{2, 3}} \rangle|
\lesssim F_1(I'_1) |I'_1| \mathscr{Q}_2^t(I_2, J_2) |I_2|^{-1} \, \mathscr{P}_3(I_3, J_3) |I_3|^{-1}
\\ \nonumber
\lesssim F_1(I'_1) |I_1| \, F_2(I_2, I_2) \frac{F_3(I_3, I_3)}{\rd(I_3, K_3)} 
\bigg[ \rd(I_3, J_3) + \frac{1}{\rd(I_3, J_3)} \bigg]^{-\theta}, 
\end{align}
provided \eqref{def:PP-1}, \eqref{def:QQ}, and that $\frac{\ell(I_1) \ell(I_2)}{\d(I_3, J_3)} = \frac{1}{\rd(I_3, J_3)}$. Accordingly, \eqref{HH131} and \eqref{HH132} lead to 
\begin{align}\label{HH13}
|\G_{IJ}|
\lesssim \widehat{F}_1(I_1, I_1) F_2(I_2, I_2)
\frac{F_3(I_3, I_3)}{\rd(I_3, K_3)}
\bigg[ \rd(I_3, J_3) + \frac{1}{\rd(I_3, J_3)} \bigg]^{-\theta}. 
\end{align}

\subsection{Identical/Adjacent}\label{sec:IA}
In the case $I_1 = J_1$, $\rd(I_{2, 3}, J_{2, 3}) = 1$, and $I_{2, 3} \cap J_{2, 3} = \emptyset$, write 
\begin{align}\label{HH141}
\G_{IJ}
= \sum_{I'_1, I''_1 \in \ch(I_1)} 
|I_1|^{-\frac12} \langle h_{I_1} \rangle_{I''_1}
\langle T(\mathbf{1}_{I'_1} \otimes h_{I_{2, 3}}^0), \mathbf{1}_{I''_1} \otimes h_{J_{2, 3}} \rangle.
\end{align}
If $I'_1 \neq I''_1$, then $\rd(I'_1, I''_1) = 1$, $I'_1 \cap I''_1 =\emptyset$, $\rd(I_{2, 3}, J_{2, 3}) = 1$, and $I_{2, 3} \cap J_{2, 3} = \emptyset$. This is similar to the case in Section \ref{sec:AA}. Thus, there holds 
\begin{align}\label{HH142}
|\langle T(\mathbf{1}_{I'_1} \otimes h_{I_{2, 3}}^0), \mathbf{1}_{I''_1} \otimes h_{J_{2, 3}} \rangle|
&\lesssim F_1(I'_1, I'_1) |I'_1| \, F_2(I_2, I_2) F_3(I_3, I_3)
\\ \nonumber
&\lesssim F_1(I_1, I_1) |I_1| \, F_2(I_2, I_2) F_3(I_3, I_3). 
\end{align}
If $I'_1 = I''_1$, then the compact partial kernel representation, the size condition of $K_{\mathbf{1}_{I'_1}, \mathbf{1}_{I'_1}}$, and \eqref{def:R23} yield 
\begin{align}\label{HH143}
|\langle T(\mathbf{1}_{I'_1} \otimes h_{I_{2, 3}}^0), \mathbf{1}_{I''_1} \otimes h_{J_{2, 3}} \rangle|
& \lesssim C(\mathbf{1}_{I'_1}, \mathbf{1}_{I'_1}) 
\mathscr{R}_{2, 3}^{\ell(I_1)}(I_{2, 3}, J_{2, 3}) |I_2|^{-1} |I_3|^{-1} 
\\ \nonumber
&\lesssim F_1(I'_1) |I_1| \, F_2(I_2, I_2) F_3(I_3, I_3), 
\end{align}
Thus, from \eqref{HH141}--\eqref{HH143}, we obtain 
\begin{align}\label{HH14}
|\G_{IJ}|
\lesssim \widehat{F}_1(I_1, I_1) F_2(I_2, I_2) F_3(I_3, I_3). 
\end{align}

\subsection{Identical/Identical}\label{sec:II}
In the case $I_1=J_1$ and $I_{2, 3} = J_{2, 3}$, we rewrite 
\begin{align}\label{HH151}
\G_{IJ}
&= \sum_{\substack{I'_1, I''_1 \in \ch(I_1) \\ I'_{2, 3}, \, I''_{2, 3} \in \ch(I_{2, 3})}}
|I_1|^{-\frac12} |I_{2, 3}|^{-\frac12} 
\langle h_{I_1} \rangle_{I''_1}
\langle h_{I_{2, 3}} \rangle_{I''_{2, 3}}
\langle T(\mathbf{1}_{I'_1} \otimes \mathbf{1}_{I'_{2, 3}}), 
\mathbf{1}_{I''_1} \otimes \mathbf{1}_{I''_{2, 3}} \rangle.
\end{align}
If $I'_1=I''_1$ and $I'_{2, 3}=I''_{2, 3}$, the weak compactness property gives 
\begin{align}\label{HH152}
|\langle T(\mathbf{1}_{I'_1} \otimes \mathbf{1}_{I'_{2, 3}}), 
\mathbf{1}_{I'_1} \otimes \mathbf{1}_{I'_{2, 3}} \rangle|
\lesssim F_1(I'_1) |I_1| \, F_2(I'_2) |I_2| \, F_3(I'_3) |I_3|.  
\end{align}
If $I'_1=I''_1$ and $I'_{2, 3} \neq I''_{2, 3}$, we use the compact partial kernel representation, the size condition of $K_{\mathbf{1}_{I'_1}, \mathbf{1}_{I'_1}}$, and \eqref{def:R23} to arrive at
\begin{align}\label{HH153}
|\langle T(\mathbf{1}_{I'_1} \otimes \mathbf{1}_{I'_{2, 3}}),
\mathbf{1}_{I'_1} \otimes \mathbf{1}_{I''_{2, 3}} \rangle| 
& \lesssim C(\mathbf{1}_{I'_1},  \mathbf{1}_{I'_1}) 
\mathscr{R}_{2, 3}^{\ell(I'_1)}(I'_{2, 3}, I''_{2, 3})
\\ \nonumber
& \lesssim F_1(I'_1) |I_1| \, F_2(I'_2, I'_2) |I_2| F_3(I'_3, I'_3) |I_3|. 
\end{align}
Similarly, if $I'_1 \neq I''_1$ and $I'_{2, 3}=I''_{2, 3}$, it follows from \eqref{def:PP-2} that  
\begin{align}\label{HH154}
|\langle T (\mathbf{1}_{I'_1} \otimes \mathbf{1}_{I'_{2, 3}}),
\mathbf{1}_{I''_1} \otimes \mathbf{1}_{I'_{2, 3}} \rangle|
& \lesssim C(\mathbf{1}_{I'_{2, 3}}, \mathbf{1}_{I'_{2, 3}}) \mathscr{P}_1(I'_1, I''_1) 
\\ \nonumber
&\lesssim F_1(I'_1, I'_1) |I_1| \, F_2(I'_2) |I_2| \, F_3(I'_3) |I_3|.
\end{align}
If $I'_1 \neq I''_1$ and $I'_{2, 3} \neq I''_{2, 3}$, then in light of \eqref{def:RIJ}, the compact full kernel representation and the size condition of $K$ imply
\begin{align}\label{HH155}
|\langle T (\mathbf{1}_{I'_1} \otimes \mathbf{1}_{I'_{2, 3}}),
\mathbf{1}_{I''_1} \otimes \mathbf{1}_{I''_{2, 3}} \rangle|
\lesssim \mathscr{R}(I', J') 
\lesssim \prod_{i=1}^3 F_i(I'_i, I'_i) |I_i|. 
\end{align}
Therefore, by \eqref{HH151}--\eqref{HH155} and that $F_i(I'_i, I'_i) \simeq F_i(I_i, I_i)$, $i=1, 2, 3$, there holds 
\begin{align}\label{HH15}
|\G_{IJ}|
\lesssim \widehat{F}_1(I_1, I_1) \widehat{F}_2(I_2, I_2) \widehat{F}_3(I_3, I_3). 
\end{align}
The proof is complete. 
\qed

\section{Weighted compactness with Zygmund dilations}\label{sec:Zd}
We would like to demonstrate Theorem \ref{thm:cpt}, which is based on Theorems \ref{thm:SD-cpt} and \ref{thm:RdF-cpt}.

\subsection{Kolmogorov--Riesz theorems}
Let us present some compactness criterions. Define $\tau_v f(x) := f(x-v)$ for all $x, v \in \R^3$. Given $r>0$ and $x \in \R^3$, let 
\begin{align*}
B(x, r) & := \{y \in \R^3: |y-x| = |y_1 - x_1| + |y_2 - x_2| + |y_3 - x_3| < r\}, 
\\ 
R(x, r) & := \{y \in \R^3: |y_1 - x_1| < r, \, |y_2 - x_2| < r, \, |y_3 - x_3| < r^2\}.
\end{align*} 

\begin{theorem}\label{thm:KRLp}
Let $p \in (0, \infty)$ and $\mathcal{K} \subset L^p(\R^3)$. Then $\mathcal{K}$ is precompact in $L^p(\R^3)$ if and only if the following are satisfied:
\begin{list}{\rm (\theenumi)}{\usecounter{enumi}\leftmargin=1.2cm \labelwidth=1cm \itemsep=0.2cm \topsep=0.2cm \renewcommand{\theenumi}{\alph{enumi}}}
 
\item[{\rm (a)}] ${\displaystyle \sup_{f \in \mathcal{K}} \|f\|_{L^p} < \infty}$, 

\item[{\rm (b)}] ${\displaystyle \lim_{A \to \infty} \sup_{f \in \mathcal{K}} 
\|f \mathbf{1}_{B(0, A)^c}\|_{L^p}=0}$, 

\item[{\rm (c)}] $\displaystyle \lim_{|v| \to 0} \sup_{f \in \mathcal{K}} 
\|\tau_v f - f \|_{L^p}=0$. 
\end{list}  
\end{theorem}

\begin{proof}
The proof is almost the same as that of \cite[Theorem 1.3]{CLSY}. 
\end{proof}

\begin{theorem}\label{thm:KRB}
Let $p \in (1, \infty)$, $w \in A_{p, \Z}$,  and $\mathcal{K} \subset L^p(w)$. Then $\mathcal{K}$ is precompact in $L^p(w)$ if and only if the following are satisfied:
\begin{list}{\rm (\theenumi)}{\usecounter{enumi}\leftmargin=1.2cm \labelwidth=1cm \itemsep=0.2cm \topsep=0.2cm \renewcommand{\theenumi}{\alph{enumi}}}
			
\item[\textup{(a)\phantom{$'$}}] ${\displaystyle \sup_{f \in \mathcal{K}} \|f\|_{L^p(w)} < \infty}$, 

\item[\textup{(b)\phantom{$'$}}] ${\displaystyle \lim_{A \to \infty} \sup_{f \in \mathcal{K}} 
\|f \mathbf{1}_{R(0, A)^c}\|_{L^p(w)}=0}$, 

\item[\textup{(c)$'$}] ${\displaystyle \lim_{r \to 0} \sup_{f \in \mathcal{K}} 
\|f - \langle f \rangle_{R(\cdot, r)}\|_{L^p(w)}=0}$. 
\end{list}
\end{theorem}

\begin{proof}
The proof is similar to that of \cite[Theorem 4.1]{CLSY}. We omit the details. 
\end{proof}

\subsection{Interpolation of compactness}
We are going to utilize Theorem \ref{thm:KRB} to establish an interpolation of compactness as follows.

\begin{theorem}\label{thm:inter}
Let $\mathcal{B} \in \{\mathcal{R}, \mathcal{Z}\}$. Let $p_i \in (1, \infty)$ and $w_i \in A_{p_i, \mathcal{B}}$, $i=0, 1$. Assume that $T$ is a linear operator such that $T$ is bounded on $L^{p_0}(w_0)$ and compact on $L^{p_1}(w_1)$. Then $T$ is compact on $L^p(w)$, where 
\begin{align}\label{ITpp}
\frac1p = \frac{1-\eta}{p_0} + \frac{\eta}{p_1}, \quad 
w^{\frac1p} = w_0^{\frac{1-\eta}{p_0}}  w_1^{\frac{\eta}{p_1}}, 
\quad\text{and}\quad 
0<\eta<1.  
\end{align}
\end{theorem}

\begin{proof}
It suffices to show the case $\mathcal{B} = \mathcal{Z}$, since the proof in the case $\mathcal{B} = \mathcal{R}$ is similar. By the boundedness of $T$ on $L^{p_0}(w_0)$, there exists a constant $C_0 \in (1, \infty)$ so that
\begin{align}
\label{IT-1}
& \sup_{\|f\|_{L^{p_0}(w_0)} \le 1} \|Tf\|_{L^{p_0}(w_0)} \le C_0, 
\\
\label{IT-2}
& \sup_{\|f\|_{L^{p_0}(w_0)} \le 1} \|Tf \, \mathbf{1}_{R(0, A)^c}\|_{L^{p_0}(w_0)} \le C_0, 
\end{align}
for all $A>0$. In addition, the inequality \eqref{MB} gives 
\begin{align*}
\|Tf - \langle Tf \rangle_{R(\cdot, r)}\|_{L^{p_0}(w_0)} 
& \le \|Tf\|_{L^{p_0}(w_0)} + \|M_{\Z} (Tf)\|_{L^{p_0}(w_0)} 
\\
& \lesssim \|Tf\|_{L^{p_0}(w_0)} 
\lesssim \|f\|_{L^{p_0}(w_0)}, 
\end{align*}
for all $r>0$. Thus, there exists a constant $C_1 \in (1, \infty)$ such that 
\begin{align}\label{IT-3}
& \sup_{\|f\|_{L^{p_0}(w_0)} \le 1} \|Tf - \langle Tf \rangle_{R(\cdot, r)}\|_{L^{p_0}(w_0)} 
\le C_1, \quad\forall r>0. 
\end{align}

Let $\varepsilon>0$ be an arbitrary number. Since $T$ is compact on $L^{p_1}(w_1)$, Theorem \ref{thm:KRB} implies that there exist $A_0 = A_0(\varepsilon) > 0$ and $\rho_0 = \rho_0(\varepsilon) > 0$ such that for all $A \ge A_0$ and $0<\rho<\rho_0$, 
\begin{align}
\label{KAM-1}
\sup_{\|f\|_{L^{p_1}(w_1)} \le 1} 
& \|Tf\|_{L^{p_1}(w_1)} \le C_2, 
\\
\label{KAM-2}
\sup_{\|f\|_{L^{p_1}(w_1)} \le 1} 
& \|Tf \, \mathbf{1}_{R(0, A)^c}\|_{L^{p_1}(w_1)} \le \varepsilon, 
\\
\label{KAM-3}
\sup_{\|f\|_{L^{p_1}(w_1)} \le 1} 
& \|Tf - \langle Tf \rangle_{R(\cdot, r)}\|_{L^{p_1}(w_1)} \le \varepsilon,  
\end{align}
where $C_2>0$ is an absolute constant. 

Hence, in view of \eqref{ITpp}, \eqref{IT-1}, and \eqref{KAM-1}, we use the interpolation theorem of Stein-Weiss (cf. \cite{SW}) to obtain 
\begin{align}\label{KAM-111}
\sup_{\|f\|_{L^p(w)} \le 1} \|T f\|_{L^p(w)} 
\le C_0^{1-\eta} C_2^{\eta}.  
\end{align}
Likewise, \eqref{IT-2} and \eqref{KAM-2} lead to 
\begin{align}\label{KAM-222}
\sup_{\|f\|_{L^p(w)} \le 1} 
& \|Tf \, \mathbf{1}_{R(0, A)^c}\|_{L^p(w)} 
\le C_0^{1-\eta} \varepsilon^{\eta}, \quad \forall A \ge A_0.
\end{align}
Besides, \eqref{IT-3} and \eqref{KAM-3} imply 
\begin{align}\label{KAM-333}
\sup_{\|f\|_{L^p(w)} \le 1} 
& \|Tf - \langle Tf \rangle_{R(\cdot, r)}\|_{L^p(w)} 
\le C_1^{1-\eta} \varepsilon^{\eta}, \quad \forall \rho \in (0, \rho_0). 
\end{align} 
Therefore, from \eqref{KAM-111}--\eqref{KAM-333} and Theorem \ref{thm:KRB}, we conclude that $T$ is compact on $L^p(w)$. 
\end{proof}

\subsection{Extrapolation of compactness}
We only present the proof of Theorem \ref{thm:RdF-cpt} in the case $\mathcal{B} = \mathcal{Z}$. Fix $p \in (1, \infty)$, $w \in A_{p, \Z}$, $p_0 \in [1, \infty)$, and $w_0 \in A_{p_0, \Z}$. We claim that there exist $\eta \in (0, 1)$, $p_1 = p_1(\eta) \in (1, \infty)$, and $w_1 = w_1(\eta) \in A_{p_1, \Z}$ such that 
\begin{align}\label{wpp}
\frac1p = \frac{1 - \eta}{p_0} + \frac{\eta}{p_1}
\qquad\text{ and }\qquad 
w^{\frac1p} = w_0^{\frac{1 - \eta}{p_0}} w_1^{\frac{\eta}{p_1}}. 
\end{align}
Indeed, \eqref{wpp} can be shown as in the proof of \cite[Lemma 4.1]{COY}. Although the latter is given for cubes, the same argument still holds for Zygmund rectangles because the proof heavily depends on the reverse H\"{o}lder inequality for $A_{p, \Z}$ weights, which is true. In fact, given $p \in (1, \infty)$, $w \in A_{p, \Z}$, and a Zygmund rectangle $I=I_1 \times I_{2, 3}$, we have $[w(\cdot, x_2, x_3)]_{A_p(\R)} \le [w]_{A_{p, \Z}}$ uniformly on $x_2, x_3 \in \R$ and $[\langle w \rangle_{I_1}]_{A_p(\R^2)} \le [w]_{A_{p, \Z}}$ uniformly on $I_1$. Then, by the standard one-parameter reverse H\"{o}lder inequality there exist $r_1, r_2 > 1$ such that 
\begin{align}\label{eq:weta-1}
\bigg(\fint_{I_1} w^{r_1} \, dx_1 \bigg)^{\frac{1}{r_1}} 
\lesssim \fint_{I_1} w \, dx_1 
\qquad \text{uniformly on } x_2, x_3 \in \R
\end{align}
and 
\begin{align}\label{eq:weta-2}
\bigg(\fint_{I_{2, 3}} \langle w \rangle_{I_1}^{r_2} \, dx_2 dx_3 \bigg)^{\frac{1}{r_2}} 
\lesssim \fint_{I_{2, 3}} \langle w \rangle_{I_1} \, dx_2 dx_3 
\qquad\text{uniformly on } I_1. 
\end{align}
Picking $r := \min\{r_1, r_2\}$, we use Jensen's inequality and \eqref{eq:weta-1}--\eqref{eq:weta-2} to obtain that 
\begin{align*}
\fint_I w^r \, dx
&=\fint_{I_{2, 3}} \bigg(\fint_{I_1} w^r \, dx_1\bigg) \, dx_2 \, dx_3
\le \fint_{I_{2, 3}} \bigg(\fint_{I_1} w^{r_1} \, dx_1\bigg)^{\frac{r}{r_1}} \, dx_2 \, dx_3
\\
&\lesssim \fint_{I_{2, 3}} \bigg(\fint_{I_1} w \, dx_1\bigg)^r \, dx_2 \, dx_3
\le \bigg(\fint_{I_{2, 3}} \bigg(\fint_{I_1} w \, dx_1\bigg)^{r_2} \, dx_2 \, dx_3 \bigg)^{\frac{r}{r_2}}
\\
&\lesssim \bigg(\fint_{I_{2, 3}} \fint_{I_1} w \, dx_1 \, dx_2 \, dx_3\bigg)^r
=\bigg(\fint_I w \, dx\bigg)^r, 
\end{align*}
which shows the reverse H\"{o}lder inequality for $A_{p, \Z}$ weights. More details about the proof of \eqref{wpp} are left to the reader. 

By \eqref{MB} and \cite[Theorem 3.9]{CMP}, the assumption \eqref{RdFcpt-2} gives that 
\begin{align}\label{TL-1}
\text{$T$ is bounded on $L^q(v)$ for all $q \in (1, \infty)$ and $v \in A_{q, \Z}$}.
\end{align} 
Then \eqref{TL-1} applied to $w_1 = w_1(\eta) \in A_{p_1, \Z}$ yields that 
\begin{align}\label{TL-2}
\text{$T$ is bounded on $L^{p_1}(w_1)$}.
\end{align}  
From Theorem \ref{thm:inter}, \eqref{wpp}, the assumption \eqref{RdFcpt-1}, and \eqref{TL-2}, we conclude that $T$ is compact on $L^p(w)$.
\qed

\subsection{Weighted compactness of Zygmund shifts}
Let us give the proof of Theorem \ref{thm:SD-cpt}. Recall the weighted boundedness of Zygmund shifts (cf. \cite[Theorem 6.2]{HLMV}):
\begin{align}\label{SDZ}
\sup_{\w} \|\mathbf{S}_{\D_{\w}}^k\|_{L^r(u) \to L^r(u)} 
\lesssim (|k| +1)^2, \quad \forall r \in (1, \infty), \, u \in A_{r, \mathcal{Z}},
\end{align}
where the implicit constant is independent of $k$. By Theorem \ref{thm:RdF-cpt} and \eqref{SDZ}, it suffices to prove that $\E_{\w} \mathbf{S}_{\D_{\w}}^k$ is compact on $L^2(\R^3)$. By duality, similarity, and Definition \ref{def:shift}, it is enough to consider the case 
\begin{align*}
(\widetilde{h}_{I, \Z}, \widetilde{h}_{J, \Z}) 
= (h_{I_1} \otimes H_{I_{2, 3}, J_{2, 3}}, H_{I_1, J_1} \otimes h_{J_{2, 3}}).
\end{align*} 
Then in view of Theorem \ref{thm:KRLp} and \eqref{SDZ}, this is reduced to showing
\begin{align}
\label{SDKR-2}
\lim_{A \to \infty} &\sup_{\|f\|_{L^2} \le 1} 
\|\E_{\w} \mathbf{S}_{\D_{\w}}^k f \, \mathbf{1}_{B(0, A)^c}\|_{L^2} 
= 0, 
\\ 
\label{SDKR-3}
\lim_{|v| \to 0} &\sup_{\|f\|_{L^2} \le 1} 
\|\tau_v \, \E_{\w} \mathbf{S}_{\D_{\w}}^k f
- \E_{\w} \mathbf{S}_{\D_{\w}}^k f\|_{L^2} 
= 0. 
\end{align}

Define  
\begin{align*}
\mathbf{S}_{\D}^{k, N} f
:= \sum_{K \notin \D_{\lambda(k)}^N} 
\sum_{\substack{I, J \in \D_{\Z} \\ I^{(k)} = J^{(k)} = K}}   
a_{IJK} \, \langle f, \widetilde{h}_{I, \Z} \rangle \, \widetilde{h}_{J, \Z}. 
\end{align*}
Observe that 
\begin{align}
\label{UKK-1}
\langle f, h_{I_1} \otimes H_{I_{2, 3}, J_{2, 3}} \rangle 
& = \langle \mathcal{U}_{K, k} f, h_{I_1} \otimes H_{I_{2, 3}, J_{2, 3}} \rangle, 
\\ 
\label{UKK-2}
\langle f, H_{I_1, J_1} \otimes h_{I_{2, 3}} \rangle 
& = \langle \mathcal{U}_{K, k} f, H_{I_1, J_1} \otimes h_{I_{2, 3}} \rangle, 
\\
\label{UKK-3}
\langle f, H_{I_1, J_1} \otimes H_{I_{2, 3}, J_{2, 3}} \rangle 
& = \langle \mathcal{U}_{K, k} f, H_{I_1, J_1} \otimes H_{I_{2, 3}, J_{2, 3}} \rangle.
\end{align}
Recall that 
\begin{align}\label{HHhh}
H_{I_1, J_1} = h_{I_1}^0 - h_{J_1}^0 
\quad \text{ and } \quad 
H_{I_{2, 3}, J_{2, 3}} = h_{I_{2, 3}}^0 - h_{J_{2, 3}}^0. 
\end{align}
Then it follows from \eqref{UKK-1}--\eqref{HHhh} that 
\begin{align}\label{XX}
|\langle \mathbf{S}_{\D}^{k, N} f, g \rangle| 
\le \Xi_1 + \Xi_2 + \Xi_3 + \Xi_4, 
\end{align}
where 
\begin{align*}
\Xi_1 
& := \sum_{K \notin \D_{\lambda(k)}^N} 
\sum_{\substack{I, J \in \D_{\Z} \\ I^{(k)} = J^{(k)} = K}}   
\mathbf{F}(K) \frac{|I|}{|K|} 
|\langle \mathcal{U}_{K, k} f, h_{I_1} \otimes h_{I_{2, 3}}^0 \rangle|
|\langle \mathcal{U}_{K, k} g, h_{I_1}^0 \otimes h_{J_{2, 3}} \rangle|, 
\\
\Xi_2 
& := \sum_{K \notin \D_{\lambda(k)}^N} 
\sum_{\substack{I, J \in \D_{\Z} \\ I^{(k)} = J^{(k)} = K}}   
\mathbf{F}(K) \frac{|I|}{|K|} 
|\langle \mathcal{U}_{K, k} f, h_{I_1} \otimes h_{I_{2, 3}}^0 \rangle|
|\langle \mathcal{U}_{K, k} g, h_{J_1}^0 \otimes h_{J_{2, 3}} \rangle|, 
\\
\Xi_3 
& := \sum_{K \notin \D_{\lambda(k)}^N} 
\sum_{\substack{I, J \in \D_{\Z} \\ I^{(k)} = J^{(k)} = K}}   
\mathbf{F}(K) \frac{|I|}{|K|} 
|\langle \mathcal{U}_{K, k} f, h_{I_1} \otimes h_{J_{2, 3}}^0 \rangle|
|\langle \mathcal{U}_{K, k} g, h_{I_1}^0 \otimes h_{J_{2, 3}} \rangle|, 
\\
\Xi_4 
& := \sum_{K \notin \D_{\lambda(k)}^N} 
\sum_{\substack{I, J \in \D_{\Z} \\ I^{(k)} = J^{(k)} = K}}   
\mathbf{F}(K) \frac{|I|}{|K|} 
|\langle \mathcal{U}_{K, k} f, h_{I_1} \otimes h_{J_{2, 3}}^0 \rangle|
|\langle \mathcal{U}_{K, k} g, h_{J_1}^0 \otimes h_{J_{2, 3}} \rangle|. 
\end{align*}
To bound $\Xi_1 $, note that 
\begin{align}\label{JJK-1}
\sum_{J_1: J_1^{(k_1)} = K_1} 1 = 2^{k_1} = \frac{|K_1|}{|I_1|}
\quad \text{ and } \quad 
\sum_{\substack{J_2: J_2^{(k_2)} = K_2 \\ J_3: J_3^{(k_3)} = K_3}} \mathbf{1}_{J_{2, 3}} 
= \mathbf{1}_{K_{2, 3}}. 
\end{align}
Thus, by \eqref{JJK-1} and the fact that $I_1 \times K_2 \times K_3 \in \D_{2^{k_3-k_2}}$,  
\begin{align}\label{XX-1}
\Xi_1 
& \le \mathbf{F}_N \sum_{K \in \D_{\lambda(k)}}  
\sum_{\substack{I, J \in \D_{\Z} \\ I^{(k)} = J^{(k)} = K}} \frac{1}{|K|} 
\langle |\mathcal{U}_{K, k} f|, \mathbf{1}_{I_1} \otimes \mathbf{1}_{I_{2, 3}} \rangle
\langle |\mathcal{U}_{K, k} g|, \mathbf{1}_{I_1} \otimes \mathbf{1}_{J_{2, 3}} \rangle
\\ \nonumber 
& \le \mathbf{F}_N \sum_{K \in \D_{\lambda(k)}}  
\sum_{I_1^{(k_1)} = K_1} 
\langle |\mathcal{U}_{K, k} f| \rangle_{I_1 \times K_{2, 3}} 
\langle |\mathcal{U}_{K, k} g|, \mathbf{1}_{I_1} \otimes \mathbf{1}_{K_{2, 3}} \rangle
\\ \nonumber 
& \le \mathbf{F}_N \sum_{K \in \D_{\lambda(k)}}  
\sum_{I_1^{(k_1)} = K_1} \int_{I_1 \times K_{2, 3}}  
\big(M_{\D_{2^{k_3-k_2}}} \mathcal{U}_{K, k} f \big) \, 
|\mathcal{U}_{K, k} g| \, dx
\\ \nonumber 
& \le \mathbf{F}_N \int_{\R^3}  \sum_{K \in \D_{\lambda(k)}}
\big(M_{\D_{2^{k_3-k_2}}} \mathcal{U}_{K, k} f \big) \, 
|\mathcal{U}_{K, k} g| \, dx
\\ \nonumber 
& \le \mathbf{F}_N \bigg\|\bigg(\sum_{K \in \D_{\lambda(k)}}
\big(M_{\D_{2^{k_3-k_2}}} \mathcal{U}_{K, k} f \big)^2 \bigg)^{\frac12} \bigg\|_{L^2}
\bigg\| \bigg(\sum_{K \in \D_{\lambda(k)}} |\mathcal{U}_{K, k} g|^2 \bigg)^{\frac12} \bigg\|_{L^2}
\\ \nonumber 
& \lesssim_k \mathbf{F}_N \bigg\|\bigg(\sum_{K \in \D_{\lambda(k)}}
|\mathcal{U}_{K, k} f|^2 \bigg)^{\frac12} \bigg\|_{L^2}
\bigg\| \bigg(\sum_{K \in \D_{\lambda(k)}} |\mathcal{U}_{K, k} g|^2 \bigg)^{\frac12} \bigg\|_{L^2}
\\ \nonumber 
& \lesssim_k \mathbf{F}_N \|f\|_{L^2} \|g\|_{L^2}, 
\end{align}
provided \eqref{MM-1}--\eqref{MM-3}. Similarly, we use the fact $\sum_{I \in \D_{\Z}: I^{(k)} = K} \mathbf{1}_I  = \mathbf{1}_K$ to obtain 
\begin{align}\label{XX-2}
\Xi_2 
& \le \mathbf{F}_N \sum_{K \in \D_{\lambda(k)}}  
\sum_{\substack{I, J \in \D_{\Z} \\ I^{(k)} = J^{(k)} = K}} \frac{1}{|K|} 
\langle |\mathcal{U}_{K, k} f|, \mathbf{1}_I \rangle \, 
\langle |\mathcal{U}_{K, k} g|, \mathbf{1}_J \rangle
\\ \nonumber 
& \le \mathbf{F}_N \sum_{K \in \D_{\lambda(k)}}  
\langle |\mathcal{U}_{K, k} f| \rangle_K 
\langle |\mathcal{U}_{K, k} g|, \mathbf{1}_K \rangle
\\ \nonumber 
& \le \mathbf{F}_N \int_{\R^3}  \sum_{K \in \D_{\lambda(k)}}
\big(M_{\D_{\lambda(k)}} \mathcal{U}_{K, k} f \big) \, 
|\mathcal{U}_{K, k} g| \, dx
\\ \nonumber 
& \le \mathbf{F}_N \bigg\|\bigg(\sum_{K \in \D_{\lambda(k)}}
\big(M_{\D_{\lambda(k)}} \mathcal{U}_{K, k} f \big)^2 \bigg)^{\frac12} \bigg\|_{L^2}
\bigg\| \bigg(\sum_{K \in \D_{\lambda(k)}} |\mathcal{U}_{K, k} g|^2 \bigg)^{\frac12} \bigg\|_{L^2}
\\ \nonumber 
& \lesssim_k \mathbf{F}_N \bigg\|\bigg(\sum_{K \in \D_{\lambda(k)}}
|\mathcal{U}_{K, k} f|^2 \bigg)^{\frac12} \bigg\|_{L^2}
\bigg\| \bigg(\sum_{K \in \D_{\lambda(k)}} |\mathcal{U}_{K, k} g|^2 \bigg)^{\frac12} \bigg\|_{L^2}
\\ \nonumber 
& \lesssim_k \mathbf{F}_N \|f\|_{L^2} \|g\|_{L^2}. 
\end{align}
Moreover, by \eqref{MM-1}--\eqref{MM-3} and that 
\begin{align*}
\sum_{J \in \D_{\Z}: J^{(k)} = K} 1 = 2^{k_1 + k_2 + k_3} = |K|/|I|, 
\end{align*}
we arrive at 
\begin{align}\label{XX-3}
\Xi_3 
& \le \mathbf{F}_N \sum_{K \in \D_{\lambda(k)}}  
\sum_{\substack{I, J \in \D_{\Z} \\ I^{(k)} = J^{(k)} = K}} \frac{1}{|K|} 
\langle |\mathcal{U}_{K, k} f|, \mathbf{1}_{I_1} \otimes \mathbf{1}_{J_{2, 3}} \rangle
\langle |\mathcal{U}_{K, k} g|, \mathbf{1}_{I_1} \otimes \mathbf{1}_{J_{2, 3}} \rangle
\\ \nonumber 
& \le \mathbf{F}_N \sum_{K \in \D_{\lambda(k)}}  
\sum_{I^{(k)} = K} 
\langle |\mathcal{U}_{K, k} f| \rangle_I 
\langle |\mathcal{U}_{K, k} g|, \mathbf{1}_I \rangle
\\ \nonumber 
& \le \mathbf{F}_N \sum_{K \in \D_{\lambda(k)}}  
\sum_{I^{(k)} = K} \int_I 
\big(M_{\D_{\Z}} \mathcal{U}_{K, k} f \big) \, 
|\mathcal{U}_{K, k} g| \, dx
\\ \nonumber 
& \le \mathbf{F}_N \int_{\R^3}  \sum_{K \in \D_{\lambda(k)}}
\big(M_{\D_{\Z}} \mathcal{U}_{K, k} f \big) \, 
|\mathcal{U}_{K, k} g| \, dx
\\ \nonumber 
& \le \mathbf{F}_N \bigg\|\bigg(\sum_{K \in \D_{\lambda(k)}}
\big(M_{\D_{\Z}} \mathcal{U}_{K, k} f \big)^2 \bigg)^{\frac12} \bigg\|_{L^2}
\bigg\| \bigg(\sum_{K \in \D_{\lambda(k)}} |\mathcal{U}_{K, k} g|^2 \bigg)^{\frac12} \bigg\|_{L^2}
\\ \nonumber 
& \lesssim_k \mathbf{F}_N \bigg\|\bigg(\sum_{K \in \D_{\lambda(k)}}
|\mathcal{U}_{K, k} f|^2 \bigg)^{\frac12} \bigg\|_{L^2}
\bigg\| \bigg(\sum_{K \in \D_{\lambda(k)}} |\mathcal{U}_{K, k} g|^2 \bigg)^{\frac12} \bigg\|_{L^2}
\\ \nonumber 
& \lesssim_k \mathbf{F}_N \|f\|_{L^2} \|g\|_{L^2}. 
\end{align}
To estimate $\Xi_4$, observe that 
\begin{align}\label{JJK-2}
\sum_{I_1: I_1^{(k_1)} = K_1} \mathbf{1}_{I_1} = \mathbf{1}_{K_1}
\quad \text{ and } \quad 
\sum_{\substack{I_2: I_2^{(k_2)} = K_2 \\ I_3: I_3^{(k_3)} = K_3}} 1 
= 2^{k_2 + k_3} = \frac{|K_{2, 3}|}{|I_{2, 3}|}.
\end{align}
Along with \eqref{JJK-2} and the fact $K_1 \times J_2 \times J_3 \in \D_{2^{-k_1}}$, the estimates \eqref{MM-1}--\eqref{MM-3} imply 
\begin{align}\label{XX-4}
\Xi_4 
& \le \mathbf{F}_N \sum_{K \in \D_{\lambda(k)}}  
\sum_{\substack{I, J \in \D_{\Z} \\ I^{(k)} = J^{(k)} = K}} \frac{1}{|K|} 
\langle |\mathcal{U}_{K, k} f|, \mathbf{1}_{I_1} \otimes \mathbf{1}_{J_{2, 3}} \rangle
\langle |\mathcal{U}_{K, k} g|, \mathbf{1}_{J_1} \otimes \mathbf{1}_{J_{2, 3}} \rangle
\\ \nonumber 
& \le \mathbf{F}_N \sum_{K \in \D_{\lambda(k)}}  
\sum_{\substack{J_2: J_2^{(k_2)} = K_2 \\ J_3: J_3^{(k_3)} = K_3}} 
\langle |\mathcal{U}_{K, k} f| \rangle_{K_1 \times J_{2, 3}} 
\langle |\mathcal{U}_{K, k} g|, \mathbf{1}_{K_1} \otimes \mathbf{1}_{J_{2, 3}} \rangle
\\ \nonumber 
& \le \mathbf{F}_N \sum_{K \in \D_{\lambda(k)}}  
\sum_{\substack{J_2: J_2^{(k_2)} = K_2 \\ J_3: J_3^{(k_3)} = K_3}}  
\int_{K_1 \times J_{2, 3}} 
\big(M_{\D_{2^{-k_1}}} \mathcal{U}_{K, k} f \big) \, 
|\mathcal{U}_{K, k} g| \, dx
\\ \nonumber 
& \le \mathbf{F}_N \int_{\R^3}  \sum_{K \in \D_{\lambda(k)}}
\big(M_{\D_{2^{-k_1}}} \mathcal{U}_{K, k} f \big) \, 
|\mathcal{U}_{K, k} g| \, dx
\\ \nonumber 
& \le \mathbf{F}_N \bigg\|\bigg(\sum_{K \in \D_{\lambda(k)}}
\big(M_{\D_{2^{-k_1}}} \mathcal{U}_{K, k} f \big)^2 \bigg)^{\frac12} \bigg\|_{L^2}
\bigg\| \bigg(\sum_{K \in \D_{\lambda(k)}} |\mathcal{U}_{K, k} g|^2 \bigg)^{\frac12} \bigg\|_{L^2}
\\ \nonumber 
& \lesssim_k \mathbf{F}_N \bigg\|\bigg(\sum_{K \in \D_{\lambda(k)}}
|\mathcal{U}_{K, k} f|^2 \bigg)^{\frac12} \bigg\|_{L^2}
\bigg\| \bigg(\sum_{K \in \D_{\lambda(k)}} |\mathcal{U}_{K, k} g|^2 \bigg)^{\frac12} \bigg\|_{L^2}
\\ \nonumber 
& \lesssim_k \mathbf{F}_N \|f\|_{L^2} \|g\|_{L^2}. 
\end{align}
Therefore, by \eqref{XX}, \eqref{XX-1}, \eqref{XX-2}, \eqref{XX-3}, and \eqref{XX-4}, 
\begin{align}\label{SNL2}
\|\mathbf{S}_{\D}^{k, N} f\|_{L^2} 
\lesssim_k \mathbf{F}_N \|f\|_{L^2}, 
\end{align}
where the implicit constant is independent of $N$, $\D$, and $f$. 

Fix $k \in \N^3$. Let $A \ge 2^{|k| + 100}$ and $N := \big[\frac13 (\log_2 A - |k| - 1) \big] \ge 5$. Observe that 
\begin{align}\label{KN-1}
\supp(\widetilde{h}_{I, \Z}) \subset K, 
\qquad
\supp(\widetilde{h}_{J, \Z}) \subset K 
\quad \text{ for all } I, J \subset K, 
\end{align}
and 
\begin{align}\label{KN-2}
\bigcup_{K \in \D_{\lambda(k)}^N} K 
\subset \big\{x \in \R^3: |x_1| \le 2^{2N}, |x_2| \le 2^{2N}, |x_3| \le (\lambda(k)+1) 2^{3N} \big\} 
\subset B(0, 2A). 
\end{align}
Then, invoking Minkowski's inequality and \eqref{SNL2}--\eqref{KN-2}, we conclude  
\begin{align*}
& \|\E_{\w} \mathbf{S}_{\D_{\w}}^k f \, \mathbf{1}_{B(0, 2A)^c}\|_{L^2} 
\le \E_{\w} \|\mathbf{S}_{\D_{\w}}^k f \, \mathbf{1}_{B(0, 2A)^c}\|_{L^2} 
\\
&= \E_{\w} \|\mathbf{S}_{\D_{\w}}^{k, N} f \, \mathbf{1}_{B(0, 2A)^c}\|_{L^2} 
\le  \E_{\w} \|\mathbf{S}_{\D_{\w}}^{k, N} f\|_{L^2} 
\lesssim \mathbf{F}_N \|f\|_{L^2}, 
\end{align*}
where the implicit constant is independent of $\w$, $A$, and $f$. This implies \eqref{SDKR-2} as desired.

It remains to justify \eqref{SDKR-3}. Let $0 < |v| = |v_1| + |v_2| + |v_3| \ll 2^{-100}$ and $a \ge 2$ be an integer chosen later. There exists an integer $N=N(v) \ge 5$ such that $2^{-a (N+1)} \le |v| < 2^{- a N}$. If we denote $\phi_J^v := \tau_v (H_{I_1, J_1} \otimes h_{J_{2, 3}}) - H_{I_1, J_1} \otimes h_{J_{2, 3}}$, then   
\begin{align}\label{HFX} 
\|\tau_v \E_{\w} \mathbf{S}_{\D_{\w}}^k f
- \E_{\w} \mathbf{S}_{\D_{\w}}^k f\|_{L^2} 
\le \E_{\w} \big[\Gamma_{\w}^1(v; f)  + \Gamma_{\w}^2(v; f)  \big], 
\end{align}
where 
\begin{align*}
\Gamma_{\w}^1(v; f) 
&:= \bigg\|\sum_{K \notin \D_{\lambda(k)}^N} 
\sum_{\substack{I, J \in \D_{\w, \Z} \\ I^{(k)} = J^{(k)} = K}}   
a_{IJK} \, \langle f, h_{I_1} \otimes H_{I_{2, 3}, J_{2, 3}} \rangle \, \phi_J^v \bigg\|_{L^2}, 
\\ 
\Gamma_{\w}^2(v; f) 
&:= \bigg\|\sum_{K \in \D_{\lambda(k)}^N} 
\sum_{\substack{I, J \in \D_{\w, \Z} \\ I^{(k)} = J^{(k)} = K}}   
a_{IJK} \, \langle f, h_{I_1} \otimes H_{I_{2, 3}, J_{2, 3}} \rangle \, \phi_J^v \bigg\|_{L^2}.
\end{align*}
The inequality \eqref{SNL2} immediately gives 
\begin{align}\label{XVF}
\Gamma_{\w}^1(v; f) 
\le \|\tau_v \mathbf{S}_{\D_{\w}}^{k, N} f\|_{L^2} 
+ \|\mathbf{S}_{\D_{\w}}^{k, N} f\|_{L^2}
= 2 \|\mathbf{S}_{\D_{\w}}^{k, N} f\|_{L^2} 
\lesssim_k \mathbf{F}_N \|f\|_{L^2},  
\end{align}
where the implicit constant is independent of $\w$, $v$, and $f$. Note that $\lim_{N \to \infty} \mathbf{F}_N = 0$ and $|v| \to 0$ implies $N(v) \to \infty$. This and \eqref{XVF} lead to 
\begin{align}\label{HFX-1}
\lim_{|v| \to 0} \sup_{\w} 
\sup_{\|f\|_{L^2} \le 1} \Gamma_{\w}^1(v; f)  = 0. 
\end{align}
To estimate $\Gamma_{\w}^2(v; f)$, note that by \eqref{KN-2}, 
\begin{align}\label{car-DN}
\# \D_{\lambda(k)}^N 
\lesssim_k \, 2^{3N} \times 2^{3N} \times 2^{5N}
= 2^{11N}, \quad\forall N \ge 5, 
\end{align}
and by the fact $H_{I_1, J_1} = h_{I_1}^0 - h_{J_1}^0$, 
\begin{align}\label{PLP}
\|\phi_J^v\|_{L^2} 
& \le \|\tau_{v_1} h_{I_1}^0 - h_{I_1}^0\|_{L^2} 
\|\tau_{v_2} h_{J_2}\|_{L^2} 
\|\tau_{v_3} h_{J_3}\|_{L^2}
\\ \nonumber 
& \quad + \|\tau_{v_1} h_{J_1}^0 - h_{J_1}^0\|_{L^2} 
\|\tau_{v_2} h_{J_2}\|_{L^2} 
\|\tau_{v_3} h_{J_3}\|_{L^2}
\\ \nonumber 
& \quad + \|H_{I_1, J_1}\|_{L^2} 
\|\tau_{v_2} h_{J_2} - h_{J_2}\|_{L^2} 
\|\tau_{v_3} h_{J_3}\|_{L^2}
\\ \nonumber 
& \quad + \|H_{I_1, J_1}\|_{L^2} 
\|h_{J_2}\|_{L^2} 
\|\tau_{v_3} h_{J_3} - h_{J_3}\|_{L^2}
\\ \nonumber 
&\lesssim |v_1|^{\frac12} \ell(I_1)^{-\frac12} 
+ |v_1|^{\frac12} \ell(J_1)^{-\frac12} 
+ |v_2|^{\frac12} \ell(J_2)^{-\frac12} 
+ |v_3|^{\frac12} \ell(J_3)^{-\frac12}.  
\end{align}
Hence, it follows from \eqref{car-DN} and \eqref{PLP} that 
\begin{align*}
\Gamma_{\w}^2(v; f)
& \lesssim \sum_{K \in \D_{\lambda(k)}^N} 
\sum_{\substack{I, J \in \D_{\Z} \\ I^{(k)} = J^{(k)} = K}}   
|I|^{\frac12} \langle |f| \rangle_K
\|\phi_J^v\|_{L^2}
\\
&\lesssim |v|^{\frac12} \sum_{K \in \D_{\lambda(k)}^N} 
\sum_{\substack{I, J \in \D_{\w, \Z} \\ I^{(k)} = J^{(k)} = K}}   
|I|^{\frac12} |K|^{-\frac12}  
\big[\ell(J_1)^{-\frac12} + \ell(J_2)^{-\frac12} + \ell(J_3)^{-\frac12}\big]
\|f\|_{L^2}
\\
&\lesssim |v|^{\frac12} \# \D_{\lambda(k)}^N 
2^{2(k_1 + k_2 + k_3)} 2^{-(k_1 + k_2 + k_3)/2} 
\big[ 2^{k_1/2} 2^{N/2} + 2^{k_2/2} 2^{N/2} + 2^{k_3/2} \lambda(k)^{1/2} 2^N \big] \|f\|_{L^2}
\\
&\lesssim_k 2^{- (a/2-12)N} \|f\|_{L^2}. 
\end{align*}
If we take $a > 24$, then there holds 
\begin{align}\label{HFX-2}
\lim_{|v| \to 0} \sup_{\w} 
\sup_{\|f\|_{L^2} \le 1} \Gamma_{\w}^2(v; f)  = 0. 
\end{align} 
Therefore, \eqref{SDKR-3} is a consequence of \eqref{HFX}, \eqref{HFX-1}, and \eqref{HFX-2}.  
\qed

\subsection{Proof of Theorem \ref{thm:cpt}}
Let us conclude the compactness of CZZ operators. Let $\delta_1, \delta_{2, 3} \in (0, 1]$. Observe that 
\begin{align}\label{Daa}
(\log 2) D_1(x) \le D_{\log}(x) \le \alpha^{-1} D_{1-\alpha}(x), 
\quad\forall x \in \Rn, \, \alpha \in (0, 1).
\end{align}
Thus, it suffices to prove part \eqref{list-01} for $(D_{\theta}, \delta_1, \delta_{2, 3})$-CZZ operator with $\theta \in (0, 1)$ and part \eqref{list-02} for $(D_{\log}, 1, \delta_{2, 3})$-CZZ operators. We only present the proof of the latter because the former can be shown in a similar way. 

Let $p \in (1, \infty)$ and $w \in A_{p, \Z}$. Let $T$ be a $(D_{\log}, 1, \delta_{2, 3})$-CZZ operator satisfying the corresponding hypotheses \eqref{H1}--\eqref{H4}. In view of \eqref{Daa}, $T$ is a $(D_{\theta}, 1, \delta_{2, 3})$-CZZ operator satisfying \eqref{H1}--\eqref{H4} for any $\theta \in (0, 1)$. Thus, Theorem \ref{thm:repre} holds for $T$. Define 
\begin{align*}
T^N := \sum_{k \in \N^3: |k| \le N} \sum_{m=1}^{m_0}
\varphi(k) \, \E_{\w} \mathbf{S}_{m, \D_{\w}}^k, \quad N \ge 1.
\end{align*}
It follows from Theorem \ref{thm:repre}, \eqref{SDZ}, and Lemma \ref{lem:kkk} applied to $(\rho, \alpha_0, \alpha_1, \alpha_2, \alpha_3) = \big(2, \theta, \delta_1, \min\{\delta_{2, 3}, \theta\}, 0 \big)$ that  
\begin{equation}\label{TNN}
\begin{aligned}
& \|T - T^N\|_{L^p(w) \to L^p(w)} 
\le \sum_{|k| > N} \sum_{m=1}^{m_0} 
\varphi(k) \, \E_{\w} \|\mathbf{S}_{m, \D_{\w}}^k\|_{L^p(w) \to L^p(w)} 
\\ 
& \lesssim \sum_{|k| > N} (|k| + 1)^2 \varphi(k)
\lesssim N^{10} 2^{-N \min\{\delta_1, \, \delta_{2, 3}, \, \theta\}/6} \to 0, 
\quad\text{as } N \to \infty. 
\end{aligned}
\end{equation}
Moreover, by Theorem \ref{thm:SD-cpt}, $T^N$ is compact on $L^p(w)$, which along with \eqref{TNN} implies the compactness of $T$ on $L^p(w)$. This justifies Theorem \ref{thm:cpt} part \eqref{list-02}. 
\qed

\section{Bilinear Calder\'{o}n--Zygmund operators associated with Zygmund dilations}\label{sec:BCZZ}
Let us define bilinear Calder\'{o}n--Zygmund operators associated with Zygmund dilations. Given a bilinear operator $T$ acting on functions defined on $\R^1 \times \R^2$, its \emph{adjoint operators} are defined by
\begin{align*}
& \langle T(f_1 \otimes f_{2, 3}, g_1 \otimes g_{2, 3}), h_1 \otimes h_{2, 3} \rangle
\\
& = \langle T^{*, 1}(h_1 \otimes h_{2, 3}, g_1 \otimes g_{2, 3}), f_1 \otimes f_{2, 3} \rangle 
= \langle T^{*, 2} (f_1 \otimes f_{2, 3}, h_1 \otimes h_{2, 3}), g_1 \otimes g_{2, 3} \rangle 
\\
& = \langle T_1^{*, 1} (g_1 \otimes f_{2, 3}, g_1 \otimes g_{2, 3}), f_1 \otimes h_{2, 3} \rangle 
= \langle T_1^{*, 2} (f_1 \otimes f_{2, 3}, h_1 \otimes g_{2, 3}), g_1 \otimes h_{2, 3} \rangle 
\\
& = \langle T_{2, 3}^{*, 1} (f_1 \otimes h_{2, 3}, g_1 \otimes g_{2, 3}), h_1 \otimes f_{2, 3} \rangle 
= \langle T_{2, 3}^{*, 2} (f_1 \otimes f_{2, 3}, g_1 \otimes h_{2, 3}), h_1 \otimes g_{2, 3} \rangle.
\end{align*}
If $K$ is the integral kernel of $T$, then the kernels $K^{*, j}, K_1^{*, j}, K_{2, 3}^{*, j}$ of $T^{*, j}, T_1^{*, j}, T_{2, 3}^{*, j}$ are given by
\begin{align*}
& K^{*, 1}(x, y, z) = K(z, y, x), \quad
K^{*, 2}(x, y, z) = K(x, z, y), 
\\
& K_1^{*, 1}(x, y, z) = K((z_1, x_2, x_3), y, (x_1, z_2, z_3)), \quad
K_1^{*, 2}(x, y, z) = K(x, (z_1, y_2, y_3), (y_1, z_2, z_3)),
\\
& K_{2, 3}^{*, 1}(x, y, z) = K((x_1, z_2, z_3), y, (z_1, x_2, x_3)), \quad
K_{2, 3}^{*, 2}(x, y, z) = K(x, (y_1, z_2, z_3), (z_1, y_2, y_3)).
\end{align*}

Given  parameters $\vartheta \in (0, 2]$ and $\theta \in (0, 1]$, denote 
\begin{align*}
\mathfrak{D}_{\vartheta}(x, y) 
:= \bigg[\frac{(|x_1| + |y_1|)(|x_2| + |y_2|)}{|x_3| + |y_3|} 
+ \frac{|x_3| + |y_3|}{(|x_1| + |y_1|)(|x_2| + |y_2|)} \bigg]^{-\vartheta} 
\end{align*}
for any $(x, y) \in \R^3 \times \R^3 \setminus \{(x, y): x_i = 0 \text{ and } y_i = 0 \text{ for some } i=1, 2, 3 \}$, and 
\begin{align*}
D_{\theta}(x, y) 
:= \bigg[\frac{|x_1| (|x_2| + |y_2|)}{|x_3| + |y_3|} 
+ \frac{|x_3| + |y_3|}{|x_1| (|x_2| + |y_2|)} \bigg]^{-\theta} 
\end{align*}
for any $(x, y) \in \R^3 \times \R^3 \setminus \{(x, y): x_1 = 0 \text{ or } x_i = 0 \text{ and } y_i = 0 \text{ for some } i=2, 3 \}$.

First, let us define the compact full kernel representation. In what follows, let $\delta_1, \delta_{2, 3} \in (0, 1]$ be fixed numbers.

\begin{definition}\label{def:Bfull}
A bilinear operator $T$ admits the \emph{compact full kernel representation} if the following hold. 
If $f = f_1 \otimes f_{2, 3}$, $g = g_1 \otimes g_{2, 3}$, and $h = h_1 \otimes h_{2, 3}$ with $f_1, g_1, h_1 :\R \rightarrow \R$, $f_{2, 3}, g_{2, 3}, h_{2, 3} :\R^2 \rightarrow \R$, $\supp(f_1) \cap \supp(g_1) \cap \supp(h_1) = \emptyset$, and $\supp(f_{2, 3}) \cap \supp(g_{2, 3}) \cap \supp(h_{2, 3}) = \emptyset$, then 
\begin{align*}
\langle T(f, g), h \rangle 
= \int_{\R^3} \int_{\R^3} \int_{\R^3} K(x, y, z) f(y) g(z) h(x) \, dx \, dy \, dz,
\end{align*}
where the kernel $K: (\R^3 \times \R^3 \times \R^3) \setminus \big\{(x, y, z): x_i = y_i = z_i \text{ for some } i=1, 2, 3 \big\} \rightarrow \C$ satisfies 
\begin{list}{\rm (\theenumi)}{\usecounter{enumi}\leftmargin=1.2cm \labelwidth=1cm \itemsep=0.2cm \topsep=0.2cm \renewcommand{\theenumi}{\arabic{enumi}}} 

\item\label{Bfull-1} the size condition: 
\begin{align*}
|K(x, y, z)| 
\leq \mathfrak{D}_{\vartheta}(x - y, x - z) 
\prod_{i=1}^3 \frac{F_i(x_i, y_i, z_i)}{(|x_i - y_i| + |x_i - z_i|)^2}. 
\end{align*}

\item\label{Bfull-2} the H\"{o}lder condition: 
\begin{align*}
& |K(x, y, z) - K((x_1, x'_2, x'_3), y, z) - K((x'_1, x_2, x_3), y, z) + K(x', y, z)| 
\\
& \leq \bigg(\frac{|x_1 - x'_1|}{|x_1 - y_1| + |x_1 - z_1|}\bigg)^{\delta_1} 
\bigg(\frac{|x_2 - x'_2|}{|x_2 - y_2| + |x_2 - z_2|} 
+ \frac{|x_3 - x'_3|}{|x_3 - y_3| + |x_3 - z_3|}\bigg)^{\delta_{2, 3}}
\\
&\quad \times \mathfrak{D}_{\vartheta}(x - y, x - z) 
\prod_{i=1}^3 \frac{F_i(x_i, y_i, z_i)}{(|x_i - y_i| + |x_i - z_i|)^2}. 
\end{align*}
whenever $|x_i - x'_i| \leq \max\{|x_i - y_i|, |x_i - z_i|\}/2$ for $i=1, 2, 3$.

\item\label{Bfull-3} the mixed size-H\"{o}lder conditions:  
\begin{align*}
& |K(x, y, z) - K((x'_1, x_2, x_3), y, z)|
\\
& \leq \bigg(\frac{|x_1 - x'_1|}{|x_1 - y_1| + |x_1 - z_1|}\bigg)^{\delta_1} 
\mathfrak{D}_{\vartheta}(x - y, x - z) 
\prod_{i=1}^3 \frac{F_i(x_i, y_i, z_i)}{(|x_i - y_i| + |x_i - z_i|)^2}, 
\end{align*}
whenever $|x_1 - x'_1| \leq \max\{|x_1 - y_1|, |x_1 - z_1|\}/2$, and
\begin{align*}
& |K(x, y, z) - K((x_1, x'_2, x'_3), y, z)|
\\
& \le \bigg(\frac{|x_2 - x'_2|}{|x_2 - y_2| + |x_2 - z_2|} 
+ \frac{|x_3 - x'_3|}{|x_3 - y_3| + |x_3 - z_3|}\bigg)^{\delta_{2, 3}}
\\
& \quad \times \mathfrak{D}_{\vartheta}(x - y, x - z) 
\prod_{i=1}^3 \frac{F_i(x_i, y_i, z_i)}{(|x_i - y_i| + |x_i - z_i|)^2}, 
\end{align*}
whenever $|x_2 - x'_2| \leq \max\{|x_2 - y_2|, |x_2 - z_2|\}/2$ and $|x_3 - x'_3| \leq \max\{|x_3 - y_3|, |x_3 - z_3|\}/2$. 

\item\label{Bfull-4} the function $F_i$ above is given by 
\begin{align*}
F_i(x_i, y_i, z_i) 
:= F_{i, 1}(|x_i - y_i| + |x_i - z_i|) 
F_{i, 2}(|x_i - y_i| + |x_i - z_i|) 
F_{i, 3}(|x_i + y_i| + |x_i + z_i|),  
\end{align*}
where $(F_{i, 1}, F_{i, 2}, F_{i, 3}) \in \mathscr{F}$, $i=1, 2, 3$.

\item\label{Bfull-5} The adjoint kernels $K^{*, j}, K_1^{*, j}, K_{2, 3}^{*, j}$, $j=1, 2$, satisfy the same estimates \eqref{Bfull-1}--\eqref{Bfull-4}.
\end{list}
\end{definition}

Next, we define the compact partial kernel representation.

\begin{definition}\label{def:Bpartial-1}
A bilinear operator $T$ admits a \emph{compact partial kernel representation on the parameter $\{1\}$} if the following hold. If $f = f_1 \otimes f_{2, 3}$, $g = g_1 \otimes g_{2, 3}$, and $h = h_1 \otimes h_{2, 3}$ with $\supp(f_1) \cap \supp(g_1) \cap \supp(h_1) = \emptyset$, then 
\begin{align*}
\langle T(f, g), h \rangle 
= \int_{\R} \int_{\R} K_{f_{2, 3}, g_{2, 3}, h_{2, 3}}(x_1, y_1, z_1) 
f_1(y_1) g_1(z_1) h_1(x_1) \, dx_1 \, dy_1 \, dz_1,
\end{align*}
where the kernel $K_{f_{2, 3}, g_{2, 3}, h_{2, 3}}: (\R \times \R \times \R) \setminus \big\{(x_1, y_1, z_1): x_1 = y_1 = z_1 \big\} \rightarrow \C$ satisfies  
\begin{list}{\rm (\theenumi)}{\usecounter{enumi}\leftmargin=1.2cm \labelwidth=1cm \itemsep=0.2cm \topsep=0.2cm \renewcommand{\theenumi}{\arabic{enumi}}} 

\item\label{Bpartial-11} the size condition:
\begin{align*}
|K_{f_{2, 3}, g_{2, 3}, h_{2, 3}}(x_1, y_1, z_1)| 
\leq C(f_{2, 3}, g_{2, 3}, h_{2, 3}) \frac{F_1(x_1, y_1, z_1)}{(|x_1 - y_1| + |x_1 - z_1|)^2}. 
\end{align*}

\item\label{Bpartial-12} the H\"{o}lder conditions:
\begin{align*}
& |K_{f_{2, 3}, g_{2, 3}, h_{2, 3}}(x_1, y_1, z_1) - K_{f_{2, 3}, g_{2, 3}}(x'_1, y_1, z_1)|
\\
& \leq C(f_{2, 3}, g_{2, 3}, h_{2, 3}) F_1(x_1, y_1, z_1) 
\frac{|x_1 - x'_1|^{\delta_1}}{(|x_1 - y_1| + |x_1 - z_1|)^{2 + \delta_1}}
\end{align*}
whenever $|x_1 - x'_1| \leq \max\{|x_1 - y_1|, |x_1 - z_1|\}/2$, 
\begin{align*}
& |K_{f_{2, 3}, g_{2, 3}, h_{2, 3}}(x_1, y_1, z_1) - K_{f_{2, 3}, g_{2, 3}, h_{2, 3}}(x_1, y'_1, z_1)|
\\
& \leq C(f_{2, 3}, g_{2, 3}, h_{2, 3}) F_1(x_1, y_1, z_1) 
\frac{|y_1 - y'_1|^{\delta_1}}{(|x_1 - y_1| + |x_1 - z_1|)^{2 + \delta_1}}
\end{align*}
whenever $|y_1 - y'_1| \leq \max\{|x_1 - y_1|, |x_1 - z_1|\}/2$, and 
\begin{align*}
& |K_{f_{2, 3}, g_{2, 3}, h_{2, 3}}(x_1, y_1, z_1) - K_{f_{2, 3}, g_{2, 3}, h_{2, 3}}(x_1, y_1, z'_1)|
\\
& \leq C(f_{2, 3}, g_{2, 3}, h_{2, 3}) F_1(x_1, y_1, z_1) 
\frac{|z_1 - z'_1|^{\delta_1}}{(|x_1 - y_1| + |x_1 - z_1|)^{2 + \delta_1}}
\end{align*}
whenever $|z_1 - z'_1| \leq \max\{|x_1 - y_1|, |x_1 - z_1|\}/2$.

\item\label{Bpartial-13} the function $F_1$ above is given by 
\begin{align*}
F_1(x_1, y_1, z_1) 
:= F_{1, 1}(|x_1 - y_1| + |x_1 - z_1|) 
F_{1, 2}(|x_1 - y_1| + |x_1 - z_1|) 
F_{1, 3}(|x_1 + y_1| + |x_1 + z_1|), 
\end{align*}
where $(F_{1, 1}, F_{1, 2}, F_{1, 3}) \in \mathscr{F}$.

\item\label{Bpartial-14} the bound $C(f_{2, 3}, g_{2, 3}, h_{2, 3})$ above verifies  
\begin{align*}
& C(\mathbf{1}_{I_{2, 3}}, \mathbf{1}_{I_{2, 3}}, \mathbf{1}_{I_{2, 3}}) 
+ C(\mathbf{1}_{I_{2, 3}}, \mathbf{1}_{I_{2, 3}}, a_{I_{2, 3}}) 
\le F_2(I_2) \, |I_2| \, F_3(I_3) \, |I_3|,  
\\
& C(\mathbf{1}_{I_{2, 3}}, a_{I_{2, 3}}, \mathbf{1}_{I_{2, 3}}) 
+ C(a_{I_{2, 3}}, \mathbf{1}_{I_{2, 3}}, \mathbf{1}_{I_{2, 3}}) 
\le F_2(I_2) \, |I_2| \, F_3(I_3) \, |I_3|,  
\end{align*}
for all rectangles $I_{2, 3} \subset \R^2$ and all functions $a_{I_{2, 3}}$ satisfying $\supp(a_{I_{2, 3}}) \subset I_{2, 3}$, $|a_{I_{2, 3}}| \le 1$, and $\int_{\R^2} a_{I_{2, 3}} \, dx_{2, 3}=0$, where $F_2, F_3 \in \mathscr{F}_0$. 
\end{list} 
\end{definition}

\begin{definition}\label{def:Bpartial-2}
A bilinear operator $T$ admits a \emph{compact partial kernel representation on the parameter $\{2, 3\}$} if the following hold. If $f = f_1 \otimes f_{2, 3}$, $g = g_1 \otimes g_{2, 3}$, and $h = h_1 \otimes h_{2, 3}$ with $\supp(f_{2, 3}) \cap \supp(g_{2, 3}) \cap \supp(h_{2, 3}) = \emptyset$, then 
\begin{align*}
\langle T(f, g), h \rangle 
= \int_{\R^2} \int_{\R^2} \int_{\R^2} 
K_{f_1, g_1, h_1}(x_{2, 3}, y_{2, 3}, z_{2, 3}) 
f_{2, 3}(y_{2, 3}) g_{2, 3}(z_{2, 3}) h_{2, 3}(x_{2, 3}) \, dx_{2, 3} \, dy_{2, 3} \, dz_{2, 3}, 
\end{align*}
where the kernel $K_{f_1, g_1, h_1}: (\R^2 \times \R^2 \times \R^2) \setminus \big\{(x_{2, 3}, y_{2, 3}, z_{2, 3}): x_2 = y_2 = z_2 \text{ or } x_3 = y_3 = z_3 \big\} \rightarrow \C$ satisfies  
\begin{list}{\rm (\theenumi)}{\usecounter{enumi}\leftmargin=1.2cm \labelwidth=1cm \itemsep=0.2cm \topsep=0.2cm \renewcommand{\theenumi}{\arabic{enumi}}} 

\item\label{Bpartial-21} the size condition:
\begin{align*}
|K_{f_1, g_1, h_1}(x_{2, 3}, y_{2, 3}, z_{2, 3})| 
\leq C(f_1, g_1, h_1) D_{\theta} \big((t, x_{2, 3} - y_{2, 3}), x-z \big) 
\prod_{i=2}^3 \frac{F_i(x_i, y_i, z_i)}{(|x_i - y_i| + |x_i - z_i|)^2}, 
\end{align*}
where $t:= |\supp f_1 \cup \supp g_1 \cup \supp h_1|$.

\item\label{Bpartial-22} the H\"{o}lder conditions: 
\begin{align*}
& |K_{f_1, g_1, h_1}(x_{2, 3}, y_{2, 3}, z_{2, 3}) - K_{f_1, g_1, h_1}(x'_{2, 3}, y_{2, 3}, z_{2, 3})|
\\ 
& \leq C(f_1, g_1, h_1) 
\bigg(\frac{|x_2 - x'_2|}{|x_2 - y_2| + |x_2 - z_2|} 
+ \frac{|x_3 - x'_3|}{|x_3 - y_3| + |x_3 - z_3|}\bigg)^{\delta_{2, 3}} 
\\
& \quad \times D_{\theta} \big((t, x_{2, 3} - y_{2, 3}), x - z \big) 
\prod_{i=2}^3 \frac{F_i(x_i, y_i, z_i)}{(|x_i - y_i| + |x_i - z_i|)^2} 
\end{align*}
whenever $x'_{2, 3} = (x'_2, x'_3)$ satisfies $|x_i - x'_i| \leq \max\{|x_i - y_i|, |x_i - z_i|\}/2$ for $i=2, 3$, 
\begin{align*}
& |K_{f_1, g_1, h_1}(x_{2, 3}, y_{2, 3}, z_{2, 3}) - K_{f_1, g_1, h_1}(x_{2, 3}, y'_{2, 3}, z_{2, 3})|
\\ 
& \leq C(f_1, g_1, h_1) 
\bigg(\frac{|y_2 - y'_2|}{|x_2 - y_2| + |x_2 - z_2|} 
+ \frac{|y_3 - y'_3|}{|x_3 - y_3| + |x_3 - z_3|}\bigg)^{\delta_{2, 3}} 
\\
& \quad \times D_{\theta} \big((t, x_{2, 3} - y_{2, 3}), x - z \big) 
\prod_{i=2}^3 \frac{F_i(x_i, y_i, z_i)}{(|x_i - y_i| + |x_i - z_i|)^2} 
\end{align*}
whenever $y'_{2, 3} = (y'_2, y'_3)$ satisfies $|y_i - y'_i| \leq \max\{|x_i - y_i|, |x_i - z_i|\}/2$ for $i=2, 3$, and 
\begin{align*}
& |K_{f_1, g_1, h_1}(x_{2, 3}, y_{2, 3}, z_{2, 3}) - K_{f_1, g_1, h_1}(x_{2, 3}, y_{2, 3}, z'_{2, 3})|
\\ 
& \leq C(f_1, g_1, h_1) 
\bigg(\frac{|z_2 - z'_2|}{|x_2 - y_2| + |x_2 - z_2|} 
+ \frac{|z_3 - z'_3|}{|x_3 - y_3| + |x_3 - z_3|}\bigg)^{\delta_{2, 3}} 
\\
& \quad \times D_{\theta} \big((t, x_{2, 3} - y_{2, 3}), x - z \big) 
\prod_{i=2}^3 \frac{F_i(x_i, y_i, z_i)}{(|x_i - y_i| + |x_i - z_i|)^2} 
\end{align*}
whenever $z'_{2, 3} = (z'_2, z'_3)$ satisfies $|z_i - z'_i| \leq \max\{|x_i - y_i|, |x_i - z_i|\}/2$ for $i=2, 3$.

\item\label{Bpartial-23} the function $F_i$ above is given by 
\begin{align*}
F_i(x_i, y_i, z_i) 
:= F_{i, 1}(|x_i - y_i| + |x_i - z_i|) 
F_{i, 2}(|x_i - y_i| + |x_i - z_i|) 
F_{i, 3}(|x_i + y_i| + |x_i + z_i|), 
\end{align*}
where $(F_{i, 1}, F_{i, 2}, F_{i, 3}) \in \mathscr{F}$, $i=2, 3$.

\item\label{Bpartial-24} the bound $C(f_1, g_1, h_1)$ above verifies  
\begin{align*}
C(\mathbf{1}_{I_1}, \mathbf{1}_{I_1}, \mathbf{1}_{I_1}) 
+ C(\mathbf{1}_{I_1}, \mathbf{1}_{I_1}, a_{I_1}) 
+ C(\mathbf{1}_{I_1}, a_{I_1}, \mathbf{1}_{I_1}) 
+ C(a_{I_1}, \mathbf{1}_{I_1}, \mathbf{1}_{I_1}) 
\le F_1(I_1) \, |I_1|, 
\end{align*}
for all intervals $I_1 \subset \R$ and all functions $a_{I_1}$ satisfying $\supp(a_{I_1}) \subset I_1$, $|a_{I_1}| \le 1$, and $\int_{\R} a_{I_1} \, dx_1=0$, where $F_1 \in \mathscr{F}_0$. 
\end{list} 
\end{definition}

\begin{definition}\label{def:Bpartial}
A bilinear operator $T$ admits the \emph{compact partial kernel representation} if it admits the compact partial kernel representation on the parameters $\{1\}$ and $\{2, 3\}$. 
\end{definition}

Let us turn to the compactness and cancellation assumptions for bilinear singular integrals associated with Zygmund dilations.

\begin{definition}\label{def:BWCP}
We say that $T$ satisfies the \emph{weak compactness property} if 
\begin{align*}
|\langle T(\mathbf{1}_I, \mathbf{1}_I), \mathbf{1}_I \rangle|
\leq F_1(I_1) F_2(I_2) F_3(I_3) \, |I|, 
\end{align*}
for all Zygmund rectangles $I=I_1 \times I_2 \times I_3$, where $F_1, F_2, F_3 \in \mathscr{F}_0$. If the functions $F_1, F_2, F_3$ above are replaced by a uniform constant $C \ge 1$, we say that $T$ satisfies the \emph{weak boundedness property}.
\end{definition}

\begin{definition}\label{def:Bcancellation}
We say that $T$ satisfies the \emph{cancellation condition} if   
\begin{align*}
\langle T(1 \otimes \mathbf{1}_{I_{2, 3}}, 1 \otimes \mathbf{1}_{J_{2, 3}}), 
h_{K_1} \otimes \mathbf{1}_{K_{2, 3}} \rangle
& = \langle T(\mathbf{1}_{I_1} \otimes 1, \mathbf{1}_{J_1} \otimes 1), 
\mathbf{1}_{K_1} \otimes h_{K_{2, 3}} \rangle
= 0, 
\\
\langle T_1^{*, 1}(1 \otimes \mathbf{1}_{I_{2, 3}}, 1 \otimes \mathbf{1}_{J_{2, 3}}), 
h_{K_1} \otimes \mathbf{1}_{K_{2, 3}} \rangle
& = \langle T_{2, 3}^{*, 1}(\mathbf{1}_{I_1} \otimes 1, \mathbf{1}_{J_1} \otimes 1), 
\mathbf{1}_{K_1} \otimes h_{K_{2, 3}} \rangle
= 0, 
\\
\langle T_1^{*, 2}(1 \otimes \mathbf{1}_{I_{2, 3}}, 1 \otimes \mathbf{1}_{J_{2, 3}}), 
h_{K_1} \otimes \mathbf{1}_{K_{2, 3}} \rangle
& = \langle T_{2, 3}^{*, 2}(\mathbf{1}_{I_1} \otimes 1, \mathbf{1}_{J_1} \otimes 1), 
\mathbf{1}_{K_1} \otimes h_{K_{2, 3}} \rangle
= 0, 
\end{align*}
for all rectangles $I = I_1 \times I_2 \times I_3=: I_1 \times I_{2, 3}$, $J = J_1 \times J_2 \times J_3=: J_1 \times J_{2, 3}$, and $K = K_1 \times K_2 \times K_3=: K_1 \times K_{2, 3}$.
\end{definition}

\begin{definition}\label{def:BCZO}
Let $T$ be a bilinear operator. 
\begin{itemize}

\item We say that $T$ admits the \emph{full kernel representation} if the functions $F_1, F_2, F_3$ in the compact full kernel representation are replaced by a uniform constant $C \ge 1$. 

\item We say that $T$ admits the \emph{partial kernel representation} if the functions $F_1, F_2, F_3$ in the compact partial kernel representation are replaced by a uniform constant $C \ge 1$. 

\item $T$ is called a \emph{$(\mathfrak{D}_{\vartheta}, D_{\theta}, \delta_1, \delta_{2, 3})$-BCZZ} if $T$ admits the full and partial kernel representations, and satisfies the weak boundedness property and the cancellation condition. 
\end{itemize}
\end{definition}

Now we define bilinear dyadic shifts associated with Zygmund dilations.

\begin{definition}\label{def:Bshift}
Given $k = (k_1, k_2, k_3) \in \N^3$, $\bm{l} = (l_1, l_2, l_3)$ with $l_j = (l_j^1, l_j^2, l_j^3) \in \N^3$, and $\D = \D^1 \times \D^2 \times \D^3$ with $\D^1, \D^2, \D^3$ being dyadic grids on $\R$, a \emph{compact bilinear Zygmund shift $\mathbf{S}_{\D}^{k, \bm{l}}$ of complexity $\bm{l}$} is an operator of the form  
\begin{align*}
\mathbf{S}_{\D}^{k, \bm{l}} (f, g)
:= \sum_{Q \in \D_{\lambda'(k)}} 
\sum_{\substack{I, J, K \in \D \\ I^{(l_1)} = J^{(l_2)} = K^{(l_3)} = Q}}   
a_{IJKQ} \, \langle f, h_{I_1} \otimes h_{I_{2, 3}}^0 \rangle \,
\langle g, h_{J_1}^0 \otimes h_{J_{2, 3}} \rangle \, h_K, 
\end{align*}
where $\lambda'(k) := 2^n$ with $|n| \le 3 \max\{k_1, k_2, k_3\}$, $\D_{\lambda'(k)} := \{Q \in \D: \ell(Q_3) = \lambda'(k) \ell(Q_1) \ell(Q_2)\}$, $I^{(k)} := I_1^{(k_1)} \times I_2^{(k_2)} \times I_3^{(k_3)}$, at least one rectangle $I_1 \times J_{2, 3}, I_1 \times K_{2, 3}, K_1 \times J_{2, 3}, K_1 \times K_{2, 3}$ is a Zygmund rectangle, $l_j^i \le k_i$ and $(l_j^3 - l_j^2)_+ \le (k_3 - k_2)_+$ for all $i, j = 1, 2, 3$. Moreover, the coefficients $a_{IJK}$ satisfy  
\begin{align*}
|a_{IJKQ}| 
\le \mathbf{F}(Q) \frac{|I|^{\frac12} |J|^{\frac12} |K|^{\frac12}}{|Q|^2}
\end{align*}
with  
\begin{align*}
\mathbf{F}(Q) \le 1 
\quad\text{and}\quad 
\lim_{N \to \infty} \mathbf{F}_N 
:= \lim_{N \to \infty} \sup_{\D} \sup_{Q \not\in \D_{\lambda'(k)}^N} 
\mathbf{F}(Q) = 0,  
\end{align*}
where $Q \notin \D_{\lambda'(k)}^N$ means $Q \in \D_{\lambda'(k)} \setminus \D_{\lambda'(k)}^N$, and 
\begin{align*}
\D_{\lambda'(k)}^N 
:= \big\{Q \in \D_{\lambda'(k)}:  \, 
&2^{-N} \le \ell(Q_1) \le 2^N, \, \rd(Q_1, 2^N \I) \le N, 
\\ \nonumber 
&2^{-N} \le \ell(Q_2) \le 2^N, \, \rd(Q_2, 2^N  \I) \le N, \, 
\rd(Q_3, 2^{2N} \I) \le 2N \big\}. 
\end{align*}
Moreover, any adjoint $(\mathbf{S}_{\D}^{k, \bm{l}})^{j_1^*, j_{2, 3}^*}$, $j_1, j_{2, 3} \in \{0, 1, 2\}$, is also called a \emph{compact bilinear Zygmund shift of complexity $\bm{l}$}. Here, the adjoint $j_1^*$ and $j_{2, 3}^*$ means that, for example, functions $h_{J_1}^0$ and $h_{K_1}$ switch places in the case $j_1 = 2$, functions $h_{I_{2, 3}}^0$ and $h_{K_{2, 3}}$ switch places in the case $j_{2, 3} = 1$, and the corresponding functions do not switch places in the case $j_1 = 0$ or $j_{2, 3} = 0$.
\end{definition}

Finally, let us introduce the compact dyadic representation of bilinear Calder\'{o}n--Zygmund operators associated with Zygmund dilations. 

\begin{definition}\label{def:Brepre}
Given $\vartheta \in (0, 2]$ and $\theta, \delta_1, \delta_{2, 3} \in (0, 1]$, we say that a $(\mathfrak{D}_{\vartheta}, D_{\theta}, \delta_1, \delta_{2, 3})$-BCZZ operator $T$ admits a \emph{compact bilinear dyadic representation} if there exists a constant $C_0 = C_0(T) \in (0,  \infty)$ so that for all compactly supported and bounded functions $f, g, h$ on $\R^3$, 
\begin{align*}
\langle T(f, g), h \rangle
&= C_0 \, \E_{\w} 
\sum_{\substack{k = (k_1, k_2, k_2) \\ k_1, k_2, k_3 \ge 2}} 
\sum_{m=1}^{m(k)} \varphi(k) 
\langle \mathbf{S}_{m, \D_{\w}}^k (f, g), h \rangle, 
\end{align*} 
where
\begin{align*}
& m(k) \lesssim (|k| + 1)^2, \quad 
\varphi(k) := 2^{-k_1 \delta_1} 
2^{- k_2 \min\{\delta_{2, 3}, \, \theta\}} 
2^{- \max\{k_3 - k_1 - k_2, \, 0\} \theta}, 
\\
& \text{$\mathbf{S}_{m, \D_{\w}}^k$ is a finite sum of compact bilinear Zygmund shifts of complexity $\bm{l}$ on $\D_{\w}$},
\\
& \text{and $\bm{l} = (l_1, l_2, l_3) \in \R^9$ with $l_j^i \le k_i + 1$ for all $i, j = 1, 2, 3$}.
\end{align*}
\end{definition}

\section{A compact bilinear dyadic representation}\label{sec:Brepre}
Now let us demonstrate Theorem \ref{thm:Brepre}. Our proof utilizes a hybrid of the arguments in Section \ref{sec:cdr} and \cite{ALM, CLSY}. Given intervals $I_i, J_i, K_i \subset \R$, the sidelength of them satisfies one of the following cases: (1) $\ell(I_i), \ell(J_i) > \ell(K_i)$, (2) $\ell(I_i), \ell(K_i) > \ell(J_i)$, (3) $\ell(J_i), \ell(K_i) > \ell(I_i)$, (4) $\ell(I_i) > \ell(J_i) = \ell(K_i)$, (5) $\ell(J_i) > \ell(I_i) = \ell(K_i)$, (6) $\ell(K_i) > \ell(I_i) = \ell(J_i)$, (7) $\ell(I_i) = \ell(J_i) = \ell(K_i)$. Then as done in \eqref{TDE-1}, there holds 
\begin{align*}
\langle T(f_1, f_2), f_3 \rangle 
= \sum_{\substack{I_1, J_1, K_1 \in \D_1 \\ \ell(I_1) = \ell(J_1) = \ell(K_1)}} 
& \big[\langle T (E_{I_1} f_1, E_{J_1} f_2), \Delta_{K_1} f_3 \rangle 
+ \langle T (E_{I_1} f_1, \Delta_{I_2} f_2), E_{K_1} f_3 \rangle 
\\
& + \langle T (\Delta_{I_1} f_1, E_{J_1} f_2), E_{K_1} f_3 \rangle 
+ \langle T (E_{I_1} f_1, \Delta_{J_1} f_2), \Delta_{K_1} f_3 \rangle 
\\
& + \langle T (\Delta_{I_1} f_1, E_{J_1} f_2), \Delta_{K_1} f_3 \rangle 
+ \langle T (\Delta_{I_1} f_1, \Delta_{J_1} f_2), E_{K_1} f_3 \rangle  
\\
& + \langle T (\Delta_{I_1} f_1, \Delta_{J_1} f_2), \Delta_{K_1} f_3 \rangle \big]. 
\end{align*}
Similarly, 
\begin{align*}
\langle T (f_1, f_2), f_3 \rangle 
= \sum_{\substack{I_{2, 3}, J_{2, 3}, K_{2, 3} \in \D_{\ell(I_1)}^{2, 3} \\ \ell(I_2) = \ell(J_2) = \ell(K_2)}}  
& \big[ \langle T (E_{I_{2, 3}} f_1, E_{J_{2, 3}} f_2), \Delta_{K_{2, 3}} f_3 \rangle
+ \langle T (E_{I_{2, 3}} f_1, \Delta_{J_{2, 3}} f_2), E_{K_{2, 3}} f_3 \rangle
\\
& + \langle T (\Delta_{I_{2, 3}} f_1, E_{J_{2, 3}} f_2), E_{K_{2, 3}} f_3 \rangle
+ \langle T (E_{I_{2, 3}} f_1, \Delta_{J_{2, 3}} f_2), \Delta_{K_{2, 3}} f_3 \rangle
\\
& + \langle T (\Delta_{I_{2, 3}} f_1, E_{J_{2, 3}} f_2), \Delta_{K_{2, 3}} f_3 \rangle
+ \langle T (\Delta_{I_{2, 3}} f_1, \Delta_{J_{2, 3}} f_2), E_{K_{2, 3}} f_3 \rangle
\\
& + \langle T (\Delta_{I_{2, 3}} f_1, \Delta_{J_{2, 3}} f_2), \Delta_{K_{2, 3}} f_3 \rangle \big]. 
\end{align*}
Combining the two equalities above and then taking an expectation, we have $\langle T (f_1, f_2), f_3 \rangle = \sum_{j=1}^{49} \E_{\w} \mathscr{S}_j(\w)$. By symmetry, it is enough to treat the first term: 
\begin{align*}
\mathscr{S}_1(\w)
& := \sum_{\substack{I, J, K \in \D_{\w, \Z} \\ \ell(I) = \ell(J) = \ell(K)}} 
\langle T (E_{I_1} E_{I_{2, 3}} f_1, E_{J_1} E_{J_{2, 3}} f_2), 
\Delta_{K_1} \Delta_{K_{2, 3}} f_3 \rangle 
\\
& = \sum_{\substack{I, J, K \in \D_{\w, \Z} \\ \ell(I) = \ell(J) = \ell(K)}} 
\langle T (h_I^0, h_J^0), h_{K, \Z} \rangle 
\langle f_1, h_I^0 \rangle 
\langle f_2, h_J^0 \rangle
\langle f_3, h_{K, \Z} \rangle.
\end{align*}
Note that by the cancellation condition (cf. Definition \ref{def:Bcancellation}), 
\begin{align*}
\sum_{\substack{I, J, K \in \D_{\w, \Z} \\ \ell(I) = \ell(J) = \ell(K)}} 
\langle T (h_I^0, h_J^0), h_{K, \Z} \rangle 
\langle f_1, h_{K_1 \times I_{2, 3}}^0 \rangle 
\langle f_2, h_{K_1 \times J_{2, 3}}^0 \rangle
\langle f_3, h_{K, \Z} \rangle = 0,
\end{align*}
because the above can be rewritten as
\begin{align*}
& \sum_{\substack{K \in \D_{\w, \Z}, \, I_{2, 3}, J_{2, 3} \in \D_{\ell(K_1)}^{2, 3} \\ \ell(I_{2, 3}) = \ell(J_{2, 3}) = \ell(K_{2, 3})}} 
\sum_{K'_{2, 3} \in \ch(K_{2, 3})} 
\langle T (1 \otimes \mathbf{1}_{I_{2, 3}}, 1 \otimes \mathbf{1}_{J_{2, 3}}), 
h_{K_1} \otimes \mathbf{1}_{K'_{2, 3}} \rangle 
\\
& \qquad\qquad\qquad\qquad \times 
|K|^{-1} \langle h_{K_{2, 3}} \rangle_{K'_{2, 3}}
\langle f_1, h_{K_1 \times I_{2, 3}}^0 \rangle 
\langle f_2, h_{K_1 \times J_{2, 3}}^0 \rangle
\langle f_3, h_{K, \Z} \rangle.
\end{align*}
Likewise,  
\begin{align*}
& \sum_{\substack{I, J, K \in \D_{\w, \Z} \\ \ell(I) = \ell(J) = \ell(K)}} 
\langle T (h_I^0, h_J^0), h_{K, \Z} \rangle 
\langle f_1, h_{I_1 \times K_{2, 3}}^0 \rangle 
\langle f_2, h_{J_1 \times K_{2, 3}}^0 \rangle
\langle f_3, h_{K, \Z} \rangle = 0
\end{align*}
and
\begin{align*}
& \sum_{\substack{I, J, K \in \D_{\w, \Z} \\ \ell(I) = \ell(J) = \ell(K)}} 
\langle T (h_I^0, h_J^0), h_{K, \Z} \rangle 
\langle f_1, h_K^0 \rangle 
\langle f_2, h_K^0 \rangle
\langle f_3, h_{K, \Z} \rangle = 0.
\end{align*}
Thus, 
\begin{align*}
\mathscr{S}_1(\w)
= \sum_{\substack{I, J, K \in \D_{\w, \Z} \\ \ell(I) = \ell(J) = \ell(K)}} 
\langle T (h_I^0, h_J^0), h_{K, \Z} \rangle \, \Xi(I, J, K)
\end{align*}
where
\begin{align*}
\Xi(I, J, K)
:= & \big[\langle f_1, h_I^0 \rangle \langle f_2, h_J^0 \rangle
- \langle f_1, h_{K_1 \times I_{2, 3}}^0 \rangle \langle f_2, h_{K_1 \times J_{2, 3}}^0 \rangle
\\
& \, \, - \langle f_1, h_{I_1 \times K_{2, 3}}^0 \rangle \langle f_2, h_{J_1 \times K_{2, 3}}^0 \rangle
+ \langle f_1, h_K^0 \rangle \langle f_2, h_K^0 \rangle \big]
\langle f_3, h_{K, \Z} \rangle.
\end{align*}
Analogously to \eqref{indep-1}, by independence and \eqref{indep-2}, 
\begin{align*}
\E_{\w} \mathscr{S}_1(\w)
& = 8 \, \E_{\w} \sum_{k_1, k_2, k_3 \ge 2}  
\sum_{\substack{I, J \in \D_{\w, \Z}, \, K \in \D_{\w, \Z}^k \\ \ell(I) = \ell(J) = \ell(K) \\ 2^{k_i-3} < \max\{\rd(I_i, K_i), \rd(J_i, K_i)\} \le 2^{k_i-2}}} 
\langle T (h_I^0, h_J^0), h_{K, \Z} \rangle \, \Xi(I, J, K)
\\
& = 8 C_0 \, \E_{\w} \sum_{\substack{k = (k_1, k_2, k_3) \\ k_1, k_2, k_3 \ge 2}} 
\varphi(k) 
\sum_{Q \in \D_{\lambda(k)}} 
\sum_{\substack{I, J, K \in \D_{\w, \Z} \\ I^{(k)} = J^{(k)} = K^{(k)} = Q}} 
a_{IJKQ} \, \Xi(I, J, K)
\end{align*}
where $\lambda(k) := 2^{-k_1 - k_2 + k_3}$ and
\begin{align*}
a_{IJKQ}
:= \frac{\langle T (h_I^0, h_J^0), h_{K, \Z} \rangle}{C_0 \, \varphi(k)}
\end{align*}
if $I, J \in \D_{\w, \Z}$ and $K \in \D_{\w, \Z}^k$ satisfy $\ell(I) = \ell(J) = \ell(K)$, $2^{k_i-3} < \max\{\rd(I_i, K_i), \rd(J_i, K_i)\} \le 2^{k_i-2}$, $i=1, 2, 3$, and $a_{IJKQ} = 0$ otherwise. 

We are going to estimate $\langle T (h_I^0, h_J^0), h_{K, \Z} \rangle$ as in \eqref{GGIJ}. Once it is established, one can use a similar argument in \eqref{KDN}--\eqref{MKK} and the same decomposition of $\Xi(I, J, K)$ as in \cite[Section 5]{ALM} to obtain bilinear Zygmund shifts as desired. The details are left to the reader. 

Although we still need some good properties of admissible functions and kernel estimates as in Section \ref{sec:refined}, we do not list them explicitly and will use them below directly. Moreover, given $I, J, K \in \D_{\Z}$, the dyadic intervals $I_1, J_1, K_1$ must belong to one of the cases:
\begin{itemize}
\item Separated: $\max\{\rd(I_1, K_1), \rd(J_1, K_1)\} > 1$,

\item Adjacent: $\max\{\rd(I_1, K_1), \rd(J_1, K_1)\} = 1$ and $I_1 \cap K_1 = \emptyset$ or $J_1 \cap K_1 = \emptyset$,

\item Identical: $I_1 = J_1 = K_1$,
\end{itemize}
and the dyadic rectangles $I_{2, 3}, J_{2, 3}, K_{2, 3}$ must lie in one of the scenarios:
\begin{itemize}
\item Separated: $\max\{\rd(I_{2, 3}, K_{2, 3}), \rd(J_{2, 3}, K_{2, 3})\} > 1$,

\item Adjacent: $\max\{\rd(I_{2, 3}, K_{2, 3}), \rd(J_{2, 3}, K_{2, 3})\} = 1$ and $I_{2, 3} \cap K_{2, 3} = \emptyset$ or $J_{2, 3} \cap K_{2, 3} = \emptyset$,

\item Identical: $I_{2, 3} = J_{2, 3} = K_{2, 3}$.
\end{itemize}
As seen in Section \ref{sec:cdr}, an easy situation occurs in the Separated case. Thus, by symmetry, it suffices to handle three cases: Adjacent/Adjacent, Adjacent/Identical, and Identical/Identical.

In what follows, let $I, J \in \D_{\Z}$ and $K \in \D_{\Z}^k$ with $\ell(I) = \ell(J) = \ell(K)$, and let $Q$ denote a common dyadic ancestor of $I, J, K$, i.e., $Q = I^{(k)} = J^{(k)} = K^{(k)}$, where $k = (k_1, k_2, k_3)$. We my assume that $\vartheta \in (0, 1)$ because $\mathfrak{D}_{\vartheta}$ is monotone decreasing with respect to $\vartheta$.

\subsection{Adjacent/Adjacent}\label{sec:BAA}
Without loss of generality, we may assume that $I_1 \cap K_1 = \emptyset$ and $I_{2, 3} \cap K_{2, 3} = \emptyset$. 
In this case, by the size condition of the compact full kernel representation, 
\begin{align}\label{TAA}
|\langle T (h_I^0, h_J^0), h_{K, \Z} \rangle|
\le |I|^{-\frac32} \int_J \int_I \int_K
\mathfrak{D}_{\vartheta}(x - y, x - z) 
\prod_{i=1}^3 \frac{F_i(x_i, y_i, z_i)}{(|x_i - y_i| + |x_i - z_i|)^2} \, dx \, dy \, dz,
\end{align}
where
\begin{align}\label{FIM}
F_i(x_i, y_i, z_i) 
& := F_{i, 1}(|x_i - y_i| + |x_i - z_i|) F_{i, 2}(|x_i - y_i| + |x_i - z_i|) 
\\ \nonumber
&\qquad\times F_{i, 3} \bigg(1 + \frac{|x_i + y_i| + |x_i + z_i|}{1 + |x_i - y_i| + |x_i - z_i|} \bigg). 
\end{align}
Since $\phi(t) := \big(t + \frac1t \big)^{-\vartheta}$ is a convex function on $\R_+$ (because of $\phi''(t) > 0$), there holds
\begin{align}\label{tttheta}
\bigg(\frac{t_1 + t_2}{t_3} + \frac{t_3}{t_1 + t_2} \bigg)^{-\vartheta} 
\lesssim \bigg(\frac{t_1}{t_3} + \frac{t_3}{t_1} \bigg)^{-\vartheta} 
+ \bigg(\frac{t_2}{t_3} + \frac{t_3}{t_2} \bigg)^{-\vartheta}, 
\quad \forall t_1, t_2, t_3 > 0.
\end{align}
This leads to 
\begin{align*}
\mathfrak{D}_{\vartheta}(x - y, x - z) 
\lesssim G(x, y; z_1, z_2) + G(x, z; y_1, y_2),
\end{align*}
where
\begin{align*}
G(x, y; z_1, z_2)
:= \Bigg[\frac{\prod_{i=1}^2 (|x_i - y_i| + |x_i - z_i|)}{|x_3 - y_3|} 
+ \frac{|x_3 - y_3|}{\prod_{i=1}^2 (|x_i - y_i| + |x_i - z_i|)} \Bigg]^{-\vartheta}.
\end{align*}
If $\mathscr{A}_1$ and $\mathscr{A}_2$ denote the integrals in the right hand side of \eqref{TAA} with $G(x, y; z_1, z_2)$ and $G(x, z; y_1, y_2)$ in place of $\mathfrak{D}_{\vartheta}(x - y, x - z)$ respectively, then 
\begin{align*}
|\langle T (h_I^0, h_J^0), h_{K, \Z} \rangle|
\lesssim \mathscr{A}_1 + \mathscr{A}_2.
\end{align*}

Let us control $\mathscr{A}_1$. For all $(y_i, z_i, x_i) \in I_i \times J_i \times K_i$, the condition $I_i \cup J_i \cup K_i \subset Q_i$ implies $|x_i - y_i| + |x_i - z_i| \le 2 \ell(Q_i)$ and 
\begin{align*}
& 2 \bigg(1 + \frac{|x_i + y_i| + |x_i + z_i|}{1 + |x_i - y_i| + |x_i - z_i|}\bigg)
\ge 1 + \frac{2 (|y_i| + |z_i|)}{1 + |x_i - y_i| + |x_i - z_i|}
\\ 
& \ge 1 + \frac{4 \d(Q_i, \I)}{1 + 2 \ell(Q_i)}
\ge 1 + \frac{\d(Q_i, \I)}{\max\{\ell(Q_i), 1\}}
= \rd(Q_i, \I),
\end{align*}
which along with the monotonicity of $F_{i, 1}, F_{i, 2}, F_{i, 3}$ (see \eqref{list:P2}) gives
\begin{align}
\label{FXZ-1}
F_{i, 1}(|x_i - y_i| + |x_i - z_i|) & \le F_{i, 1}(\ell(Q_i)),
\\
\label{FXZ-2}
F_{i, 2}(|x_i - y_i| + |x_i - z_i|) & \le F_{i, 2}(|x_i - y_i|),
\\
\label{FXZ-3}
F_{i, 3} \bigg(1 + \frac{|x_i + y_i| + |x_i + z_i|}{1 + |x_i - y_i| + |x_i - z_i|} \bigg)
& \le F_{i, 3}(\rd(Q_i, \I)).
\end{align}
Then it follows from \eqref{FXZ-1}--\eqref{FXZ-3} that 
\begin{align}\label{AAF}
\mathscr{A}_1
\le |I|^{-\frac32} \int_I \int_K \mathscr{A}_1(x, y) 
\prod_{i=1}^3 F_{i, 1}(\ell(I_i)) F_{i, 2}(|x_i - y_i|) F_{i, 3}(\rd(I_i, \I)) \, dx \, dy, 
\end{align}
where
\begin{align*}
\mathscr{A}_1(x, y)
& := \int_J \frac{\Big[\frac{\prod_{i=1}^2 (|x_i - y_i| + |x_i - z_i|)}{|x_3 - y_3|} 
+ \frac{|x_3 - y_3|}{\prod_{i=1}^2 (|x_i - y_i| + |x_i - z_i|)} \Big]^{-\vartheta}}{\prod_{i=1}^3 (|x_i - y_i| + |x_i - z_i|)^2} \, dz.
\end{align*}
Note that 
\begin{align*}
\sup_{r>0} \int_{\Rn} \frac{r^{\alpha} |f(x)|}{(r + |x - x_0|)^{n+\alpha}} \, dx
\lesssim Mf(x_0), \quad \alpha > 0,
\end{align*}
which implies 
\begin{align*}
\mathscr{A}_1(x, y)
& \lesssim \int_{J_2} \int_{J_1} \frac{\Big[\frac{\prod_{i=1}^2 (|x_i - y_i| + |x_i - z_i|)}{|x_3 - y_3|} 
+ \frac{|x_3 - y_3|}{\prod_{i=1}^2 (|x_i - y_i| + |x_i - z_i|)} \Big]^{-\vartheta}}{|x_3 - y_3| \prod_{i=1}^2 (|x_i - y_i| + |x_i - z_i|)^2} \, dz_1 \, dz_2 
\\
& \le \mathbf{1}_{\{|x_1 - y_1| |x_2 - y_2| \ge |x_3 - y_3|\}} 
\int_{J_2} \int_{J_1} \frac{\Big[\frac{\prod_{i=1}^2 (|x_i - y_i| + |x_i - z_i|)}{|x_3 - y_3|}\Big]^{-\vartheta}}{|x_3 - y_3| \prod_{i=1}^2 (|x_i - y_i| + |x_i - z_i|)^2} \, dz_1 \, dz_2 
\\
&\quad + \mathbf{1}_{\{|x_1 - y_1| |x_2 - y_2| < |x_3 - y_3|\}} 
\int_{J_2} \int_{J_1} \frac{\Big[\frac{|x_3 - y_3|}{\prod_{i=1}^2 (|x_i - y_i| + |x_i - z_i|)} \Big]^{-\vartheta}}{|x_3 - y_3| \prod_{i=1}^2 (|x_i - y_i| + |x_i - z_i|)^2} \, dz_1 \, dz_2 
\\
& = \mathbf{1}_{\{|x_1 - y_1| |x_2 - y_2| \ge |x_3 - y_3|\}} 
\frac{1}{|x_3 - y_3|^{1-\vartheta}} 
\prod_{i=1}^2 \int_{J_i} \frac{dz_i}{(|x_i - y_i| + |x_i - z_i|)^{2+\vartheta}}
\\
& \quad + \mathbf{1}_{\{|x_1 - y_1| |x_2 - y_2| < |x_3 - y_3|\}} 
\frac{1}{|x_3 - y_3|^{1+\vartheta}} 
\prod_{i=1}^2 \int_{J_i} \frac{dz_i}{(|x_i - y_i| + |x_i - z_i|)^{2-\vartheta}}
\\
& \lesssim \mathbf{1}_{\{|x_1 - y_1| |x_2 - y_2| \ge |x_3 - y_3|\}} 
\frac{1}{|x_1 - y_1|^{1+\vartheta}} \frac{1}{|x_2 - y_2|^{1+\vartheta}}  \frac{1}{|x_3 - y_3|^{1-\vartheta}} 
\\
& \quad + \mathbf{1}_{\{|x_1 - y_1| |x_2 - y_2| < |x_3 - y_3|\}} 
\frac{1}{|x_1 - y_1|^{1-\vartheta}} \frac{1}{|x_2 - y_2|^{1-\vartheta}}  \frac{1}{|x_3 - y_3|^{1+\vartheta}} 
\\
& \simeq \frac{\Big[ \frac{|x_1 - y_1| |x_2 - y_2|}{|x_3 - y_3|} + \frac{|x_3 - y_3|}{|x_1 - y_1| |x_2 - y_2|} \Big]^{-\vartheta}}{\prod_{i=1}^3 |x_i - y_i|}
= D_{\vartheta}(x-y) \prod_{i=1}^3 \frac{1}{|x_i - y_i|},
\end{align*}
where we have used the condition $\vartheta \in (0, 1)$. Following the proof of \eqref{def:RIJ}, we use \eqref{AAF} to deduce
\begin{align*}
\mathscr{A}_1 
& \lesssim |I|^{-\frac32} \prod_{i=1}^3 F_{i, 1}(\ell(I_i)) F_{i, 3}(\rd(I_i, \I))
\int_I \int_K D_{\vartheta}(x-y) \prod_{i=1}^3 \frac{F_{i, 2}(|x_i - y_i|)}{|x_i - y_i|} \, dx \, dy 
\\
& \lesssim F_1(I_1, I_1) F_2(I_2, I_2) F_3(I_3, I_3)  \, |I|^{-\frac12},
\end{align*}
where $F_i(I_i, I_i)$ is defined in \eqref{def:FIK-1}. Symmetrically, $\mathscr{A}_2$ has the same bound. Therefore, 
\begin{align}\label{TAA-2}
|\langle T (h_I^0, h_J^0), h_{K, \Z} \rangle|
\lesssim F_1(I_1, I_1) F_2(I_2, I_2) F_3(I_3, I_3)  \, |I|^{-\frac12}.
\end{align}

\subsection{Adjacent/Identical}\label{sec:BAI}
We rewrite
\begin{align*}
\langle T (h_I^0, h_J^0), h_{K, \Z} \rangle
= \sum_{I'_{2, 3}, I''_{2, 3}, I'''_{2, 3} \in \ch(I_{2, 3})} 
|I|^{-1} \langle h_{I_1} \rangle_{I'''_1} 
\langle h_{I_{2, 3}} \rangle_{I'''_{2, 3}}
\langle T (\mathbf{1}_{I_1} \otimes \mathbf{1}_{I'_{2, 3}}, 
\mathbf{1}_{J_1} \otimes \mathbf{1}_{I''_{2, 3}}), 
\mathbf{1}_{K_1} \otimes \mathbf{1}_{I'''_{2, 3}} \rangle.
\end{align*}
Since the case $I'_{2, 3} \ne I'''_{2, 3}$ or $I''_{2, 3} \ne I'''_{2, 3}$ is similar to the Adjacent/Adjacent case, we have
\begin{align}\label{AIT}
|\langle T (\mathbf{1}_{I_1} \otimes \mathbf{1}_{I'_{2, 3}}, 
\mathbf{1}_{J_1} \otimes \mathbf{1}_{I''_{2, 3}}), 
\mathbf{1}_{K_1} \otimes \mathbf{1}_{I'''_{2, 3}} \rangle| 
\lesssim F_1(I_1, I_1) F_2(I_2, I_2) F_3(I_3, I_3) \, |I|.
\end{align}

In the case $I'_{2, 3} = I''_{2, 3} = I'''_{2, 3}$, we use the size condition of the compact partial representation to deduce
\begin{equation}\label{AIT-1}
\begin{aligned}
& |\langle T (\mathbf{1}_{I_1} \otimes \mathbf{1}_{I'_{2, 3}}, 
\mathbf{1}_{J_1} \otimes \mathbf{1}_{I''_{2, 3}}), 
\mathbf{1}_{K_1} \otimes \mathbf{1}_{I'''_{2, 3}} \rangle|
\\
& \le C(\mathbf{1}_{I'_{2, 3}}, \mathbf{1}_{I'_{2, 3}}, \mathbf{1}_{I'_{2, 3}}) 
\int_{J_1} \int_{I_1} \int_{K_1} 
\frac{F_1(x_1, y_1, z_1)}{(|x_1 - y_1| + |x_1 - z_1|)^2} \, dx_1 \, dy_1 \, dz_1,
\end{aligned}
\end{equation}
where $F_1(x_1, y_1, z_1)$ is given in \eqref{FIM}. By symmetry, we may assume that $I_1 \cap K_1 = \emptyset$. Then in the current setting, $k_1 = 2$ and $K_1 \subset 3I_1 \setminus I_1$. Note that for all $(x_1, y_1, z_1) \in K_1 \times I_1 \times J_1$, there hold $|x_1 - y_1| + |x_1 - z_1| \le 2 \ell(Q_1) = 8 \ell(I_1)$ and 
\begin{equation}\label{AIT-2}
\begin{aligned}
& 2 \bigg(1 + \frac{|x_1 + y_1| + |x_1 + z_1|}{1 + |x_1 - y_1| + |x_1 - z_1|}\bigg)
\ge 1 + \frac{2 (|y_1| + |z_1|)}{1 + |x_1 - y_1| + |x_1 - z_1|}
\\ 
& \ge 1 + \frac{4 \d(Q_1, \I)}{1 + 2 \ell(Q_1)}
\ge 1 + \frac{\d(Q_1, \I)}{\max\{\ell(Q_1), 1\}}
= \rd(Q_1, \I).
\end{aligned}
\end{equation}
Let $\eta \in (0, 1)$. Hence, by \eqref{AIT-1}--\eqref{AIT-2} and the monotonicity of $F_{1, 1}, F_{1, 2}, F_{1, 3}$,
\begin{align}\label{AIT-3}
& |\langle T (\mathbf{1}_{I_1} \otimes \mathbf{1}_{I'_{2, 3}}, 
\mathbf{1}_{J_1} \otimes \mathbf{1}_{I''_{2, 3}}), 
\mathbf{1}_{K_1} \otimes \mathbf{1}_{I'''_{2, 3}} \rangle|
\\ \nonumber 
& \lesssim F_2(I'_2) |I_2| \, F_3(I'_3) |I_3| \, 
F_{1, 1}(\ell(Q_1)) F_{1, 3}(\rd(Q_1, \I))
\\ \nonumber 
&\quad\times  \int_{I_1} \int_{3I_1 \setminus I_1} 
\frac{F_{1, 2}(|x_1 - y_1|)}{|x_1 - y_1|^{1 + \eta}} 
\bigg(\int_{J_1} \frac{dz_1}{|x_1 - z_1|^{1 - \eta}}\bigg) \, dx_1 \, dy_1 
\\ \nonumber
&\lesssim F_{1, 1}(\ell(I_1)) \widetilde{F}_{1, 2}(\ell(I_1)) F_{1, 3}(\rd(I_1, \I)) 
F_2(I'_2) F_3(I'_3)  \, |I|,
\end{align}
provided that Lemma \ref{lem:diag} gives 
\begin{align*}
&\int_{I_1} \int_{3I_1 \setminus I_1} 
\frac{F_{1, 2}(|x_1 - y_1|)}{|x_1 - y_1|^{1+\eta}} dx_1 \, dy_1 
\\
&\le \bigg[\int_{I_1} \int_{3I_1 \setminus I_1} 
\frac{dx_1 \, dy_1}{|x_1 - y_1|^{p(1 + \eta)}} \bigg]^{\frac1p}
\bigg[\int_{I_1} \int_{3I_1 \setminus I_1} 
F_{1, 2}(|x_1 - y_1|)^{p'} dx_1 \, dy_1 \bigg]^{\frac{1}{p'}} 
\\ 
&\lesssim \ell(I_1)^{\frac{2}{p}-1-\eta}  
\widetilde{F}_{1, 2}(\ell(I_1)) \ell(I_1)^{\frac{2}{p'}}
\le \widetilde{F}_{1, 2}(\ell(I_1)) \, \ell(I_1)^{1-\eta},  
\end{align*}
where $1 < p < \frac{2}{1+\eta}$. As a consequence of \eqref{AIT} and \eqref{AIT-3}, there holds
\begin{align*}
|\langle T (h_I^0, h_J^0), h_{K, \Z} \rangle|
\lesssim F_1(I_1, I_1) \widehat{F}_2(I_2, I_2) \widehat{F}_3(I_3, I_3) \, |I|^{-\frac12},
\end{align*}
where $F_i(\cdot, \cdot)$ and $\widehat{F}_i(\cdot, \cdot)$ are defined in \eqref{def:FIK-1}--\eqref{def:FIK-2}.

\subsection{Identical/Identical}\label{sec:BII}

\begin{align*}
\langle T (h_I^0, h_I^0), h_{I, \Z} \rangle
= \sum_{\substack{I'_1, I''_1, I'''_1 \in \ch(I_1) \\ I'_{2, 3}, I''_{2, 3}, I'''_{2, 3} \in \ch(I_{2, 3})}} 
|I|^{-1} \langle h_{I_1} \rangle_{I'''_1} 
\langle h_{I_{2, 3}} \rangle_{I'''_{2, 3}}
\langle T (\mathbf{1}_{I'_1} \otimes \mathbf{1}_{I'_{2, 3}}, 
\mathbf{1}_{I''_1} \otimes \mathbf{1}_{I''_{2, 3}}), 
\mathbf{1}_{I'''_1} \otimes \mathbf{1}_{I'''_{2, 3}} \rangle.
\end{align*}
If $I'_1 = I''_1 = I'''_1$ and $I'_{2, 3} = I''_{2, 3} = I'''_{2, 3}$, then the weak compactness property implies
\begin{align*}
|\langle T (\mathbf{1}_{I'_1} \otimes \mathbf{1}_{I'_{2, 3}}, 
\mathbf{1}_{I''_1} \otimes \mathbf{1}_{I''_{2, 3}}), 
\mathbf{1}_{I'''_1} \otimes \mathbf{1}_{I'''_{2, 3}} \rangle|
\lesssim F_1(I'_1) F_2(I'_2) F_3(I'_3) \, |I|.
\end{align*}
If $I'_1 \ne I'''_1$ or $I''_1 \ne I'''_1$ and $I'_{2, 3} = I''_{2, 3} = I'''_{2, 3}$, then it is similar to the second scenario in Section \ref{sec:BAI}. Hence, \eqref{AIT-3} gives
\begin{align*}
|\langle T (\mathbf{1}_{I'_1} \otimes \mathbf{1}_{I'_{2, 3}}, 
\mathbf{1}_{I''_1} \otimes \mathbf{1}_{I''_{2, 3}}), 
\mathbf{1}_{I'''_1} \otimes \mathbf{1}_{I'''_{2, 3}} \rangle|
\lesssim F_1(I_1, I_1) F_2(I'_2) F_3(I'_3) \, |I|.
\end{align*}
If $I'_1 = I''_1 = I'''_1$ and $I'_{2, 3} \ne I'''_{2, 3}$ or $I''_{2, 3} \ne I'''_{2, 3}$, it lies in the Identical/Adjacent case, which is symmetrical to the Adjacent/Identical case. Thus,
\begin{align*}
|\langle T (\mathbf{1}_{I'_1} \otimes \mathbf{1}_{I'_{2, 3}}, 
\mathbf{1}_{I''_1} \otimes \mathbf{1}_{I''_{2, 3}}), 
\mathbf{1}_{I'''_1} \otimes \mathbf{1}_{I'''_{2, 3}} \rangle|
\lesssim F_1(I'_1) F_2(I_2, I_2) F_3(I_3, I_3)  \, |I|.
\end{align*}
In addition, the case $I'_1 \ne I'''_1$ or $I''_1 \ne I'''_1$ and $I'_{2, 3} \ne I'''_{2, 3}$ or $I''_{2, 3} \ne I'''_{2, 3}$  is similar to the Adjacent/Adjacent case, which leads to
\begin{align*}
|\langle T (\mathbf{1}_{I'_1} \otimes \mathbf{1}_{I'_{2, 3}}, 
\mathbf{1}_{I''_1} \otimes \mathbf{1}_{I''_{2, 3}}), 
\mathbf{1}_{I'''_1} \otimes \mathbf{1}_{I'''_{2, 3}} \rangle|
\lesssim F_1(I_1, I_1) F_2(I_2, I_2) F_3(I_3, I_3) \, |I|.
\end{align*}
Collecting the estimates above, we conclude
\begin{align*}
|\langle T (h_I^0, h_I^0), h_{I, \Z} \rangle|
\lesssim \widehat{F}_1(I_1, I_1) \widehat{F}_2(I_2, I_2) \widehat{F}_3(I_3, I_3) \, |I|^{-\frac12}.
\end{align*}

\section{Compactness of bilinear operators}\label{sec:Bcpt}
In this section, we would like to prove Theorems \ref{thm:Bcpt} and \ref{thm:BSD-cpt}.

\subsection{Proof of Theorem \ref{thm:Bcpt}}
Fix $p, p_1, p_2 \in (1, \infty)$ such that $\frac1p = \frac{1}{p_1} + \frac{1}{p_2}$. Define 
\begin{align*}
T^N := \sum_{k \in \N^3: |k| \le N} \sum_{m=1}^{m(k)}
\varphi(k) \, \E_{\w} \mathbf{S}_{m, \D_{\w}}^k, \quad N \ge 1.
\end{align*}
It was shown in the proof of \cite[Proposition 6.2]{ALM} that 
\begin{align}\label{trr}
\sup_{\w} \|\mathbf{S}_{\D_{\w}}^{k, \bm{l}}\|_{L^{s_1} \times L^{s_2} \to L^s} 
\lesssim 2^{k_1 \eta} (|k| +1)^2, 
\end{align}
for all $\frac1s = \frac{1}{s_1} + \frac{1}{s_2}$ with $s_1, s_2 \in (1, \infty)$, where $\eta \in (0, 1)$ and the implicit constant is independent of $k$. Pick $\eta \in (0, \delta_1)$. Then by Theorem \ref{thm:Brepre}, \eqref{trr}, and Lemma \ref{lem:kkk} applied to $(\rho, \alpha_0, \alpha_1, \alpha_2, \alpha_3) = \big(4, \theta, \delta_1 - \eta, \min\{\delta_{2, 3}, \theta\}, 0 \big)$, we deduce  
\begin{equation}\label{trr-1}
\begin{aligned}
& \|T - T^N\|_{L^{p_1} \times L^{p_2} \to L^p} 
\le \sum_{|k| > N} \sum_{m=1}^{m(k)} 
\varphi(k) \, \E_{\w} \|\mathbf{S}_{m, \D_{\w}}^k\|_{L^{p_1} \times L^{p_2} \to L^p} 
\\
& \lesssim \sum_{|k| > N} 2^{k_1 \eta} (|k| + 1)^4 \varphi(k)
\lesssim N^{16} 2^{-N \min\{\delta_1 - \eta, \, \delta_{2, 3}, \, \theta\}/6} \to 0, 
\end{aligned}
\end{equation}
as $N \to \infty$. In addition, Theorem \ref{thm:BSD-cpt} gives
\begin{align}\label{trr-2}
\text{$T^N$ is compact from $L^{p_1}(\R^3) \times L^{p_2}(\R^3)$ to $L^p(\R^3)$}.
\end{align}
Hence, \eqref{trr-1} and \eqref{trr-2} imply
\begin{align}\label{trr-3}
\text{$T$ is compact from $L^{p_1}(\R^3) \times L^{p_2}(\R^3)$ to $L^p(\R^3)$}.
\end{align}

Now let $r_1, r_2 \in (1, \infty)$ and denote $\frac1r = \frac{1}{r_1} + \frac{1}{r_2}$. To prove the compactness of $T$ from $L^{r_1}(\R^3) \times L^{r_2}(\R^3)$ to $L^r(\R^3)$, in view of \eqref{trr-3}, we may assume that $r_1 \ne p_1$ or $r_2 \ne p_2$. Then choose $q_1, q_2 \in (1, \infty)$ so that the point $\big(\frac{1}{r_1}, \frac{1}{r_2}\big)$ lies in the segment connecting $\big(\frac{1}{p_1}, \frac{1}{p_2} \big)$ and $\big(\frac{1}{q_1}, \frac{1}{q_2}\big)$. That is, 
\begin{align}\label{trr-4}
\frac{1}{r_1} = \frac{1-\theta}{p_1} + \frac{\theta}{q_1}
\quad \text{and} \quad
\frac{1}{r_2} = \frac{1-\theta}{p_2} + \frac{\theta}{q_2}
\quad \text{for some } \theta \in (0, 1).
\end{align}
Let $\frac1q = \frac{1}{q_1} + \frac{1}{q_2}$. Obviously, $\frac1r = \frac{1 - \theta}{p} + \frac{\theta}{q}$. By \cite[Theorem 1.2]{ALM}, 
\begin{align}\label{trr-5}
\text{$T$ is bounded from $L^{q_1}(\R^3) \times L^{q_2}(\R^3)$ to $L^q(\R^3)$},
\end{align}
As a consequence of \eqref{trr-3}--\eqref{trr-5} and Theorem \ref{thm:inter-cpt}, we conclude that $T$ is compact from $L^{r_1}(\R^3) \times L^{r_2}(\R^3)$ to $L^r(\R^3)$. 
\qed

\subsection{Proof of Theorem \ref{thm:BSD-cpt}}
Using \eqref{trr} and the same strategy in the preceding section, we are reduced to proving 
\begin{align}\label{ES44}
\text{$\E_{\w} \mathbf{S}_{\D_{\w}}^{k, \bm{l}}$ is compact from $L^4(\R^3) \times L^4(\R^3)$ to $L^2(\R^3)$}.
\end{align}
By Theorem \ref{thm:KRLp}, \eqref{trr}, and Minkowski's inequality, it suffices to show
\begin{align}
\label{BSDKR-2}
\lim_{A \to \infty} & \sup_{\w} \sup_{\substack{\|f_1\|_{L^4} \le 1 \\ \|f_2\|_{L^4} \le 1}} 
\|\mathbf{S}_{\D_{\w}}^{k, \bm{l}} (f_1, f_2) \, \mathbf{1}_{B(0, A)^c}\|_{L^2} 
= 0, 
\\ 
\label{BSDKR-3}
\lim_{|v| \to 0} & \sup_{\w} \sup_{\substack{\|f_1\|_{L^4} \le 1 \\ \|f_2\|_{L^4} \le 1}}
\|\tau_v \, \mathbf{S}_{\D_{\w}}^{k, \bm{l}} (f_1, f_2)
- \mathbf{S}_{\D_{\w}}^{k, \bm{l}} (f_1, f_2)\|_{L^2} 
= 0. 
\end{align}

To proceed, we define  
\begin{align*}
\mathbf{S}_{\D}^{k, \bm{l}, N} (f_1, f_2)
:= \sum_{\substack{Q \notin \D_{\lambda'(k)}^N, \, I, J, K \in \D \\ I^{(l_1)} = J^{(l_2)} = K^{(l_3)} = Q}}   
a_{IJKQ} \, \langle f_1, h_{I_1} \otimes h_{I_{2, 3}}^0 \rangle \,
\langle f_2, h_{J_1}^0 \otimes h_{J_{2, 3}} \rangle \, h_K, 
\end{align*}
where the restriction on $I, J, K, Q$ is the same as in Definition \ref{def:Bshift}. Fix $k \in \N^3$. Let $A \ge 2^{3|k| + 100}$ and $N := \big[\frac13 (\log_2 A - 3|k| - 1) \big] \ge 10$. Similarly to \eqref{KN-2}, there holds
\begin{align}\label{BKN-2}
\bigcup_{Q \in \D_{\lambda'(k)}^N} Q 
\subset \big\{x \in \R^3: |x_1| \le 2^{2N}, |x_2| \le 2^{2N}, |x_3| \le 2^{3|k| + 3N +1} \big\} 
\subset B(0, A). 
\end{align}
Then, using \eqref{BKN-2} and mimicking the proof of \eqref{trr}, we have  
\begin{align}\label{BKN-3}
\sup_{\w} \|\mathbf{S}_{\D_{\w}}^{k, \bm{l}} (f_1, f_2) \, \mathbf{1}_{B(0, A)^c}\|_{L^2} 
\le \sup_{\w} \|\mathbf{S}_{\D_{\w}}^{k, \bm{l}, N} (f_1, f_2)\|_{L^2} 
\lesssim_k \mathbf{F}_N \|f_1\|_{L^4} \|f_2\|_{L^4}, 
\end{align}
where gives \eqref{BSDKR-2}.

Let us move on to the proof of \eqref{BSDKR-3}. Let $0 < |v| = |v_1| + |v_2| + |v_3| \ll 2^{-100}$ and $a \ge 2$ be an integer chosen later. There exists an integer $N=N(v) \ge 10$ such that $2^{-a (N+1)} \le |v| < 2^{- a N}$. We split  
\begin{align}\label{BHFX} 
\|\tau_v \, \mathbf{S}_{\D_{\w}}^{k, \bm{l}} (f_1, f_2)
- \mathbf{S}_{\D_{\w}}^{k, \bm{l}} (f_1, f_2)\|_{L^2} 
\le \big[\Gamma_{\w}^1(v; f_1, f_2)  + \Gamma_{\w}^2(v; f_1, f_2)  \big], 
\end{align}
where 
\begin{align*}
\Gamma_{\w}^1(v; f_1, f_2) 
&:= \bigg\|\sum_{\substack{Q \notin \D_{\lambda'(k)}^N, \, I, J, K \in \D \\ I^{(l_1)} = J^{(l_2)} = K^{(l_3)} = Q}}   
a_{IJKQ} \, \langle f_1, h_{I_1} \otimes h_{I_{2, 3}}^0 \rangle \,
\langle f_2, h_{J_1}^0 \otimes h_{J_{2, 3}} \rangle 
(\tau_v h_K - h_K) \bigg\|_{L^2}, 
\\ 
\Gamma_{\w}^2(v; f_1, f_2) 
&:= \bigg\|\sum_{\substack{Q \in \D_{\lambda'(k)}^N, \, I, J, K \in \D \\ I^{(l_1)} = J^{(l_2)} = K^{(l_3)} = Q}}   
a_{IJKQ} \, \langle f_1, h_{I_1} \otimes h_{I_{2, 3}}^0 \rangle \,
\langle f_2, h_{J_1}^0 \otimes h_{J_{2, 3}} \rangle
(\tau_v h_K - h_K) \bigg\|_{L^2}.
\end{align*}
The last inequality in \eqref{BKN-3} implies
\begin{equation}\label{BXVF}
\begin{aligned}
\Gamma_{\w}^1(v; f_1, f_2) 
& \le \|\tau_v \mathbf{S}_{\D_{\w}}^{k, \bm{l}, N} (f_1, f_2)\|_{L^2} 
+ \|\mathbf{S}_{\D_{\w}}^{k, \bm{l}, N} (f_1, f_2)\|_{L^2}
\\
& = 2 \|\mathbf{S}_{\D_{\w}}^{k, \bm{l}, N} (f_1, f_2)\|_{L^2} 
\lesssim_k \mathbf{F}_N \|f_1\|_{L^4} \|f_2\|_{L^4},  
\end{aligned}
\end{equation}
where the implicit constant is independent of $\w$, $v$, $f_1$, and $f_2$. Since $\lim_{N \to \infty} \mathbf{F}_N = 0$ and $|v| \to 0$ implies $N(v) \to \infty$, the estimate \eqref{BXVF} leads to 
\begin{align}\label{BHFX-1}
\lim_{|v| \to 0} \sup_{\w} 
\sup_{\substack{\|f_1\|_{L^4} \le 1 \\ \|f_2\|_{L^4} \le 1}} 
\Gamma_{\w}^1(v; f_1, f_2)  = 0. 
\end{align}
In order to analyze $\Gamma_{\w}^2(v; f_1, f_2)$, observe that
\begin{align}\label{Bcar-DN}
& \# \D_{\lambda'(k)}^N 
\lesssim_k \, 2^{3N} \times 2^{3N} \times 2^{5N}
= 2^{11N}, \quad\forall N \ge 10, 
\end{align}
and
\\
\begin{align}\label{BPLP}
\|\tau_v h_K - h_K\|_{L^2}
& \le \|\tau_{v_1} h_{K_1}\|_{L^2} \|\tau_{v_2} h_{K_2}\|_{L^2} 
\|\tau_{v_3} h_{K_3} - h_{K_3}\|_{L^2}
\\ \nonumber
& \quad + \|\tau_{v_1} h_{K_1}\|_{L^2} \|\tau_{v_2} h_{K_2} - h_{K_2}\|_{L^2} \|h_{K_3}\|_{L^2}
\\ \nonumber
& \quad + \|\tau_{v_1} h_{K_1} - h_{K_1}\|_{L^2} \|h_{K_2}\|_{L^2} \|h_{K_3}\|_{L^2}
\\ \nonumber
& \lesssim |v_3|^{\frac12} \ell(K_3)^{-\frac12} 
+ |v_2|^{\frac12} \ell(K_2)^{-\frac12} 
+ |v_1|^{\frac12} \ell(K_1)^{-\frac12}.  
\end{align}
Then by \eqref{Bcar-DN} and \eqref{BPLP},
\begin{align*}
\Gamma_{\w}^2(v; f_1, f_2)
& \lesssim \sum_{\substack{Q \in \D_{\lambda'(k)}^N, \, I, J, K \in \D \\ I^{(l_1)} = J^{(l_2)} = K^{(l_3)} = Q}}    
|Q|^{\frac12} \langle |f_1| \rangle_Q \langle |f_2| \rangle_Q
\|\tau_v h_K - h_K\|_{L^2}
\\
&\lesssim |v|^{\frac12} \sum_{\substack{Q \in \D_{\lambda'(k)}^N, \, I, J, K \in \D \\ I^{(l_1)} = J^{(l_2)} = K^{(l_3)} = Q}}   
\big[\ell(K_1)^{-\frac12} + \ell(K_2)^{-\frac12} + \ell(K_3)^{-\frac12}\big]  
\|f_1\|_{L^4} \|f_2\|_{L^4}
\\
&\lesssim |v|^{\frac12} \# \D_{\lambda'(k)}^N 
\big[ 2^{l_3^1/2} 2^{N/2} + 2^{l_3^2/2} 2^{N/2} + 2^{l_3^3/2} \lambda'(k)^{1/2} 2^N \big]
\|f_1\|_{L^4} \|f_2\|_{L^4}
\\
&\lesssim_k 2^{- (a/2-12)N} \|f_1\|_{L^4} \|f_2\|_{L^4}. 
\end{align*}
Picking $a > 24$, we deduce
\begin{align}\label{BHFX-2}
\lim_{|v| \to 0} \sup_{\w} 
\sup_{\substack{\|f_1\|_{L^4} \le 1 \\ \|f_2\|_{L^4} \le 1}} 
\Gamma_{\w}^2(v; f_1, f_2)  = 0. 
\end{align}
Consequently, \eqref{BSDKR-3} follows from\eqref{BHFX}, \eqref{BHFX-1}, and \eqref{BHFX-2}.  
\qed


\begin{thebibliography}{00}


\bibitem{ALM}E. Airta, K. Li, and H. Martikainen, 
\emph{Zygmund dilations: bilinear analysis and commutator estimates}, 
Trans. Amer. Math. Soc. (2025), to appear.


\bibitem{CZ}A.P. Calder\'{o}n and A. Zygmund,
\emph{On the existence of certain singular integrals},
Acta Math. \textbf{88} (1952), 85--139.


\bibitem{CIRXY}M. Cao, G. Iba\~{n}ez-Firnkorn, I.P. Rivera-R\'{i}os, Q. Xue, and K. Yabuta, 
\emph{A class of multilinear bounded oscillation operators on measure spaces and applications}, 
Math. Ann. \textbf{388} (2024), 3627--3755. 


\bibitem{CLSY}M. Cao, H. Liu, Z. Si, and K. Yabuta,
\emph{A characterization of compactness via bilinear $T1$ theorem},
https://arxiv.org/abs/2404.14013. 


\bibitem{COY}M. Cao, A. Olivo, and K. Yabuta,
\emph{Extrapolation for multilinear compact operators and applications},
Trans. Amer. Math. Soc. \textbf{375} (2022), 5011--5070.


\bibitem{CY}M. Cao and K. Yabuta,
\emph{Dyadic analysis of compactness on product spaces}, 
https://arxiv.org/abs/2410.10304. 


\bibitem{CYY}M. Cao, K. Yabuta, and D. Yang,
\emph{A compact extension of Journ\'{e}'s $T1$ theorem on product spaces},
Trans. Amer. Math. Soc. \textbf{377} (2024), 6251--6309. 


\bibitem{CF80}S.Y.A. Chang and R. Fefferman,
\emph{A continuous version of duality of $\H^1$ with $\BMO$ on the bidisc},
Ann. Math. \textbf{112} (1980), 179--201.


\bibitem{CF85}S.Y.A. Chang and R. Fefferman,
\emph{Some recent developments in Fourier analysis and $\H^p$ theory on product domains},
Bull. Amer. Math. Soc. \textbf{12} (1985), 1--43.


\bibitem{CMP}D. Cruz-Uribe, J.M. Martell, and C. P\'{e}rez, 
\emph{Weights, Extrapolation and the Theory of Rubio de Francia}, 
Operator Theory: Advances and Applications, vol. 215, Birkh\"{a}user/Springer, Basel AG, Basel, 2011.


\bibitem{DJ}G. David and J.-L. Journ\'{e},
\emph{A boundedness criterion for generalized Calder\'{o}n-Zygmund operators},
Ann. Math. \textbf{120} (1984),  371--397.


\bibitem{FJR}E. Fabes, M. Jodeit, and N. Rivi\'{e}re,
\emph{Potential techniques for boundary value problems on $C^1$-domains},
Acta Math. \textbf{141} (1978), 165--186.



\bibitem{Fef86}R. Fefferman,
\emph{Multi-parameter Fourier analysis },
Beijing Lectures in Harmonic Analysis. Annals of Mathematics Studies, vol. 112, pp. 47--130. 
Princeton University Press, Princeton, 1986.


\bibitem{Fef87}R. Fefferman,
\emph{Harmonic analysis on product spaces},
Ann. of Math. \textbf{126} (1987), 109--130.


\bibitem{Fef88}R. Fefferman,
\emph{$A_p$ weights and singular integrals},
Amer. J. Math. \textbf{110} (1988), 975--987.


\bibitem{FP}R. Fefferman and J. Pipher,
\emph{Multiparameter operators and sharp weighted inequalities},
Amer. J. Math. \textbf{11} (1997), 337--369.


\bibitem{FS}R. Fefferman and E. Stein,
\emph{Singular integrals on product spaces},
Adv. Math. \textbf{45} (1982), 117--143.


\bibitem{FL}S. Ferguson and M. Lacey,
\emph{A characterization of product $\BMO$ by commutators},
Acta Math. \textbf{189} (2002), 143--160.


\bibitem{GR}J. Garc{\'i}a-Cuerva and J.L. Rubio de Francia,
\emph{Weighted Norm Inequalities and Related Topics},
North-Holland Mathematics Studies, vol. 116, North-Holland Publishing Co., Amsterdam, 1985.


\bibitem{G} A. Grau de la Herr\'{a}n,
\emph{Comparison of $T1$ conditions for multi-parameter operators}, 
Proc. Am. Math. Soc. \textbf{144} (2016),  2437--2443.


\bibitem{GH}A. Grau de la Herr\'{a}n and T. Hyt\"{o}nen, 
\emph{Dyadic representation and boundedness of non-homogeneous 
	   Calder\'{o}n--Zygmund operators with mild kernel regularity},  
Michigan Math. J. \textbf{67} (2018), 757--786. 


\bibitem{HLLT}Y. Han, J. Li, C.-C. Lin, and C. Tan,
\emph{Singular integrals associated with Zygmund dilations},
J. Geom. Anal. \textbf{29} (2019), 2410--2455.


\bibitem{HLLTW}Y. Han, J. Li, C.-C. Lin, C. Tan, and X. Wu,
\emph{Weighted endpoint estimates for singular integral operators associated with Zygmund dilations},
Taiwanese J. Math. \textbf{23} (2019), 375--408.


\bibitem{HMT}S. Hofmann, M. Mitrea, and M. Taylor,
\emph{Singular integrals and elliptic boundary problems on regular Semmes-Kenig-Toro domains},
Int. Math. Res. Not. (2010), 2567--2865.


\bibitem{HPW}I. Holmes, S. Petermichl, and B.D. Wick, 
\emph{Weighted little $\mathrm{bmo}$ and two-weight inequalities for Journ\'{e} commutators}, 
 Anal. PDE, \textbf{11} (2018), 1693--1740. 


\bibitem{Hyt}T. Hyt\"{o}nen, 
\emph{The sharp weighted bound for general Calder\'{o}n--Zygmund operators}, 
Ann. of Math. \textbf{175} (2012), 1473--1506.


\bibitem{HL}T. Hyt\"{o}nen and S. Lappas, 
\emph{Extrapolation of compactness on weighted spaces}, 
Rev. Mat. Iberoam. \textbf{39} (2023), 91--122.


\bibitem{HLMV}T. Hyt\"{o}nen, K. Li, H. Martikainen, and E. Vuorinen,
\emph{Multiresolution analysis and Zygmund dilations},
Amer. J. Math. (2025), to appear. 


\bibitem{J85}J.-L. Journ\'{e},
\emph{Calder\'{o}n-Zygmund operators on product spaces},
Rev. Mat. Iberoam. \textbf{1} (1985), 55--91.


\bibitem{Mar}H. Martikainen,
\emph{Representation of bi-parameter singular integrals by dyadic operators},
Adv. Math. \textbf{229} (2012), 1734--1761.


\bibitem{MPTT}C. Muscalu, J. Pipher, T. Tao, and C. Thiele,
\emph{Bi-parameter paraproducts},
Acta Math. \textbf{193} (2004), 269--296.


\bibitem{NW}A. Nagel and S. Wainger, 
\emph{$L^2$ boundedness of Hilbert transforms along surfaces and 
	  convolution operators homogeneous with respect to a multiple parameter group},  
Amer. J. Math. \textbf{99} (1977), 761--785. 


\bibitem{RS}F. Ricci and E.M. Stein,
\emph{Multiparameter singular integrals and maximal functions},
Ann. Inst. Fourier (Grenoble) \textbf{42} (1992), 637--670.


\bibitem{Ou}Y. Ou,
\emph{Multi-parameter singular integral operators and representation theorem},
Rev. Mat. Iberoam. \textbf{33} (2017), 325--350.


\bibitem{PV} S. Pott and P. Villarroya,
\emph{A $T(1)$ theorem on product spaces},
https://arxiv.org/abs/1105.2516.


\bibitem{SW}E.M. Stein and G. Weiss, 
\emph{Interpolation of operators with change of measures}, 
Trans. Amer. Math. Soc. \textbf{87} (1958), 159--172. 


\bibitem{Vil}P. Villarroya,
\emph{A characterization of compactness for singular integrals},
J. Math. Pures Appl. (9) \textbf{104} (2015), 485--532.


\end{thebibliography}
\end{document}